\documentclass[a4paper, reqno, 12pt]{amsart}

\usepackage[usenames,dvipsnames]{color}
\usepackage{amsthm,amsfonts,amssymb,amsmath,amsxtra}
\usepackage[all]{xy}
\SelectTips{cm}{}
\usepackage{xr-hyper}
\usepackage[colorlinks=
   citecolor=Black,
   linkcolor=Red,
   urlcolor=Blue]{hyperref}
\usepackage{verbatim}
\usepackage{tikz,tikz-cd}
\usepackage{graphicx}
\usepackage{adjustbox}
\usetikzlibrary{positioning}
\usepackage{rotating} 

\usepackage[margin=1.25in]{geometry}
\usepackage{mathrsfs}

\RequirePackage{xspace}
\RequirePackage{etoolbox}
\RequirePackage{varwidth}
\RequirePackage{enumitem}
\RequirePackage{tensor}
\RequirePackage{mathtools}
\RequirePackage{longtable}
\RequirePackage{multirow}

\def\p{\pi}

\def\<{\langle}
\def\>{\rangle}

\newcommand{\iddots}{\reflectbox{$\ddots$}}

\newcommand{\sF}{\ensuremath{\mathscr{F}}\xspace}

\newcommand{\sU}{\ensuremath{\mathscr{U}}\xspace}

\newcommand{\fka}{\ensuremath{\mathfrak{a}}\xspace}

\newcommand{\fke}{\ensuremath{\mathfrak{e}}\xspace}

\newcommand{\fkg}{\ensuremath{\mathfrak{g}}\xspace}
\newcommand{\fkh}{\ensuremath{\mathfrak{h}}\xspace}

\newcommand{\fkl}{\ensuremath{\mathfrak{l}}\xspace}

\newcommand{\fkn}{\ensuremath{\mathfrak{n}}\xspace}

\newcommand{\fkp}{\ensuremath{\mathfrak{p}}\xspace}

\newcommand{\fku}{\ensuremath{\mathfrak{u}}\xspace}
\newcommand{\fkv}{\ensuremath{\mathfrak{v}}\xspace}
\newcommand{\fkw}{\ensuremath{\mathfrak{w}}\xspace}

\newcommand{\fkz}{\ensuremath{\mathfrak{z}}\xspace}

\newcommand{\BC}{\ensuremath{\mathbb {C}}\xspace}

\newcommand{\BF}{\ensuremath{\mathbb {F}}\xspace}
\newcommand{{\BG}}{\ensuremath{\mathbb {G}}\xspace}

\newcommand{{\BK}}{\ensuremath{\mathbb {K}}\xspace}

\newcommand{\BR}{\ensuremath{\mathbb {R}}\xspace}
\newcommand{\BS}{\ensuremath{\mathbb {S}}\xspace}

\newcommand{\BZ}{\ensuremath{\mathbb {Z}}\xspace}

\newcommand{\CA}{\ensuremath{\mathcal {A}}\xspace}

\newcommand{\CC}{\ensuremath{\mathcal {C}}\xspace}

\newcommand{\CE}{\ensuremath{\mathcal {E}}\xspace}
\newcommand{\CF}{\ensuremath{\mathcal {F}}\xspace}

\newcommand{\CI}{\ensuremath{\mathcal {I}}\xspace}
\newcommand{\CJ}{\ensuremath{\mathcal {J}}\xspace}

\newcommand{\CL}{\ensuremath{\mathcal {L}}\xspace}

\newcommand{\CO}{\ensuremath{\mathcal {O}}\xspace}

\newcommand{\CS}{\ensuremath{\mathcal {S}}\xspace}
\newcommand{\CT}{\ensuremath{\mathcal {T}}\xspace}

\newcommand{\CV}{\ensuremath{\mathcal {V}}\xspace}

\newcommand{\CY}{\ensuremath{\mathcal {Y}}\xspace}

\newcommand{\ad}{{\mathrm{ad}}}

\newcommand{\GL}{\mathrm{GL}}

\newcommand{\WF}{\mathrm{WF}}

\DeclareMathOperator{\Hom}{Hom}

\newcommand{\id}{\ensuremath{\mathrm{id}}\xspace}

\newcommand{\Ind}{{\mathrm{Ind}}}
\newcommand{\CJF}{\widehat{\mathcal{J}}_{\mathfrak{u}}}

\DeclareMathOperator{\Ker}{Ker}

\DeclareMathOperator{\Lie}{Lie}

\newcommand{\SO}{{\mathrm{SO}}}

\DeclareMathOperator{\Sym}{Sym}

\DeclareMathOperator{\tr}{tr}

\newcommand{\U}{\mathrm{U}}

\newcommand{\dslash}{/\!\!/}

\newcommand{\wt}{\mathrm{wt}}

\newcommand{\ov}{\overline}

\newcommand{\lra}{\longrightarrow}

\newcommand{\rmh}{\mathrm{H}}
\newcommand{\Tot}{\mathrm{Tot}}
\newcommand{\Wh}{\mathrm{Wh}}
\newcommand{\mind}{\bar{\times}}

\newcommand{\Span}{{\operatorname{Span}}}

\newcommand{\Supp}{\operatorname{Supp}}

\newcommand{\SInd}{\mathcal{S}\mathrm{Ind}}
\newcommand{\Smod}{\mathcal{S}\mathrm{mod}}
\newcommand{\Ext}{\mathrm{Ext}}
\newcommand{\EP}{\mathrm{EP}}

\DeclareMathOperator{\supp}{supp}

\newtheorem{theorem}{Theorem}
\newtheorem{proposition}[theorem]{Proposition}
\newtheorem{lemma}[theorem]{Lemma}
\newtheorem {conjecture}[theorem]{Conjecture}
\newtheorem{corollary}[theorem]{Corollary}

\theoremstyle{definition}
\newtheorem{definition}[theorem]{Definition}
\newtheorem{example}[theorem]{Example}

\newtheorem{remark}[theorem]{Remark}

\numberwithin{equation}{section}
\numberwithin{theorem}{section}

\setitemize[0]{leftmargin=*,itemsep=\the\smallskipamount}
\setenumerate[0]{leftmargin=*,itemsep=\the\smallskipamount}

\renewcommand{\to}{%
   \ifbool{@display}{\longrightarrow}{\rightarrow}%
   }
\let\shortmapsto\mapsto
\renewcommand{\mapsto}{%
   \ifbool{@display}{\longmapsto}{\shortmapsto}%
   }
\newlength{\olen}
\newlength{\ulen}
\newlength{\xlen}
\newcommand{\xra}[2][]{%
   \ifbool{@display}%
      {\settowidth{\olen}{$\overset{#2}{\longrightarrow}$}%
       \settowidth{\ulen}{$\underset{#1}{\longrightarrow}$}%
       \settowidth{\xlen}{$\xrightarrow[#1]{#2}$}%
       \ifdimgreater{\olen}{\xlen}%
          {\underset{#1}{\overset{#2}{\longrightarrow}}}%
          {\ifdimgreater{\ulen}{\xlen}%
             {\underset{#1}{\overset{#2}{\longrightarrow}}}
             {\xrightarrow[#1]{#2}}}}%
      {\xrightarrow[#1]{#2}}
   }
\makeatother
\newcommand{\xyra}[2][]{%
   \settowidth{\xlen}{$\xrightarrow[#1]{#2}$}%
   \ifbool{@display}%
      {\settowidth{\olen}{$\overset{#2}{\longrightarrow}$}%
       \settowidth{\ulen}{$\underset{#1}{\longrightarrow}$}%
       \ifdimgreater{\olen}{\xlen}%
          {\mathrel{\xymatrix@M=.12ex@C=3.2ex{\ar[r]^-{#2}_-{#1} &}}}%
          {\ifdimgreater{\ulen}{\xlen}%
             {\mathrel{\xymatrix@M=.12ex@C=3.2ex{\ar[r]^-{#2}_-{#1} &}}}
             {\mathrel{\xymatrix@M=.12ex@C=\the\xlen{\ar[r]^-{#2}_-{#1} &}}}}}%
      {\mathrel{\xymatrix@M=.12ex@C=\the\xlen{\ar[r]^-{#2}_-{#1} &}}}%
   }
\makeatletter
\newcommand{\xla}[2][]{%
   \ifbool{@display}%
      {\settowidth{\olen}{$\overset{#2}{\longleftarrow}$}%
       \settowidth{\ulen}{$\underset{#1}{\longleftarrow}$}%
       \settowidth{\xlen}{$\xleftarrow[#1]{#2}$}%
       \ifdimgreater{\olen}{\xlen}%
          {\underset{#1}{\overset{#2}{\longleftarrow}}}%
          {\ifdimgreater{\ulen}{\xlen}%
             {\underset{#1}{\overset{#2}{\longleftarrow}}}
             {\xleftarrow[#1]{#2}}}}%
      {\xleftarrow[#1]{#2}}
   }
\newcommand{\isoarrow}{%
   \ifbool{@display}{\overset{\sim}{\longrightarrow}}{\xrightarrow\sim}%
   }
   
\begin{document}

\title[Bernstein-Zelevinsky Theory]{Archimedean Bernstein-Zelevinsky Theory and Homological Branching Laws}
\author[ Wu \& Zhang]{Kaidi Wu and Hongfeng Zhang}

\address{(Wu) Department of Mathematics and New Cornerstone Science Laboratory, The University of Hong Kong, HK.}
\email{kaidiwu24@connect.hku.hk}

\address{(Zhang) Department of Mathematics, The University of Hong Kong, HK.}
\email{zhanghongf@pku.edu.cn}

\thanks{}


\date{\today}

\begin{abstract}
We develop the Bernstein-Zelevinsky theory for quasi-split real classical groups and employ this framework to establish an Euler-Poincar\'e characteristic formula for general linear groups. The key to our approach is establishing the Casselman-Wallach property for the homology of the Jacquet functor, which also provides an affirmative resolution to an open question in~\cite[3.1.(1)]{AGS15a}. Furthermore, we prove the vanishing of higher extension groups for arbitrary pairs of generic representations, confirming a conjecture of Dipendra Prasad.

We also utilize the Bernstein-Zelevinsky theory to establish two additional results: the Leibniz law for the highest derivative and a unitarity criterion for general linear groups.

Lastly, we apply the Bernstein-Zelevinsky theory to prove the Hausdorffness and exactness of the twisted homology of split even orthogonal groups.
\end{abstract}

\maketitle

\tableofcontents

\section{Introduction}
This is the first article in a series developing Bernstein-Zelevinsky theory for real classical groups. The Bernstein–Zelevinsky filtration, which characterizes the restriction of smooth representations to the mirabolic subgroup, provides a foundational theory for general linear groups over $p$-adic fields, with numerous applications to the local Langlands correspondence and branching laws. Compared to the non-Archimedean case, the Archimedean setting presents two intrinsic challenges:
\begin{itemize}
    \item \textit{Analytic difficulty}: There is no suitable analogue of $\ell$-sheaves in the Archimedean case. The behavior of Schwartz functions along closed Nash submanifolds is subtle, although governed by normal derivatives (via Borel's lemma).
    \item \textit{Topological difficulty}: Unlike the $p$-adic case, representations are Fr\'echet spaces. Consequently, establishing the Hausdorff property for (twisted) Jacquet modules is non-trivial. Furthermore, the complexity of Fr\'echet topologies precludes a classification of irreducible smooth representations for non-reductive groups.
\end{itemize}

To tackle the analytic difficulty, we utilize Fourier transforms. The key insight is that while group actions are transitive on the original domain, the dual domain  may decompose into many orbits under the actions after Fourier transforms. This allows us to apply Borel's lemma to achieve an irreducible quotient filtration of representations in the dual domain, yielding a spectral expansion along characters of the unipotent radical in the mirabolic subgroup. For applications to homological branching laws, we provide an axiomatic definition of Archimedean Bernstein-Zelevinsky filtration (Definition~\ref{def-BZ}). Our first main result establishes:

\begin{theorem}
    Let $\pi$ be a Casselman-Wallach representation of $\GL_n(\mathbf{k})$ where $\mathbf{k}=\BR$ or $\BC$. The restriction of $\pi$ to the mirabolic subgroup admits a Bernstein-Zelevinsky filtration.
\end{theorem}
We note that despite additional requirements in our filtration definition, it remains less canonical than its $p$-adic counterpart. In the proof of the theorem, we establish the Bernstein-Zelevinsky filtration for parabolically induced representations. This will imply the following Leibniz law for highest derivatives (for the definition of highest derivatives, see section~\ref{derivative-sec}).
\begin{theorem}
     Let $\pi_i$ be Casselman-Wallach representations of $\GL_{n_i}$ for $1\leq i\leq k$, where $\sum_{i=1}^k n_i = n$. Then
\[
\operatorname{s.s.}\ (\pi_1 \times \dots \times \pi_k)^-  \simeq \operatorname{s.s.}\big( \pi_1^- \times \dots \times \pi_k^- \big).
\]
Here, $\pi_1 \times \dots \times \pi_k$ denotes the normalized parabolic induction of $\GL_n$, and ``$\operatorname{s.s.}$'' stands for the semi-simplification of representations of finite length.

\end{theorem}

In fact, the topological difficulty is one of the motivations to investigate such spectral expansion. In the homological branching law, the Hausdorffness of various derivatives is essential. Here, a derivative is a kind of reduction at a specific character of the unipotent radical (for the definition of derivatives, see section~\ref{derivative-sec}). Our second main result affirmatively resolves an open question posed in \cite[3.1.(1)]{AGS15a}:

\begin{theorem}
Let $\pi$ be a Casselman-Wallach representation of $\GL_n(\mathbf{k})$ where $\mathbf{k}=\BR$ or $\BC$. Then $L^iB^k(\pi)$ is a Casselman-Wallach representation of $\GL_{n-k}(\mathbf{k})$ for all integers $0\leq k\leq n$ and all $i$. In particular, $L^iB^k(\pi)$ is Hausdorff.
\end{theorem}

In the proof, we demonstrate a stronger result.
\begin{theorem}\label{intro-CW}
    Let $\pi$ be a Casselman-Wallach representation of $\GL_n(\mathbf{k})$ where $\mathbf{k}=\BR$ or $\BC$. Let $P$ be a parabolic subgroup of $\GL_n(\mathbf{k})$ with Levi decomposition $P=LU$. Then $\rmh_i(\fku,\pi)$ is a Casselman-Wallach representation of $L$ for every integer $i$.
\end{theorem}
This result also gives convincing evidence to the Casselman's homological comparison conjecture, which is important for automorphic representation theory, see \cite{LLY21} and \cite[Conjecture 10.3]{Vog08} for details. In the proof of Theorem~\ref{intro-CW}, we establish a coarse spectral filtration of the restriction  $\pi|_P$ based on the Bernstein-Zelevinsky filtration (for the description of coarse spectral filtration, see Proposition~\ref{rest_max}). Moreover, we find a suitable category such that $\rmh_i(\fku,\tau)$ is Casselman-Wallach for every object $\tau$ in this category, and observe that through the Casselman-Jacquet functor, the trivial extension spectrum of $\pi|_P$ can be synthesized as an object in this category. Here, the trivial extension spectrum refers to the irreducible subquotient in the filtration of $\pi|_{P}$ which is isomorphic to a trivial extension from an irreducible representation of $L$.

Another motivation arises in (homological) branching laws and relative Langlands programs. Initiated by restricting orthogonal group representations, the Gan-Gross-Prasad conjecture has become fundamental in relative Langlands programs (see \cite{GGP}). In his ICM proceedings \cite{Pr18}, Dipendra Prasad proposed an alternative approach, observing that the Euler-Poincar\'e characteristic
\begin{equation*}
    \EP(\pi,\tau) := \sum_{i\in\BZ}(-1)^i \dim \text{``}\Ext^i_{\GL_n}(\pi,\tau)\text{''}, \quad 
    \pi \in \mathrm{Rep}(\GL_{n+1}(\BF)), \tau \in \mathrm{Rep}(\GL_{n}(\BF))
\end{equation*}
is a more natural invariant than multiplicity for local fields $\BF$ of characteristic 0. This characteristic should be computationally accessible, and vanishing of higher extension groups would recover multiplicity data. This approach has proven fruitful for $p$-adic groups (e.g., \cite{Chan21,CSa21}).

For real reductive groups $G$, the theory encounters obstacles for two reasons:
\begin{itemize}
    \item The primary category $\Smod_G$ consists of smooth, moderate-growth Fr\'echet representations. This non-abelian category lacks sufficient injective objects.
    \item We do not have the right adjoint functor for Schwartz inductions in category $\Smod_G$. 
\end{itemize}
Consequently, we define the Euler-Poincar\'e characteristic as
\begin{equation}
    \EP(\pi,\tau) := \sum_{i\in\BZ}(-1)^i \dim \Ext^i_{\GL_n}(\pi \widehat{\otimes} \tau^{\vee}, \BC),
\end{equation}
where $\pi$ is a Casselman-Wallach representation of $\GL_{n+1}$ and $\tau^{\vee}$ is the contragredient of a Casselman-Wallach representation $\tau$ of $\GL_n$. And we define the extension group through the strong relative projective resolution developed in~\cite{CS}.

Before defining this characteristic, one must establish finite-dimensionality and vanishing of extension groups in high degrees (homological finiteness). For $p$-adic spherical pairs satisfying finite multiplicity, this follows from local finiteness \cite{AS}. While unavailable generally in the Archimedean case, homological finiteness for GGP pairs follows from the Bernstein-Zelevinsky filtration. Our third main result is:

\begin{theorem}
    Let $\pi$ and $\tau$ be Casselman-Wallach representations of $\GL_{n+1}(\mathbf{k})$ and $\GL_n(\mathbf{k})$ respectively, where $\mathbf{k}=\BR$ or $\BC$. Then $\pi$ satisfies the homological finiteness for $\tau$, and 
    \[
    \EP_{\GL_n}(\pi,\tau) = \Wh(\pi) \cdot \Wh(\tau).
    \]
    Here, $\Wh(\cdot)$ denotes Whittaker model multiplicity. 
\end{theorem}
For higher extension groups, Rankin-Selberg theory developed by Jacquet, Piatetski-Shapiro, and Shalika shows that for generic $\pi,\tau$,
\[
\Hom_{\GL_n}(\pi,\tau) = \Wh(\pi) \cdot \Wh(\tau).
\]
Based on this, Dipendra Prasad conjectured the vanishing of higher extensions for irreducible generic representations. Our fourth main result confirms this conjecture.
\begin{theorem}
     Let $\pi$ and $\tau$ be irreducible generic representations of $\GL_{n+1}(\mathbf{k})$ and $\GL_n(\mathbf{k})$ respectively, where $\mathbf{k}=\BR$ or $\BC$. Then
    \[
    \Ext^i_{\GL_n}(\pi\widehat{\otimes}\tau^{\vee},\BC)=0 \text{ for every integer }i>0.
    \]
\end{theorem}
Our proof essentially uses the opposite Bernstein-Zelevinsky filtration, which is a filtration of the opposite mirabolic subgroup. The existence of such a filtration can be deduced from the Bernstein-Zelevinsky filtration of the contragredient representation. The following diagram summarizes our framework.
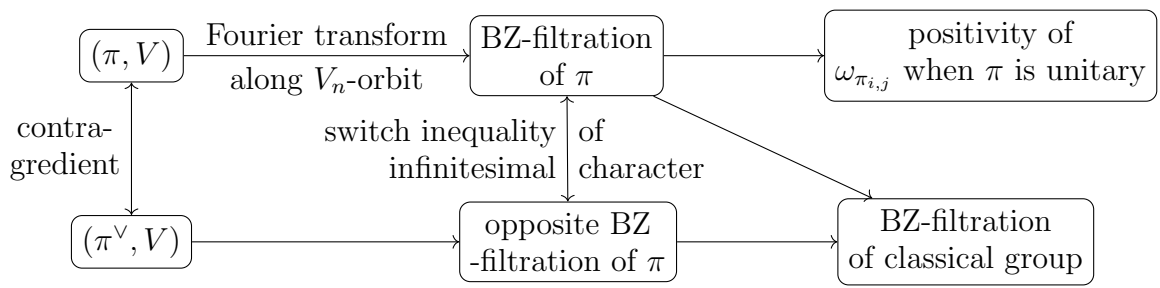
\begin{center}
   \begin{tikzpicture}[node distance=10pt]
  \node[draw, rounded corners]            (U1)   {$(\pi,V)$};
  \node[draw, align=center, rounded corners, right=107pt of U1]    
   (U2)  {BZ-filtration\\ of $\pi$};
 \node[draw, align=center, rounded corners, right=60pt of U2]         (U3)  {positivity of \\ $\omega_{\pi_{i,j}}$ when $\pi$ is unitary};
\node[draw, rounded corners, below=50pt of U1]            (D1)   {$(\pi^{\vee},V)$};
\node[draw, align=center, rounded corners, right=100pt of D1]    
   (D2)  {opposite BZ\\-filtration of $\pi$};
\node[draw, align=center, rounded corners, right=60pt of D2]         (D3)  {BZ-filtration \\ of classical group};
  \draw[->] (U1)  -- node[above]{Fourier transform}node[below]{along $V_n$-orbit}(U2);
  \draw[->] (U2) -- (U3);
\draw[<->] (U1)  --node[left, align=center]{contra-\\gredient} (D1);
\draw[->] (D1)  -- (D2);
\draw[<->] (U2)  --node[left, align=right]{switch inequality \\ infinitesimal }node[right, align=left]{of \\ character } (D2);
\draw[->] (U2)  -- (D3);
\draw[->] (D2)  -- (D3);
\end{tikzpicture} 
\end{center}
Here, $(\pi,V)$ is a Casselman-Wallach representation of $\GL_n$, and $\pi_{i,j}$ is the irreducible representation appearing in the Bernstein-Zelevinsky filtration of $\pi$. The central character of $\pi_{i,j}$ is denoted by $\omega_{\pi_{i,j}}$.  The contragredient representation $\pi^{\vee}$ is realized on the same space $V$ by $\pi^{\vee}(g):=\pi(g^{-t})$. 

In addition to the opposite Bernstein-Zelevinsky filtration, we employ two technical methods. The first, called \textit{substitution}, identifies an irreducible generic representation $\widetilde{\pi}$ that is compatible with the inductive argument and satisfies the isomorphism
\[
 \Ext^i_{\GL_n}(\pi\widehat{\otimes}\tau^{\vee},\BC)= \Ext^i_{\GL_n}(\widetilde{\pi}\widehat{\otimes}\tau^{\vee},\BC)\text{ for every integer }i.
\]
The key step is to construct a sequence quasi-isomorphic to the long exact sequence associated to the open-closed orbits of $\pi$.
 The second, \textit{switching}, exchanges the positions of $\pi$ and $\tau$. Substitution was inspired by \cite{CSa21}, while switching originated from \cite{CS15}. We emphasize that the combinatorics in the Archimedean case is significantly more complicated than the $p$-adic case due to the absence of Zelevinsky classifications and obstacles from normal derivatives.

As indicated in the diagram above, our fifth main result provides a necessary condition for unitarity in irreducible $\GL_n$-representations, generalizing the $p$-adic unitary criterion of \cite[section 7.3]{Ber84} to the Archimedean setting.
\begin{theorem}\label{uni-crit}
    Let $\pi$ be an irreducible unitary representation of $\GL_n(\mathbf{k})$ of depth $d$, where $\mathbf{k}=\BR$ or $\BC$. For every irreducible subquotient $I^{k-1}E(\tau)$ in the Bernstein-Zelevinsky filtration of $\pi|_{M_n}$ satisfying $k \neq d$ (where $\tau$ denotes an irreducible representation of $\GL_{n-k}$), we have
    \[
    \mathrm{Re} \, \omega_{\tau} >0 .
    \]
\end{theorem}
Here $I$ and $E$ denote the Mackey induction and trivial extension, respectively, see Section~\ref{induction section} for details.

The last part of this article is devoted to the Bernstein-Zelevinsky filtration of the isometry group of split $\epsilon$-Hermitian space. The Bernstein-Zelevinsky filtration we pursue constitutes a smooth spectral expansion over the coadjoint orbits in $E_n^*$ of the maximal parabolic subgroup $R_{n-1,1}$(see Section~\ref{derivative-sec} and Section~\ref{BZ-cl-sec} for precise definitions). For orthogonal groups, $E_n$ is abelian, so its irreducible representations are characters. For unitary and symplectic groups, there exist Weil representations of $E_n$, which will contribute to the Fourier-Jacobi model. 

In this article, we establish the Bernstein-Zelevinsky filtration for orthogonal groups based on the Bernstein-Zelevinsky filtration of $\GL_n$ and prove that, similar to $\GL_n$, the twisted homology of $E_n$ is Hausdorff and its higher homology vanishes (see Theorem~\ref{tw-homo-thm} for details). This result refines the exactness and finite-dimensionality properties of the Whittaker model. It is also useful for the further study of the Euler-Poincar\'e characteristic formula.

Finally, we note that the idea of the smooth spectral expansion can be generalized to any maximal parabolic subgroup of a classical group. By~\cite[Theorem 1.1]{HG99}, the maximal parabolic subgroups of classical groups have finitely many co-adjoint orbits on their unipotent radicals. In the specific case where the unipotent radical is abelian, we propose Conjecture~\ref{coa-conj}. As previously mentioned, the filtration in Conjecture~\ref{coa-conj} is also non-canonical. In particular, it is not unique due to the flexibility in the choice of order, nor is it functorial. However, our philosophical stance is that the coarse spectral filtration is sufficiently strong to allow us to derive a canonical filtration and obtain the comparison conjecture. This program will be carried out in a forthcoming article~\cite{CWYZ}.

\subsection{Convention and notation}\label{notation_sec}
In this subsection, we introduce some notations that will be used throughout this article.
\subsubsection{General groups}
\begin{itemize}
    \item $G$, $H$, etc. (capital English letters): various real Lie groups (almost linear Nash groups or real reductive groups).
    \item $\mathfrak{g} := \mathrm{Lie}(G)_{\mathbb{C}}$ (Gothic letters): the complexified Lie algebras.
    \item $\delta_H$: the modular character of Lie group $H$.
    \item $\mathcal{U}(\mathfrak{g})$: the universal enveloping algebra of $\mathfrak{g}$;  $\mathcal{U}(\mathfrak{g})^{<k}$ the subspace of elements with degree $<k$; $\mathcal{Z}(\mathfrak{g})$ the center of $\mathcal{U}(\mathfrak{g})$.
\end{itemize}

\subsubsection{Real reductive groups}
For a real reductive group $G$:
\begin{itemize}
    \item We fix a Cartan involution $\theta$ and a $\theta$-stable maximally split Cartan subgroup $A$ (from now on, the Cartan involution will no longer be involved and
 $\theta$ is free for other notation).
    \item $P^0 = L^0 U^0$: a minimal parabolic subgroup with Levi decomposition such that $L^0$ contains $A$.
    \item $P = LU$: standard parabolic subgroup $P \supset P^0$ with Levi decomposition.
    \item $\overline{P}$: the opposite parabolic subgroup of $P$.
    \item $K$ (resp.\ $K_L$): the complexification of the maximal compact subgroup of $G$ (resp.\ $L$) fixed by the Cartan involution.
    \item $\mathfrak{b} \subset \mathfrak{p}^0$: a Borel subalgebra with $\mathfrak{a} \subset \mathfrak{b}$.
    \item $\Delta(\mathfrak{a}, \mathfrak{g})$: the roots of $\mathfrak{a}$ in $\mathfrak{g}$ with positive roots corresponding to $\mathfrak{b}$.
    \item $\rho$: the half-sum of positive roots.
    \item $\rho_{\mathfrak{l}}$: half-sum of positive roots in $\Delta(\mathfrak{a}, \mathfrak{l})$, where $L \subset P$ is a standard Levi subgroup.
    \item $W$ (resp.\ $W_L$): the Weyl group of $G$ (resp.\ $L$). For general linear groups, $W$ is represented by permutation matrices.
    \item $\chi_{\lambda}$: the infinitesimal character (algebra homomorphism $\mathcal{Z}(\mathfrak{g}) \to \mathbb{C}$) corresponding to $\lambda \in \mathfrak{a}^*$ via the Harish-Chandra isomorphism, normalized so that $-\rho$ corresponds to the trivial representation.
\end{itemize}

\subsubsection{General linear groups}
Let $\mathrm{GL}_n = \mathrm{GL}_n(\mathbf{k})$ where $\mathbf{k} = \mathbb{R}$ or $\mathbb{C}$:
\begin{itemize}
    \item Cartan involution: transpose inverse; Cartan subgroup $A$: diagonal matrices.
    \item $B_n$: Borel subgroup of upper triangular matrices; $N_n$: unipotent radical of $B_n$.
    \item $M_n$: mirabolic subgroup (matrices with last row $(0,\dots,0,1)$).
    \item $V_n$: subgroup of $M_n$ of matrices $\begin{pmatrix} I_{n-1} & v \\ & 1 \end{pmatrix}$.
    \item $H_{n,d}$: subgroup of $M_n$ of matrices $\begin{pmatrix} a & x \\ 0 & u \end{pmatrix}$ with $a \in \mathrm{GL}_{n-d}$, $u \in N_d$, and $x$ an $(n-d) \times d$ matrix.
    \item $P_{k,n-k}$: standard parabolic subgroup with Levi factor $\mathrm{GL}_k \times \mathrm{GL}_{n-k}$.
    \item $U_{k,n-k}$: unipotent radical of $P_{k,n-k}$.
\end{itemize}
For a subgroup $H \subseteq \mathrm{GL}_n$, $\overline{H}$ denotes the transpose of $H$. Fixed characters:
\begin{itemize}
    \item $\psi$: a fixed unitary character of $\mathbf{k}$.
    \item $\psi_n$: character of $V_n$ defined by $\psi_n\left(\begin{bmatrix} I_{n-1} & v \\ & 1 \end{bmatrix}\right) := \psi(x_{n-1})$ for $v = [x_1,\dots,x_{n-1}]^t \in \mathbf{k}^{n-1}$; which also denotes the corresponding character of the Lie algebra $\mathfrak{v}_n$.
    \item $\psi_{n,d}$: character of $H_{n,d}$ defined by 
        \[
        \psi_{n,d}\left(\begin{bmatrix} a & x \\ 0 & u \end{bmatrix}\right) := \psi\left(\sum_{i=1}^{d-1} u_{i,i+1}\right), \quad u = (u_{ij})_{1 \leq i,j \leq d}.
        \]
\end{itemize}
 
Characters of $\mathbf{k}^\times$ have the form
\[
\chi_{\epsilon,s}(x) = 
\begin{cases}
\left(\frac{x}{|x|}\right)^\epsilon |x|^s, & \epsilon = 0,1,\ s \in \mathbb{C},\ \mathbf{k} = \mathbb{R}, \\
\left(\frac{x}{|x|}\right)^\epsilon |x|^{2s}, & \epsilon \in \mathbb{Z},\ s \in \mathbb{C},\ \mathbf{k} = \mathbb{C}.
\end{cases}
\]
The \textbf{real part} of $\chi = \chi_{\epsilon,s}$ is $\mathrm{Re}\,\chi := \mathrm{Re}\,s$. We also regard
$\chi_{\epsilon,s}$ as a character of some general linear group by composing determinant.

For a character $\xi$ of $H$ and automorphism $\theta$ of $H$, let ${}^\theta\xi := \xi \circ \theta$.

\subsubsection{Representation-theoretic conventions}

Let $G$ be an almost linear Nash group.  

\begin{itemize}
    \item For a Nash action of $G$ on a Nash manifold $X$, let $G^x$ denote the stabilizer of $x \in X$ in $G$.
    \item If there is no other clarification, representations of $G$ mean \textbf{Fréchet representations which are of moderate growth and smooth under the $G$-action}. This category is denoted by $\Smod_G$.
    \item $V'$: strong dual of a locally convex topological vector space $V$.
    \item Maps between locally convex topological vector spaces are continuous.
    \item For irreducible representation $\pi$, $\omega_\pi$ denotes the central character of $Z_G$.
    \item For a vector space $V$ over $\mathbf{k}$, $V^* := \mathrm{Hom}_{\mathbf{k}}(V, \mathbf{k})$ denotes the algebraic dual.
    \item $\widehat{G}$: equivalence classes of irreducible unitarizable representations in $\Smod_G$.
    \item When $G \cong \mathbb{R}^n$, we identify $\widehat{G}$ with $\sqrt{-1}\,\mathrm{Lie}(G)^* \simeq \mathrm{Lie}(G)$.
\end{itemize}

If there is no other clarification, representations of a Lie algebra $\mathfrak{g}$ are \textbf{Fréchet representations continuous under the $\mathfrak{g}$-action}, and \textbf{subrepresentations} are closed subspaces.

\bigskip
\centerline{\scshape Acknowledgements}
We sincerely thank Binyong Sun and Dipendra Prasad for their inspiring conversations. We also appreciate Hao Ying for introducing us to the topological spectral sequence. Special thanks are extended to Kei Yuen Chan for his valuable insights into the Casselman-Jacquet functor. In particular,  Proposition~\ref{surj_CJ}  is due to him. We would like to thank Zhibin Geng for reading the manuscript and providing many helpful comments. We are grateful to Jun Yu for several fruitful conversations. The main part of this article was written during our visits to the Institute of Advanced Study in Mathematics of Zhejiang University, SIMIS, the Chern Institute of Mathematics, the Institute for Mathematical Science, and the Tianyuan Mathematics Research Center. We are grateful for their warm hospitality. Wu is partially supported by the New Cornerstone Science Foundation through the New Cornerstone Investigator Program awarded to Professor Xuhua He. Wu is also supported by the National Natural Science Foundation of China (Grant No. 123B1004). He also thanks Xuhua He for his continued support and encouragement. Zhang is supported by the Research Grants Council of the Hong Kong Special Administrative Region, China (Project No: 17305223) and the NSFC grant for Excellent Young Scholar (Project No: 12322120). Zhang thanks the kind support from Kei Yuen Chan.

\section{Preliminary}
\subsection{Casselman-Wallach representations}
For \textbf{Casselman-Wallach representations} of a reductive group $G$, we mean smooth moderate growth Fr\'echet representations of finite length. They appear as Archimedean components of automorphic representations. Readers may consult \cite[Chapter 11]{Wal92} for details about Casselman-Wallach representations. Harish-Chandra modules, on the other hand, offer algebraic advantages. For \textbf{Harish-Chandra modules}, we mean the $(\fkg,K)$-module which is admissible and finitely generated over $\U(\fkg)$. Casselman and Wallach constructed a canonical globalization for each Harish-Chandra module as the smooth vectors of every Banach globalization. Using such a globalization, they proved the following result.
\begin{theorem}[see \cite{Wal92}, 11.6.8]
    The functor taking $K$-finite vectors defines an equivalence between the category of Casselman-Wallach representations and the category of Harish-Chandra modules.
\end{theorem}

We denote the Harish-Chandra module consisting of $K$-finite vectors of the Casselman-Wallach representation $\pi$ by $\pi^K$. The theorem has a direct corollary.
\begin{corollary}[see \cite{AGS15a}, Corollary 2.2.5(2)]\label{closed image}
    Any morphism between Casselman-Wallach representations has a closed image.
\end{corollary}
We remark that a more general result can be deduced from Corollary~\ref{closed image}.
\begin{remark}\label{closed-im-rem}
    Let $\varphi: L \to M$ be a continuous map from $L \in \Smod_G$ to a Casselman–Wallach representation $M$ of $G$. Then $\varphi$ induces a $G$-equivariant, injective, continuous map $L/\Ker\varphi \to M$, which implies that $L/\Ker\varphi$ is admissible and admits finite generalized infinitesimal characters. Consequently, $L/\Ker\varphi$ is a Casselman–Wallach representation, and $\mathrm{Im}(\varphi)$ is closed by Corollary~\ref{closed image}.
\end{remark}

We recall some basic facts about the parabolic production of $(\fkg,K)$-module. For $(\fkg,K)$-module, we always assume the $K$-action is locally finite. Let $P=LU$ be a parabolic subgroup of $G$ with Levi decomposition. Let $\beta$ be a $(\fkl,K_L)$-module, which is also viewed as a $(\fkp,K_L)$-module by trivial extension on $\fku$. Then we can define two functors from the category of $(\fkl,K_L)$-modules to the category of $(\fkg,K)$-modules:
\begin{itemize}
    \item Parabolic production functor $P_{\fkp,K_L}^{\fkg,K}$:
    \[
    \beta \mapsto R(\fkg,K)\otimes_{R(\fkp,K_L)}\beta;
    \]
    \item Parabolic induction functor $I_{\fkp,K_L}^{\fkg,K}$:
    \[
    \beta \mapsto \Hom_{R(\fkp,K_L)}(R(\fkg,K),\beta)^{K\text{-finite}}.
    \]
\end{itemize}
Here ``$R$" indicates the Hecke algebra of a Lie pair (see \cite[Chapter I, Section 5]{KV}). The parabolic production has the following two properties which we will use. The first property is \textbf{Mackey isomorphism}.
\begin{lemma}[see \cite{KV}, Theorem 2.103]
    Let $\beta$ be a $(\fkl,K_L)$-module and $\pi$ be a $(\fkg,K)$-module, then there is a natural isomorphism as $(\fkg,K)$-modules:
    \[
    \pi\otimes P_{\fkp,K_L}^{\fkg,K}(\beta)\simeq P_{\fkp,K_L}^{\fkg,K}(\pi|_{\fkp,K_L}\otimes \beta).
    \]
\end{lemma}
The second property is \textbf{Shapiro's lemma}. Although the proof is well-known, we still include it for the reader's convenience.
\begin{lemma}
    Let $\beta$ be a $(\fkl,K_L)$-module and $\pi$ be a finite-dimensional $(\fkg,K)$-module, then there is a natural isomorphism for every integer $i$
    \[
    \Ext^i_{\fkg,K}( P_{\fkp,K_L}^{\fkg,K}(\beta),\pi)\simeq \Ext^i_{\fkp,K_L}(\beta,\pi|_{\fkp,K_L})
    \]
\end{lemma}
\begin{proof}
    Let $C_{\bullet}$ be a projective resolution of $\beta$ in the category of $(\fkp,K_L)$-modules. Then by \cite[Proposition 11.2]{KV} and \cite[Corollary 2.35]{KV}, $P_{\fkp,K_L}^{\fkg,K}(C_{\bullet})$ is a projective resolution of $P_{\fkp,K_L}^{\fkg,K}(\beta)$. Hence the result follows from the usual Shapiro lemma
    \[
    \Hom_{\fkg,K}( P_{\fkp,K_L}^{\fkg,K}(C_{\bullet}),\pi)\simeq \Hom_{\fkp,K_L}(C_{\bullet},\pi|_{\fkp,K_L})
    \]
    by \cite[Proposition 2.33, 2.34]{KV}.
\end{proof}

\subsection{Derivative for quasi-split classical groups}\label{derivative-sec}
We first introduce various derivatives for representations in the $\rm{GL}_n$ case. 

Define the absolute value for Archimedean local field as $|x|_{\mathbb{R}}=|x|$ for $x\in \mathbb{R}$, while $|x|_{\mathbb{C}}=|x|^2$ for $x\in \mathbb{C}$.
\begin{definition}\label{def-der-gl}
    Let $\sigma$ be a smooth moderate growth Fr\'echet representation of $M_n$, we define
    $$\Psi(\sigma):=|\det|_{\mathbf{k}}^{-1/2} \otimes \sigma/\Span\{\alpha v -\psi_n(\alpha)v \mid v\in\sigma, \alpha\in \fkv_n\}$$
and
$$\Phi(\sigma):= \varprojlim_{l} \sigma /\Span\{\kappa v\mid v  \in \sigma ,\kappa \in ({\mathfrak{v}}_{n})^{\otimes l} \}, \quad \Phi_0(\sigma):= \sigma /\Span\{\kappa v\,\mid v  \in \sigma ,\kappa \in \fkv_n\}.$$
Here, $\Psi(\sigma)$ is a representation of $M_{n-1}$, $\Phi(\sigma)$ and $\Phi_0(\sigma)$ are representations of $\GL_{n-1}$.   For convenience, we also introduce the following notations.

\begin{itemize}
\item Define $\Psi_0(\sigma):=\Psi(\sigma)\cdot |\det|_{\mathbf{k}}^{1/2} $.
    \item The $k$-th derivative of $\sigma$ is defined to be $D^k(\sigma):=\Phi\Psi^{k-1}(\sigma)$. The \textbf{depth} of representation $\sigma$ is the maximal positive integer $k_0$ such that $D^{k_0}(\sigma)\neq 0$, and $D^{k_0}(\sigma)$ is called the \textbf{highest derivative} of $\sigma$, denoted by $\sigma^{-}$.
    \item For $k\neq 0$, define $B^k(\sigma):=\Phi_0\Psi^{k-1}(\sigma)$. It is a representation of $\GL_{n-k}$. The following are some variants of $B^k$ that appear in the context.
    
    Let $B^k_0(\sigma):=\Phi_0\Psi_0^{k-1}$, and let $B^k_{-}(\sigma):=B^k(\sigma)\cdot |\det|_{\mathbf{k}}^{-1/2} $. For $k=0$, set $B^k(\sigma)=B^k_0(\sigma)=B_{-}^k(\sigma)=\sigma$.
\end{itemize}
\end{definition}
\begin{remark}\label{B^k rem}
The functor $B^k_0$ has an alternative interpretation that is crucial in our proof of its Casselman-Wallach property. We note that
\[
B^k_0(\sigma) = \Psi_0^{k-1}(\rmh_0(\mathfrak{u}_{n-k,k}, \sigma)),
\]
where $\rmh_0(\mathfrak{u}_{n-k,k}, \sigma)$ is a $\mathrm{GL}_{n-k} \times \mathrm{GL}_k$-representation, and $\Psi_0^{k-1}$ is taken with respect to the $\mathrm{GL}_k$-representation.
\end{remark}
A priori, the representations $B^k_0(\sigma)$ are \textbf{possibly non-Hausdorff}. But we will show that these representations are Hausdorff when $\sigma$ is the restriction of some Casselman-Wallach representation of $\GL_n$.

In order to introduce derivatives for other classical groups, we fix the following notations. Let $\mathbf{k}/\mathbf{k}'$ be an archimedean local field extension with $[\mathbf{k}:\mathbf{k}']\leq 2$. Let $(V,(\cdot,\cdot))$ be a $\epsilon$-Hermitian space over $\mathbf{k}$, where $\epsilon=1$ or $-1$. That is, 
\[
(\cdot,\cdot):V\times V\lra \mathbf{k}
\]
is a bilinear form over $\mathbf{k}'$, which is  $\mathbf{k}$-linear over the first variable and satisfies
\[
(x,y)=\epsilon(y,x)^c, \text{ for }x,y\in V,
\]
where ${\bullet}^c$ is complex conjugation when $\mathbf{k}'=\BC$ and identity when $\mathbf{k}'=\BR$. When $V$ is quasi-split, it is determined up to isomorphism by its dimension. From now on, we assume $V$ is split with dimension $m=2n$. We use $G_n$ to denote the isometry group of $V$. Fix a decomposition of $V$ as follows,
\[
V=\langle e_1\rangle\oplus \cdots \oplus\langle e_n\rangle\oplus  \langle f_n\rangle\oplus \dots \oplus \langle f_1\rangle,
\]
where $e_i,f_i,1\leq i\leq n$ are isotropic vectors such that $(e_i,f_j)=\delta_{i,j}$. Let $J$ be the presentation matrix under this basis, in other words, $J$ is anti-diagonal,
\[J=\begin{pmatrix} 0_{n\times n} & A_n\\ \epsilon\cdot A_n & 0_{n\times n} \end{pmatrix},\]
where $A_n=\begin{pmatrix} 0 & \dots & 0 & 1\\ 0 & \dots & 1 & 0\\ \vdots &  \iddots & \vdots & \vdots \\ 1 & \dots & 0 & 0\end{pmatrix}$ is an $n\times n$ matrix with anti-diagonal elements $1$.

We define the following subgroups of $G_n$:  
\begin{itemize}
\item Let $G_{n-k}$ embeds into $G_n$ as the subgroup fixing $e_1,\dots,e_k$.
\item The Cartan subgroup $A$ is chosen as the subgroup stabilizing $\langle e_1\rangle,\langle e_2\rangle,\dots ,\langle e_n\rangle$.
    \item \textbf{Mirabolic subgroup} $Q_n$: the subgroup fixing $e_1$, with unipotent radical denoted $E_n$.
    \item Standard \textbf{maximal parabolic subgroup} $R_{n-k,k}$ ($0\leq k\leq n$): the parabolic subgroup stabilizing the subspace $\langle e_1\rangle\oplus \cdots \oplus\langle e_{n-k}\rangle$. When $k=0$, we simply denoted $R_{n-k,k}$ by $C_n$, which is called the \textbf{Siegel parabolic subgroup}. It has a standard Levi decomposition $C_n=\GL_n\ltimes U_n$, where $\GL_n$ is the general linear group of $\langle f_1\rangle\oplus \cdots \oplus\langle f_{n}\rangle$ and $U_n$ refers to the unipotent radical of $C_n$.  
\end{itemize}
If we write elements in $V$ as column vectors according to basis $\{e_1,\dots,e_n, f_n,\dots,f_1 \}$, then 
\[Q_n=\begin{pmatrix} 1 &  * & * \\ 0_{ (2n-2)\times 1} & * & * \\ 0 & 0_{1\times (2n-2)} & 1 \end{pmatrix}\cap G_n, \quad   R_{n-k,k}=\begin{pmatrix} * &  * & * \\ 0_{ (2n-2k)\times k} & * & * \\ 0_{k\times k} & 0_{k\times (2n-2k)} & * \end{pmatrix}\cap G_n.\]   
Moreover, let $C_{n-k}$, $U_{n-k}$ denote the subgroups of $G_{n-k}\subset G_n$.

We define a character of $E_n$:
\[
\psi_{n}(x):=\psi((x\cdot f_1,e_2)),\ x\in E_n.
\]
The stabilizer of $\psi_n$ under $G_{n-1}$-action is $Q_{n-1}$. Note that there is an abuse of notation since $\psi_n$ is used as a character of $V_n$ as well. Since the character is attached to different groups, it will cause no confusion. 

For convenience, inside $G_n$, let $V_n$ denote the subgroup $E_n\cap \GL_n$, and let $F_n$ denote the subgroup $E_n\cap U_n$.

\begin{definition}
    Let $\sigma$ be a smooth moderate growth Fr\'echet representation of $M_n$, we define
     $$\Psi(\sigma):=|\det|_{\mathbf{k}}^{-1/2} \otimes \sigma/\Span\{\alpha v -\psi_n(\alpha)v \mid v\in\sigma, \alpha\in \fke_n\}$$
and
$$ \Phi_0(\sigma):= \sigma /\Span\{\kappa v\,\mid v  \in \sigma ,\kappa \in \fke_n\}.$$
Here, $\Psi(\sigma)$ is a representation of $M_{n-1}$ and $\Phi_0(\sigma)$ is a representation of $G_{n-1}$.
\end{definition}
For applications to the Bessel model, it is helpful to introduce the following functor. Let $V'\subset V$ be a hermitian subspace such that $(V')^{\perp}=\mathrm{Span}_{\mathbf{k}}\{ e_1,f_1\}\oplus^{\perp} \mathbf{k}\cdot Z$ for an anisotropic vector $Z$. Let $\phi_n$ be the unitary character of $E_n$ defined by
\[
\phi_n(x):=\psi((x\cdot f_1,Z)),x\in E_n.
\]
\begin{definition}
    Let $\sigma$ be a smooth moderate growth Fr\'echet representation of $M_n$, we define
    \[
    \Upsilon(\sigma):=|\det|_{\mathbf{k}}^{-1/2} \otimes \sigma/\Span\{\alpha v -\phi_n(\alpha)v \mid v\in\sigma, \alpha\in \fke_n\}.
    \]
    Here, $\Upsilon(\sigma)$ is a representation of the isometry group of $V'$.
\end{definition}

\subsection{Lie algebra homology} 
In this subsection, let $\fkh$ be a complexified Lie algebra of some almost linear Nash group $H$. Let $M$ be an object in the abelian category of algebraic $\fkh$-representations, then the $i$-th Lie algebra homology $\rmh_i(\fkh,M)$ is defined as the $i$-th left derived functor of the right exact functor ``co-invariant":
\begin{equation*}
    \text{Rep}(\fkh)  \lra  \text{Vect}_{\BC}, \qquad 
    M \mapsto M/\Span\{ X\cdot m\mid X\in \fkh,m\in M\}.
\end{equation*} 
It is sometimes helpful to interpret the ``co-invariant" functor as the ``tensor product" functor:
\begin{equation*}
    \text{Rep}(\fkh)  \lra  \text{Vect}_{\BC}, \qquad 
    M \mapsto M\otimes_{\U(\fkh)}\text{triv},
\end{equation*} 
where ``$\mathrm{triv}$'' is the trivial representation of $\fkh$. By the Koszul resolution of the trivial representation,  $\rmh_i(\fkh,M)$ is isomorphic to the $i$-th homology of the Koszul complex \begin{equation}\label{koszul}
    0 \stackrel{d_0}{\longleftarrow} M \stackrel{d_1}{\longleftarrow} \fkh\otimes M \stackrel{d_2}{\longleftarrow} \dots\stackrel{d_{\dim(\fkh)}}{\longleftarrow}\wedge^{\dim(\fkh)}\fkh\otimes M  \longleftarrow 0.
\end{equation}
When $M$ is equipped with a Fr\'echet topology, we would like to equip 
\[
\rmh_i(\fkh,M)\simeq \Ker(d_i)/ \mathrm{Im}(d_{i+1}) 
\]
with the subquotient topology. Note that this topology is not necessarily Hausdorff.

Given a right exact functor $F$ between two abelian categories with enough projective objects, let $L^iF$ denote the $i$-th left derived functor of $F$. What we concern in this article are left derived functors of various derivatives. When topology matters, we also equip these left derived functors with topology given by Koszul resolution.

The following homological version of Mittag-Leffler lemma is critical for deducing the Hausdorffness of the homology of the space, which admits Borel filtration, see Lemma \ref{ext_Hausd}. Recall that an inverse system $\{V_k,\alpha_k:V_{k+1}\to V_k\}_{k\geq 0}$ is called \textbf{stationary} if for every positive integer $n$, there exists an integer $\nu(n)\geq n$ such that for all $p\geq \nu(n)$, \[\mathrm{Im}(V_p\to V_n)=\mathrm{Im}(V_{\nu(n)}\to V_n).\]
\begin{lemma}[see \cite{Gr}, Chapter 0, Proposition 13.2.3]\label{ML-lem}
Let $\{V_k,\alpha_k:V_{k+1}\to V_k\}_{k\geq 0}$ be an inverse system of $\fkh$-representations. Let $V:=\varprojlim\limits_{k}V_k$. Assume
    \begin{enumerate}
        \item $\{V_k,\alpha_k:V_{k+1}\to V_k\}_{k\geq 0}$ is stationary;
        \item $\{\rmh_i(\fkh,V_k),\alpha_k:\rmh_i(\fkh,V_{k+1})\to \rmh_i(\fkh,V_k)\}_{k\geq 0}$ is also stationary for each $i\in\BZ$.
    \end{enumerate}
    Then the complex
    \begin{equation*}
      0\longleftarrow \mathop{\varprojlim}\limits_{k}\rmh_i ( \fkh, V_k)\longleftarrow   \Ker d_i\longleftarrow \wedge^{i+1}\fkh\otimes V
    \end{equation*}
    is exact, where $d_i$ is the differential map in \eqref{koszul} with $M=V$.
\end{lemma}
 The following lemma plays a fundamental role in deducing the Hausdorffness of the extension of two Hausdorff representations. Let 
 \[
 0\lra L\lra M\lra N\lra 0
 \]
 be a short exact sequence of nuclear Fr\'echet representations of $\fkh$. Before the lemma, we remark the following two points.
 \begin{enumerate}
     \item When $\rmh_i(\fkh,L)$ is Hausdorff, by \cite[Proposition 5.3.2]{AGS15b}, we have
 \[
 \rmh_i(\fkh,L)'\simeq  \rmh^i(\fkh,L').
 \]
 \item The boundary map $\partial_i:\rmh_{i+1}(\fkh,N)\lra \rmh_i(\fkh,L)$ is continuous by the topology of the Fr\'echet space.
 \end{enumerate}
 \begin{lemma}\label{exa-haus-lem}
     Let $i$ be an arbitrary integer. Suppose $\rmh_i(\fkh,L)$ and $\rmh_i(\fkh,N)$ are Hausdorff, and the boundary map
     \begin{equation*}
       \partial_i:  \rmh_{i+1}(\fkh,N)\lra \rmh_i(\fkh,L)
     \end{equation*}
     has closed image, then $\rmh_i(\fkh,M)$ is Hausdorff.
 \end{lemma}
 \begin{proof}
     Consider the Koszul resolution of short exact sequence:
     \[
     \begin{tikzcd}
         0\ar[r] &\wedge^i \fkh\otimes L \ar[r,"\phi_i"]& \wedge^i \fkh\otimes M \ar[r,"\varphi_i"] & \wedge^i \fkh\otimes N \ar[r]& 0\\
         0\ar[r] &\wedge^{i+1} \fkh\otimes L \ar[r,"\phi_{i+1}"]\ar[u,"\alpha_i"]& \wedge^{i+1} \fkh\otimes M \ar[r,"\varphi_{i+1}"]\ar[u,"\kappa_i"] & \wedge^{i+1} \fkh\otimes N \ar[r]\ar[u,"\gamma_i"]& 0\\
     \end{tikzcd}
     \]
    Note that $\rmh_i(\fkh,M)$ is Hausdorff is equivalent to $\mathrm{Im}(\kappa_i)$ is closed in $\wedge^i \fkh\otimes M $. Consider the short exact sequence
    \begin{equation*}
        0\lra N'\lra M'\lra L'\lra 0
    \end{equation*}
    and its dual Koszul resolution
    \[
    \begin{tikzcd}
         0 &\wedge^i \fkh^*\otimes L'\ar[l] \ar[d,"\alpha_i'"]& \wedge^i \fkh^*\otimes M'\ar[l,"\phi_i'"] \ar[d,"\kappa_i'"] & \wedge^i \fkh^*\otimes N' \ar[l,"\varphi_i'"]\ar[d,"\gamma_i'"]&\ar[l] 0\\
         0 &\wedge^{i+1} \fkh^*\otimes L' \ar[l]& \wedge^{i+1} \fkh^*\otimes M'  \ar[l,"\phi_{i+1}'"]& \wedge^{i+1} \fkh^*\otimes N' \ar[l,"\varphi_{i+1}'"]& \ar[l]0\\
     \end{tikzcd}
    \]
    Let $x\in\Ker \kappa_{i-1}$, such that $\eta(x)=0$ for every $\eta\in\Ker \kappa_i'$. Note that $\mathrm{Im}\kappa_i$ is closed if and only if $x\in \mathrm{Im}\kappa_i$ for every such $x$. Since $\rmh_i(\fkh,N)$ is Hausdorff, we have $\varphi_i(x)\in \mathrm{Im}\gamma_i$. Thus, we can take an element $x'\in \mathrm{Im}\kappa_i$ such that $\varphi_i(x')=\varphi_i(x)$. We have $x-x'\in \Ker\alpha_{i-1}$, and it is equivalent to show $x-x'\in \Ker \alpha_{i-1}\cap \mathrm{Im}\kappa_i$. We project $x-x'$ into $\rmh_i(\fkh,L)$, and still use the same notation. We need only to show 
    \begin{equation*}
        x-x'\in (\Ker \alpha_{i-1}\cap \mathrm{Im}\kappa_i)/\mathrm{Im}\alpha_i.
    \end{equation*}
    Consider the long exact sequence associated to the short exact sequence
    \begin{equation*}
 \dots  \lra  \rmh_{i+1}(\fkh,N)\stackrel{\partial_i}{\lra}   \rmh_i(\fkh,L) \stackrel{d_i}{\lra} \rmh_i(\fkh,M)\lra \rmh_i(\fkh,N)\lra \dots.
    \end{equation*}
    By definition, $(\Ker \alpha_{i-1}\cap \mathrm{Im}\kappa_i)/\mathrm{Im}\alpha_i=\Ker d_i=\mathrm{Im}\partial_i$. Hence it is closed in $\rmh_i(\fkh,L)$. Moreover, it has the following characterization since $\rmh_i(\fkh,L)$ is Hausdorff: an element $y\in\rmh_i(\fkh,L)$ falls in $ (\Ker \alpha_{i-1}\cap \mathrm{Im}\kappa_i)/\mathrm{Im}\alpha_i$ if and only if for every $\theta\in \mathrm{Im}d_i'$, $\theta(y)=0$. Here $d_i'$ is the dual map in the long exact sequence of Lie algebra cohomology
    \[
    d_i': \rmh^i(\fkh,L') \stackrel{d_i'}{\longleftarrow} \rmh^i(\fkh,M').
    \]
    This holds for $y=x-x'$ by our requirement on $x$.
 \end{proof}
 \begin{remark}\label{rem_haus}
     The proof of the above lemma utilizes the dual nuclear Fr\'echet complex. In fact, such a proof also applies to a more general statement that we will use. Let 
     \[
     0\lra Y_{\bullet}\lra Z_{\bullet} \lra W_{\bullet} \lra 0
     \]
     be a short exact sequence of nuclear Fr\'echet complexes. Let $i$ be an arbitrary integer. If $\rmh_i(Y_{\bullet})$ and $\rmh_i(W_{\bullet})$ are Hausdorff, and the boundary map 
     \[
     \partial_i:\rmh_{i+1}(W_{\bullet})\lra \rmh_i(Y_{\bullet})
     \]
     has closed image, then $\rmh_i(Z_{\bullet})$ is Hausdorff.
 \end{remark}
 By an argument similar to that in Lemma~\ref{exa-haus-lem}, we can establish the following result. The details are left to the reader.
 \begin{lemma}\label{exa-hau-2}
 Following the notation in Lemma~\ref{exa-haus-lem}. Let $i$ be an arbitrary integer. Suppose that $\rmh_i(\fkh,M)$ and $\rmh_{i+1}(\fkh,N)$ are Hausdorff, and the induced map on the homology
 \[
 \rmh_{i+1}(\fkh,M)\lra \rmh_{i+1}(\fkh,N)
 \]
 has a closed image, then $\rmh_i(\fkh,L)$ is Hausdorff. 
 \end{lemma}

 We will need one more lemma for proving Hausdorffness of the homology group in the context of the Borel filtration. For general setting, let $\pi$ be a representation of $G\ltimes H$, where $G$ is a real reductive group. Assume $\pi$ admits a decreasing filtration $\{F^i\pi\}_{i\in\BZ_{\geq 0}}$ of $G\ltimes H$ such that the canonical map
 \[
 \pi\lra \mathop{\varprojlim}\limits_{i} \pi/F^i\pi
 \]
 is an isomorphism. 
 \begin{lemma}\label{ext_Hausd}
     Let $i$ be an arbitrary integer. Assume that $\rmh_i(\fkh,F^{j-1}\pi/F^j\pi)$ is a Casselman-Wallach representation of $G$ for every integer $j$, then $\rmh_i(\fkh,\pi)$ is Hausdorff.
 \end{lemma}
 \begin{proof}
     We first show that $\rmh_i(\fkh,\pi/F^j\pi)$ is Casselman-Wallach for every integer $i,j$. Let us argue by induction on $j$. The base case $j=1$ follows from assumption. Assuming that it is true for $j$, we prove it for $j+1$. Consider the long exact sequence
     \[
     \rmh_{i+1}(\fkh,\pi/F^{j}\pi)\stackrel{\partial_i}{\lra}\rmh_i(\fkh,F^j\pi/F^{j+1}\pi)\lra \rmh_i(\fkh,\pi/F^{j+1}\pi)\lra \rmh_i(\fkh,\pi/F^{j}\pi), 
     \]
     since $\rmh_i(\fkh,F^j\pi/F^{j+1}\pi)$ is Casselman-Wallach, $\partial_i$ has closed image by Remark~\ref{closed-im-rem}. Hence we conclude $\rmh_i(\fkh,\pi/F^{j+1}\pi)$ is Casselman-Wallach from Lemma~\ref{exa-haus-lem}. 
    
     Note that the Casselman-Wallach representation has finite length, hence satisfies two stationary conditions of Mittag-Leffler Lemma~\ref{ML-lem}. Consequently, we have an exact sequence:
     \[
      0\longleftarrow \mathop{\varprojlim}\limits_{j}\rmh_i ( \fkh, \pi/F^j\pi)\longleftarrow   \Ker d_i\stackrel{d_{i+1}}{\longleftarrow} \wedge^{i+1}\fkh\otimes \pi.
     \]
     Thus 
     $$\mathrm{Im}d_{i+1}=\bigcap_j {(p^j_i)}^{-1}(0) \quad \text{ is closed }, $$
     where $p^j_i:\Ker d_i\to \Ker d_i^j\to \rmh_i ( \fkh, \pi/F^j\pi)$ is a continuous map. Here, $d_i^j$ is the differential of the Koszul complex for $\pi/F^j\pi$
     \[
     \lra\wedge^{i+1}\fkh \otimes \pi/F^j\pi\stackrel{d_{i+1}^j}{\lra}\wedge^{i}\fkh \otimes \pi/F^j\pi\stackrel{d_{i}^j}{\lra}\wedge^{i-1}\fkh \otimes \pi/F^j\pi\lra.
     \]
 \end{proof}

On the other hand, instead of a single representation, we will encounter a complex of representations with a finite filtration. Let $Y_{\bullet}$ be a complex of nuclear Fr\'echet spaces with a finite increasing filtration by closed subspaces
\[
 Y_{\bullet}=\CF^k\supset\CF^{k-1}\supset \dots \supset \CF^0=0.
\]
Let $E_0^{p,\bullet}=\CF^p/\CF^{p-1}$ and 
\[
d_0^{p,\bullet}: \CF^p/\CF^{p-1}\lra \CF^p/\CF^{p-1}
 \]
be the differential map in the complex. Inductively, we can define a spectral sequence $(E_r^{p,q},d_r^{p,q})_{r\geq 0}$, see \cite[section 5.4]{Wei} for details. Moreover, for two short exact sequences of complexes
\begin{equation}\label{fil-exa}
    0\lra \CF^p/\CF^{p-1}\lra \CF^{p+r-1}/\CF^{p-1}\lra \CF^{p+r-1}/\CF^p\lra 0,
\end{equation}
and 
\begin{equation*}
    0\lra \CF^{p-1}/\CF^{p-r}\lra \CF^p/\CF^{p-r}\lra \CF^{p}/\CF^{p-1}\lra 0,
\end{equation*}
 we define
 \begin{equation}\label{eq_B_r}
      B_r^{p,q}=\mathrm{Im} \left(\partial_r^{p,q}:  \rmh_{q+1}(\CF^{p+r-1}/\CF^p)\lra \rmh_{q} (\CF^p/\CF^{p-1})\right),
 \end{equation}
 and
 \[
 Z_r^{p,q}=\Ker \left(\epsilon_r^{p,q}:\rmh_{q}(\CF^p/\CF^{p-1})\lra \rmh_{q-1}(\CF^{p-1}/\CF^{p-r})\right).
 \]
 It is a standard fact that $B_r^{p,q}\subset Z_r^{p,q}$ and $E_r^{p,q}\simeq Z_r^{p,q}/B_r^{p,q}$ as topological vector spaces for every integer $p,q$ and $r\geq 1$.
\begin{lemma}\label{spec-haus}
    The notation is the same as above. If $E_r^{p,q}$ is Hausdorff for every $r\geq 1$, then $\rmh_i(Y_{\bullet})$ is Hausdorff for every integer $i$. 
\end{lemma}
\begin{proof}
    By equation~\eqref{eq_B_r}, we observe that ``$E_r^{p,q}$ is Hausdorff" is equivalent to ``$\partial_r^{p,q}$ has closed image". We prove by induction on $r$ that for every $p$, $\rmh_{\bullet}(\CF^p/\CF^{p+r})$ is Hausdorff. When $r=1$, then the result follows from
    \[
    \rmh_{\bullet}(\CF^p/\CF^{p+r})\simeq E_1^{p,\bullet}.
    \]
    Assume that the statement holds for some $r-1$, we prove the statement for $r$. Consider the short exact sequence \eqref{fil-exa}, by the induction hypothesis, we have $\rmh_{\bullet}(\CF^{p+r-1}/\CF^p)$ and $\rmh_{\bullet}(\CF^p/\CF^{p-1})$ is Hausdorff. Furthermore, $\partial_r^{p,\bullet}$ has closed image. Consequently, by Remark~\ref{rem_haus}, the statement follows.
\end{proof}

\subsection{Filtration of a representation}
To understand the branching law of the restriction to the parabolic subgroup, we introduce the following definition of filtration.
\begin{definition}\label{def_fil}
Given a representation $\sigma$ of an almost linear Nash group $G$, a \textbf{level $\leq 1$ filtration} of $\sigma$ consists of the data
\begin{itemize}
    \item[(i)] Finite decreasing subrepresentations of $\sigma$, \[\sigma= \sigma_0 \supset \sigma_1 \supset \dots \supset \sigma_m,\]
    \item[(ii)] For all $0 \leq i \leq m-1$, a finite or infinite decreasing chain of subrepresentations of $\sigma_i/\sigma_{i+1}$, 
    \[
    \sigma_i = \sigma_{i,0} \supset \sigma_{i,1} \supset \sigma_{i,2} \supset \dots \supset \sigma_{i+1},
    \]
    such that the canonical map
    $
    \sigma_i/\sigma_{i+1} \to \varprojlim_j \sigma_{i} / \sigma_{i,j}
    $
    is a topological isomorphism of $G$-representations.
\end{itemize}

A \textbf{level $\leq r$ filtration} of $\sigma$ consists of the data described above, with the additional requirement that each quotient $\sigma_{i,j} / \sigma_{i,j+1}$ is equipped with a level $\leq r-1$ filtration.

Given a level $\leq r$ filtration, for every pair of subrepresentations $\sigma^{\flat} \supset \sigma^{\sharp}$ in the filtration such that there are no other terms between $\sigma^{\flat}$ and $\sigma^{\sharp}$, we call the quotient $\sigma^{\flat} / \sigma^{\sharp}$ a \textbf{successive quotient} of the filtration.
\end{definition}

\begin{example}
Considering the multiplication action of $\mathbb{R}^{\times}$ on $\mathbb{R}$. Since $\mathbb{R}^{\times}\subset \mathbb{R}$ is $\mathbb{R}^{\times}$-stable open subset, $\mathcal{S}(\mathbb{R}^{\times})$ is an $\mathbb{R}^{\times}$-stable subspace.

By Borel's lemma, one has a short exact sequence \[0\to \mathcal{S}(\mathbb{R}^{\times})\to \mathcal{S}(\mathbb{R})\to \mathbb{C}[[x]]\to 0.\]
The formal power series $\mathbb{C}[[x]]$ admits a level $\leq 1$ filtration as representation of $\mathbb{R}^{\times}$,
\[\mathbb{C}[[x]]\supset x\cdot \mathbb{C}[[x]]\supset \dots\supset  x^k\cdot \mathbb{C}[[x]]\supset \dots,\]
whose successive quotients are $(\det)^i$, for $i\in \mathbb{Z}_{\geq 0}$.
\end{example}

The following lemma is useful in the study of twisted homology. Let $H$ be an almost linear Nash group.
\begin{lemma}\label{fil-lem}
    Let $\sigma$ be a representation of $H$ with a level $\leq r$ filtration. Suppose that each successive quotient $\beta$ of the filtration satisfies $\rmh_l(\fkh,\beta)=0$ for every integer $l\geq 1$
    and $\rmh_0(\fkh,\beta)$ is Hausdorff. Then, $\rmh_l(\fkh,\sigma)=0$ for every integer $l\geq 1$
    and $\rmh_0(\fkh,\sigma)$ is Hausdorff. 
\end{lemma}
\begin{proof}
We proceed by induction on the level of the filtration. First, assume $r=1$. Following the notation of Definition~\ref{def_fil}, it suffices to prove the statement for $\sigma_i/\sigma_{i+1}$ with $0 \leq i \leq m-1$. By Lemma~\ref{exa-haus-lem}, for every integer $j\geq 0$, we have $\rmh_l(\fkh,\sigma_i/\sigma_{i,j}) = 0$ for every integer $l \geq 1$, and $\rmh_0(\fkh,\sigma_i/\sigma_{i,j})$ is Hausdorff. Moreover, the map
\[
\rmh_0(\fkh,\sigma_i/\sigma_{i,j}) \to \rmh_0(\fkh,\sigma_i/\sigma_{i,j'})
\]
is surjective for every $j \geq j'$. Therefore, the inverse system $\{\rmh_l(\fkh,\sigma_i/\sigma_{i,j})\}_{j \geq 0}$ is stationary for every integer $l$. By an argument similar to that in Lemma~\ref{ext_Hausd}, the statement for $\sigma_i/\sigma_{i+1}$ follows.

Now, assume the statement holds for filtrations of level $\leq r-1$. Then the statement for filtrations of level $\leq r$ holds by the same argument used for filtrations of level $\leq 1$.
\end{proof}

\subsection{Category $\CC(\fkg,L)$} \label{category C}
In this subsection, our main result is that the Lie algebra homology of objects in a certain category $\CC(\fkg,L)_f$ yields Casselman-Wallach representations. We first set up notations of this subsection. Let $P$ be a parabolic subgroup of a real reductive group $G$ with Levi decomposition $P=LU$. Let $Z_L$ be the center of $L$. For a representation $\tau$ of $L$, we define the generalized $\fkz_L$-weight subspace of weight $\alpha\in\fkz_L^*$ by
\[
 \tau_{\alpha}:= \{v\in\tau\mid (X-\alpha(X))^kv=0\ \text{ for some } k\in\BZ_{\geq 0}, \forall X\in \fkz_L\}.
\]
Moreover, let $\mathrm{wt}(\tau)$ denote the set of generalized $\fkz_L$-weights of $\tau$ such that the weight space is non-zero. The set of $\fkz_L$-weights in $\U(\fku)$ is denoted by $\Omega$. We define a partial order on $\fkz_L^*$ as follows:
\begin{equation}\label{partial_order}
\alpha \leq \kappa \quad \text{ if and only if } \quad \kappa-\alpha\in \Omega.
\end{equation}
\begin{definition}\label{def of cate}
A $(\fkg,L)$-module is a Fr\'echet space $V$ equipped with compatible continuous $\U(\mathfrak{g})$-action and smooth moderate growth $L$-representation structures. Let $\mathcal{C}(\fkg,L)$ be the category of $(\fkg,L)$-modules $V$ such that
\begin{itemize}
\item[(i)] Let $V^{\fkz_L\text{-finite}}$ be the $\fkz_L\text{-finite}$ subspace of $V$. Then $\ov{\fku}$-action on $V^{\fkz_L\text{-finite}}$ is locally finite. Moreover, for every $\alpha\in \fkz_L^*$, $V_{\alpha}$  equipped with subspace topology is a Casselman-Wallach representation of $L$.
\item[(ii)] For every finite subset $S\subset \fkz_L^*$, as topological vector space, 
\[V\simeq \mathop{\oplus}\limits_{\alpha\in S}V_{\alpha}\bigoplus \overline{ \mathop{\oplus}\limits_{\alpha\in \fkz_L^*\setminus S}V_{\alpha}}.\]
\item[(iii)] The canonical map gives an isomorphism as topological vector space 
\[V\simeq \mathop{\varprojlim}\limits_{S\subset \fkz_L^* \text{ finite}}V/\overline{ \mathop{\oplus}\limits_{\alpha\in \fkz_L^*\setminus S}V_{\alpha}}.\]
\end{itemize}
Morphisms in this category are the continuous linear maps intertwining the action of both $L$ and $\U(\fkg)$.
\end{definition}

For $V\in\CC(\fkg,L)$, denote by $\mathrm{pr}_{\kappa}$ the natural projection 
\begin{equation}\label{projection}
     V\lra V_{\kappa}\simeq V/\ov{\oplus_{\alpha\neq \kappa} V_{\alpha}}.
\end{equation}
For a closed submodule $M$ of an object in $V\in\CC(\fkg,L)$, the subobject generated by it is $\ov{\bigoplus_{\kappa\in \fkz_L^*} \mathrm{pr}_{\kappa}(M)}.$
 
 Let $\CC(\fkg,L)_f$ be the full subcategory consisting of finite length objects. Such a category assembles the characteristics of the BGG category $\CO$ and Casselman-Wallach representations. We now introduce a standard object in this category. 
\begin{definition}
Given a Casselman-Wallach representation $\tau$ of $L$, the 
\textbf{formal Verma module} $\CV(\tau)$ is defined as the topological inverse limit
\[
 \mathop{\varprojlim}\limits_{k\geq 0} \left(\U(\fkg)\otimes_{\U(\ov{\fkp})}\tau\right)/ \left(\fku^k\U(\fkg)\otimes_{\U(\ov{\fkp})}\tau\right).
\]
Here $\tau$ extends trivially to be a $\U(\ov{\fkp})$-module. As a topological vector space, $\U(\fku)^{<k}$ is equipped with Euclidean topology and 
\[
\left(\U(\fkg)\otimes_{\U(\ov{\fkp})}\tau\right)/ \left(\fku^k\U(\fkg)\otimes_{\U(\ov{\fkp})}\tau\right)\simeq \U(\fku)^{< k }\otimes \tau
\]
is equipped with the tensor product Fr\'echet topology.
\end{definition}

The functor ``$\tau\mapsto \CV(\tau)$" is an exact functor from the category of Casselman-Wallach representations of $L$ to $\CC(\fkg,L)$. 

Let $\tau$ be an irreducible representation of $L$. For a closed submodule $M$ of $\CV(\tau)$, we can define $L$-equivariant continuous maps 
$$\mathrm{pr}_{\kappa}:M \to M_{\kappa}\text{ for all }\kappa \in \fkz_L^*$$
using the action of the center $Z_L$. Let $\wt(\tau) = \alpha$. Given an element $m \in M$ and an element $z \in \fkz_L^*$ such that $\gamma(z) > 0$ for all $\gamma \in \Omega\setminus\{0\}$, the expression $\exp(-t z)\exp(t\alpha(z)) \cdot m$ converges to some $m_{\alpha} \in M_{\alpha}$ as $t \to \infty$. This convergence defines the map $\mathrm{pr}_{\alpha}$. Considering $m - m_{\alpha}$ and proceeding inductively along the partial order $\leq$, we obtain all the maps $\mathrm{pr}_{\kappa}$. When $M\in\CC(\fkg,L)$, this map coincides with~\eqref{projection}, thus we use the same notation. 
\begin{lemma}
    Let $\tau$ be an irreducible representation of $L$. Suppose $M$ is a closed submodule of $\CV(\tau)$, then $M$ and $N:=\CV(\tau)/M$ are objects in category $\CC(\fkg,L)$.
\end{lemma}
\begin{proof}
    By the map $\mathrm{pr}_{\kappa}$, we know that $M$ is the same as the subobject generated by it. On the other hand, by the construction of the map $\mathrm{pr}_{\kappa}$, we have the following short exact sequence for each $\kappa\in \fkz_L^*$
    \[
    0\lra M_{\kappa}\lra \CV(\tau)_{\kappa}\lra N_{\kappa}\lra 0.
    \]
    Therefore, we have the short exact sequence for surjective inverse system
    \[
    0\lra \mathop{\varprojlim}\limits_{S\subset \fkz_L^* \text{ finite} }\bigoplus_{\kappa\in S}M_{\kappa}\lra \mathop{\varprojlim}\limits_{S\subset \fkz_L^* \text{ finite}} \bigoplus_{\kappa\in  S}\CV(\tau)_{\kappa}\lra \mathop{\varprojlim}\limits_{S\subset \fkz_L^* \text{ finite}} \bigoplus_{\kappa\in S}N_{\kappa}\lra 0,
    \]
    which implies the topological isomorphism $ N\simeq \mathop{\varprojlim}\limits_{S\subset \fkz_L^* \text{ finite}} \bigoplus_{\kappa\in S}N_{\kappa}$.
\end{proof}

We now investigate irreducible objects in $\CC(\fkg,L)$. 
\begin{lemma}\label{irr_quo}
  Let $\tau$ be an irreducible Casselman-Wallach representation of $L$.Then the formal Verma module $\CV(\tau)$ has a unique maximal closed proper submodule, hence, a unique non-zero irreducible quotient.
\end{lemma}
\begin{proof}
    We need only to prove that for proper closed submodules $M_1,M_2$, $\ov{M_1+M_2}$ remains  proper. This is because, under the partial order \eqref{partial_order}, $\mathrm{wt}(\tau)< $ every element in both $\mathrm{wt}(M_1)$ and $\mathrm{wt}(M_2)$ by the structure of $\CV(\tau)$.  This implies $\mathrm{wt}(\tau)\notin \mathrm{wt}(M_1+M_2)$. By the condition (iii) of Definition~\ref{def of cate} on the topology, we find $\mathrm{wt}(\tau)\notin \mathrm{wt}(\ov{M_1+M_2})$.
\end{proof}

We denote the unique irreducible quotient of $\CV(\tau)$ by $\CL(\tau)$. Note that $\CL(\tau_1)\simeq \CL(\tau_2)$ if and only if $\tau_1\simeq \tau_2$ as $L$-representation. On the other hand, we have the following lemma:
\begin{lemma}\label{Irr_obj}
Let $V$ be an irreducible object in $\CC(\fkg,L)$. Then there exists an irreducible Casselman-Wallach representation $\tau$ of $L$, such that $V$ is a quotient of the formal Verma module $\CV({\tau})$.
\end{lemma}
\begin{proof}
Since $V$ is irreducible, $\mathrm{wt}(V)$ has a unique minimal element $\alpha$. Then $V_{\alpha}$ is irreducible as $L$-representation; otherwise, we can take some proper subrepresentation of $V_{\alpha}$ to generate a proper submodule of $V$. 

Let $S_k:=\alpha+\mathrm{wt}(\U(\fku)^{<k})$. Then by the PBW theorem, there exists a continuous surjective map for each $k$ by $\fku$-action
\[
 \U(\fku)^{< k }\otimes V_{\alpha}\lra \bigoplus_{\kappa\in S_k} V_{\kappa}\simeq V/\ov{\bigoplus_{\kappa\notin S_k} V_{\kappa}},
\]
which implies a continuous surjective map:
\[
  \mathop{\varprojlim}\limits_{k\geq 0} \left(\U(\fkg)\otimes_{\U(\ov{\fkp})}\tau\right)/ \left(\fku^k\U(\fkg)\otimes_{\U(\ov{\fkp})}\tau\right)\lra  \mathop{\varprojlim}\limits_{k\geq 0}V/\ov{\bigoplus_{\kappa\notin S_k} V_{\kappa}}\simeq \mathop{\varprojlim}\limits_{S\subset \fkz_L^* \text{ finite}}V/\overline{ \mathop{\oplus}\limits_{\alpha\in \fkz_L^*\setminus S}V_{\alpha}} \simeq V.
\]
\end{proof}

Hence, we get the classification of irreducible objects in $\CC(\fkg,L)$.
\begin{corollary}
    There is a one-to-one correspondence between irreducible Casselman-Wallach representations of $L$ and irreducible objects in $\CC(\fkg,L)$ given by
    \[
    \tau \mapsto \CL(\tau).
    \]
\end{corollary}

Moreover, if the infinitesimal character of $\tau$ is $\chi_{\lambda}$, then the infinitesimal character of $\CV(\tau)$ is $\chi_{\overline{\lambda-\rho+\rho_{\fkl}}}$, where $\overline{\lambda}$ is the image of $\lambda$ under the following natural projection:
\begin{equation*}
    \fka^*  \dslash W_L \lra \fka^* \dslash W .
\end{equation*}
For an infinitesimal character $\chi_{\mu}$ of $\fkg$, we use $\CT_{\mu}$ to denote the set of irreducible Casselman-Wallach representations $\tau$ of $L$, for which $\CV(\tau)$ has infinitesimal character $\chi_{\mu}$. Then $\CT_{\mu}$ is a finite set. 
\begin{lemma}\label{finite length lem}
  If $V\in \CC(\fkg,L)$ has some infinitesimal character $\chi_{\mu}$, then $V\in\CC(\fkg,L)_f$.
\end{lemma}
\begin{proof}
    Suppose that $V$ is not of finite length. Then we can successively apply the following operations to get an infinite filtration by subobjects of $V$ such that the successive quotient is irreducible. 
    
    Since $V$ has infinitesimal character and is locally $\overline{\mathfrak{u}}$-finite, by the relation between the infinitesimal characters of formal Verma module $\mathcal{V}(\tau)$ and that of $\tau$, it follows  that $\mathrm{wt}(V)$ has a minimal value under the partial order \eqref{partial_order}.  Take an element $\alpha\in \min \mathrm{wt}(V)$, and an irreducible sub $L$-representation $\tau$ of $V_{\alpha}$. Consider the subobject $V_{\tau}$ generated by $\tau$. Since $\alpha\in \min \mathrm{wt}(V)$, the $\ov{\fku}$ acts on $\tau$ trivially. By the same argument as Lemma~\ref{irr_quo}, $V_{\tau}$ will have an irreducible quotient
    \[
    \varphi: V_{\tau} \twoheadrightarrow V'.
    \]
Then apply similar operation to $\Ker \varphi$, other irreducible subquotients of $V_{\alpha}$, and then $V/\ov{\U(\fkg)\cdot V_{\alpha}}$. 

Note that the successive quotients also have infinitesimal character $\chi_{\mu}$, hence, they are of the form $\CL(\tau)$ for $\tau\in\CT_{\mu}$. On the other hand, each $\CL(\tau)$ can only appear finitely many times, since $V_{\alpha}$ is finite length $L$-representation for every $\alpha\in \fkz_L^*$. This contradicts the infiniteness of the filtration.
 
\end{proof}

In BGG category $\CO$, the Verma module is irreducible when the lowest weight $\lambda$ is dominant. An analogous phenomenon occurs in the category $\CC(\fkg,L)$. For $\mu\in \fka^*$, we introduce the following notation:
\[
 \mathrm{wt}(\CT_{\mu}):=\{\mathrm{wt}(\tau)\mid \tau\in \CT_{\mu}\}.
\]
\begin{lemma}\label{max irr}
    Let $\chi_{\mu}$ be an infinitesimal character of $\fkg$. If $\tau$ is an element in $\CT_{\mu}$ such that $\mathrm{wt}(\tau)$ is maximal in $\mathrm{wt}(\CT_{\mu})$, then $\CV(\tau)$ is irreducible.
\end{lemma}
\begin{proof}
    Consider the short exact sequence 
    \[
     0\lra \omega \lra \CV(\tau)\lra \CL(\tau)\lra 0.
    \]
    Suppose $\omega$ is non-zero, then it will have an irreducible subquotient by the argument in Lemma~\ref{finite length lem}. Suppose it is of the form $\CL(\tau')$. Then we have $\mathrm{wt}(\tau')>\mathrm{wt}(\tau)$, which contradicts the assumption that $\mathrm{wt}(\tau)$ is maximal in $\mathrm{wt}(\CT_{\mu})$.
\end{proof}

We remark that category $\CC(\fkg,L)_f$ is related to the Casselman-Wallach representations of $G$ by the Casselman-Jacquet functor, and has a better algebraic structure with respect to $\fku$-actions. The following proposition, which is our primary concern, is a good illustration. 
\begin{proposition}\label{homo-CW}
    For every object $V\in \CC(\fkg,L)_f$, $\rmh_i(\fku,V)$ is a Casselman-Wallach representation of $L$.
\end{proposition}
\begin{proof}
    \textbf{Step 1:} Hausdorffness of $\rmh_i(\fku,V)$. Define the   finite set:
    \[
    S_k:=\{ \gamma+\kappa\mid \gamma\in \min \mathrm{wt}(V),\kappa\in \mathrm{wt}(\U(\fku)^{<k})\}.
    \]
    Then we have a decreasing filtration of $V$ as $\fkp$-module 
    \[
    F^k(V):=\bigoplus_{\alpha\notin S_k} V_{\alpha}.
    \]
    Such a filtration satisfies the condition of Lemma~\ref{ext_Hausd}. Thus,   $\rmh_i(\fku,V)$ is Hausdorff.
    
    \textbf{Step 2:} Finite-length of $\rmh_i(\fku,V)$ for every $i$. Since $\CC(\fkg,L)_f$ consists of finite-length objects,  it suffices to consider irreducible $V$.  Then $V\simeq \CL(\tau)$ for some irreducible Casselman-Wallach $L$-representation $\tau$. Consider the short exact sequence and its associated long exact sequence of homology:
    \[
     0\lra \iota \lra \CV(\tau)\lra \CL(\tau)\lra 0.
    \]
     To show that $\rmh_i(\fku,V)$ is finite length for every $i$, it suffices to show $\rmh_i(\fku,\iota)$ and $\rmh_i(\fku,\CV(\tau))$ have  finite length for every $i$. Note that
    \[
    \min \mathrm{wt}(\iota)>\mathrm{wt}(\tau).
    \]
    Applying similar argument as above to the composition factors of $\iota$, and by Lemma~\ref{max irr}, after finite steps, we reduce to proving that the homology of the formal Verma module has finite length.
    
    For every Casselman-Wallach representation $\tau$,  
    \[
    \CV(\tau)\simeq \U[[\fku]]\widehat{\otimes}\tau 
    \]
    as $\U(\fku)$-module by left multiplication on $\U[[\fku]]$, where $\U[[\fku]]:= \mathop{\varprojlim}\limits_{k\geq 0} \U(\fku)/\left(\fku^k\U(\fku)\right)$ is a completion of $\U(\fku)$. Since $\U(\fku)$ is Noetherian, $\U[[\fku]]$ is flat over $\U(\fku)$( for $\fku$-abelian case, see \cite[Corollary 23.1]{Mat80}; the proof also applies to general case). Hence, the Koszul complex for $\U[[\fku]]$
    \[\dots \to \wedge^{i+1} \mathfrak{u}\otimes \U[[\fku]]\to \wedge^{i} \mathfrak{u}\otimes \U[[\fku]]\to \dots \]
    is exact at $i>0$. This implies that the Koszul complex for $\U[[\fku]]\widehat{\otimes}\tau$
    \[\dots \to \wedge^{i+1} \mathfrak{u}\otimes \U[[\fku]]\widehat{\otimes}\tau\to \wedge^{i} \mathfrak{u}\otimes \U[[\fku]]\widehat{\otimes}\tau\to \dots \]
    is also exact at $i>0$. Therefore, one has
    \[\rmh_i(\mathfrak{u},\CV(\tau))\simeq \begin{cases} \tau , & i=0;\\ 0, & i>0.\end{cases}
    \]
   
\end{proof}

Let $\sigma$ be a $(\fkg,L)$-module. Our prototype of $\sigma$ is the subquotient occurring in the filtration of principal series of $G$, determined by a $P$-orbit in the flag variety. The key ingredient for transferring $\sigma$ into the category $\CC(\fkg,L)$ is the Casselman-Jacquet functor.
\begin{definition}
    The Casselman-Jacquet functor $\widehat{\CJ}_\fku$ sends $(\fkg,L)$-modules to $(\fkg,L)$-modules, defined by:
    \[
    \widehat{\CJ}_\fku(\sigma):=\mathop{\varprojlim}\limits_{k} \sigma/\overline{\fku^k\sigma}.
    \]
\end{definition}
An alternative interpretation of the Casselman-Jacquet functor is as follows.  Let $(\sigma')^{\fku}$ be the space of $\fku$-finite continuous linear functionals on $\sigma$. Then 
\[
 (\sigma')^{\fku}=\mathop{\varinjlim}\limits_{k}\left( \sigma/\overline{\fku^k\sigma} \right)'.
\]
Hence, we equip $ (\sigma')^{\fku}$ with the direct limit topology. If $\sigma$ is nuclear, as is our prototype, then 
\[
 \CJF(\sigma)\simeq \Hom_{cts}( (\sigma')^{\fku},\BC)
\]
since the nuclear Fr\'echet space is reflexive (see \cite[Appendix A]{CHM}). This implies (1) of the following lemma.
\begin{lemma}
    Let $\sigma$ be a nuclear $(\fkg,L)$-module, which has the infinitesimal character $\chi_{\lambda}$. Then
    \begin{itemize}
    \item[(1)] $\CJF(\sigma)$ has the same infinitesimal character $\chi_{\lambda}$.
     \item[(2)] Assume that 
\begin{equation}\label{CW_assump}\forall k\in \mathbb{Z}_{>0}, \ \sigma/\fku^k\sigma\ \text{is a Casselman-Wallach representation of}\ L.\end{equation} 
Then $\CJF(\sigma)$ is in the category $\CC(\fkg,L)_f$.
    \end{itemize} 
\end{lemma}
\begin{proof} 
It remains to show (2). By the  definition of $\CC(\fkg,L)$, it suffices to verify condition (i) in Definition \ref{def of cate} for $\CJF(\sigma)$. By the following surjective $L$-morphism
\[
 \fku^k\otimes \sigma/\fku\sigma \lra \fku^k\sigma/\fku^{k+1}\sigma,
\]
we have 
\[\min \mathrm{wt}(\CJF(\sigma))=\min  \mathrm{wt}(\sigma/\fku\sigma).\] Thus, the $\ov{\fku}$-action is locally finite on $\CJF(\sigma)^{\fkz_L\text{finite}}$. 
\end{proof}
Under the assumption \ref{CW_assump}, one has the exact sequence
\[0\to \mathop{\bigcap}_k\fku^k\sigma\to \sigma\to \varprojlim_k\sigma/\fku^k\sigma\to 0,\] 
which follows from the following proposition.
\begin{proposition}\label{surj_CJ}
Under the assumption \ref{CW_assump}, the natural map $\sigma\to \varprojlim_k \sigma/\mathfrak{u}^k\sigma$ is surjective.
\end{proposition}

We require a technical lemma. If $p$ is a semi-norm on a Fr\'echet space $V$, and $W$ is a closed subspace of $V$, then the induced semi-norm on $V/W$ is defined as 
\[
\ov{p}(\ov{v}):=\inf_{w\in W} p(v+w)\text{ for }v\in V,
\]
where $\ov{v}$ is the image of $v$ in $V/W$.
\begin{lemma}\label{zero-norm}
Under the assumption \ref{CW_assump}, let $p$ be a continuous semi-norm on $\sigma$. Then, for sufficiently large $r$, and for $k>r$, the induced semi-norm of $\ov{p}$ on $\sigma/\mathfrak{u}^k\sigma$ is identically zero on $\mathfrak{u}^r\sigma/\mathfrak{u}^k\sigma$. 
\end{lemma}

\begin{proof}
Since $\sigma$ is an $L$-representation of moderate growth, there exist a semi-norm $q$ and some Nash function $f$ on $L$ such that 
\[p(g\cdot v)\leq f(g)q(v),\ \forall g\in L, \forall v\in \sigma.\]

Fix $a\in \fkz_L$ with $\alpha(a)>0$ for all $\alpha\in \Omega$.  
As $\sigma/\mathfrak{u}\sigma$ is a Casselman-Wallach representation of $L$, there are finitely many generalized weights of $a$ on $\sigma/ \mathfrak{u}\sigma$, denoted by \[\gamma_1(a),\dots,\gamma_s(a).\]  Thus, the generalized weights of $a$ on $\mathfrak{u}^r\sigma/\mathfrak{u}^k\sigma$ take the form
\[(\gamma_i+\sum_{\alpha\in \Omega} m_{\alpha}\alpha)(a)\] where $m_{\alpha}\in \mathbb{Z}_{\geq 0}$ and $r\leq \sum m_{\alpha}\leq k$.

Suppose that the induced semi-norm $\overline{p}$ on $\mathfrak{u}^r\sigma/\mathfrak{u}^k\sigma$ is non-zero, then there exists some $v\in \mathfrak{u}^r\sigma$ with $\overline{p}(\overline{v})>0$.  

Moreover, we can assume that $\overline{v}$ is a generalized eigenvector of $a$ with eigenvalue $\gamma(a)$. Note that for $v_1,v_2\in \sigma$, if $\overline{p}(\overline{v_1})=0$, then $\overline{p}(\overline{v_1}+\overline{v_2})=\overline{p}(\overline{v_2})$. Consider the finite-dimensional space generated by $\left\{\overline{\sigma(a)^lv}\ |\ l=0,1,\dots\right\}$, by choosing $u$ in this space properly, one has 
\[\overline{p}(\exp(ta)\cdot \overline{u})=e^{\gamma(ta)}\overline{p}(\overline{u})\neq 0,t\in\BR.\]

By the definition of Nash function, for $r$ sufficiently large, one has 
\[e^{\gamma(ta)}\geq C\cdot f(\exp(ta))\] for every constant $C$ when $t\to +\infty$. It contradicts the moderate growth condition $\overline{p}(\exp(ta)\cdot \overline{u})\leq f(\exp(ta)) q(u)$.
\end{proof}

Let us go back to prove Proposition \ref{surj_CJ}. 

\begin{proof}[Proof of Proposition \ref{surj_CJ}]
Take an arbitrary element in $\varprojlim_i \sigma/\mathfrak{u}^i\sigma$, i.e. a sequence 
\[\{v_i\in \sigma\}_{i\in \mathbb{N}},\  \text{with}\ v_j-v_k\in \mathfrak{u}^k\sigma, \ \forall j>k.\] 
To prove the statement, it suffices to find a sequence $\{v_i'\}_{i\in \mathbb{N}}$ in $\sigma$ such that $v_i-v_i'\in \mathfrak{u}^i\sigma$, and $\{v_i'\}$ converges in $\sigma$.

Let  $\{p_i\}_{i\in \mathbb{N}}$ be the countable family of semi-norms which define the topology of $\sigma$. By taking $\sum_{j\leq i}p_j$, we can assume that 
\[p_0\leq p_1\leq \dots.\] 

For $p_i$, by Lemma \ref{zero-norm}, there exists $r_i$ such that the induced semi-norm of $p_i$ on $\mathfrak{u}^r\sigma/\mathfrak{u}^k\sigma$ is zero for all $k>r>r_i$. We can take the $r_i$'s such that $1<r_0<r_1<\dots$. 

Define $\{\widetilde{v}_i\}_{i\in \mathbb{N}}$ as $\widetilde{v}_i=v_{r_i+1}$. Since $r_i>i$, $v_i-\widetilde{v}_i\in  \mathfrak{u}^i\sigma$. For the semi-norm $p_j$, when $l>j$, 
\[\overline{p_j}(\widetilde{v}_l-\widetilde{v}_j)=0\ \text{on} \  \mathfrak{u}^{r_j+1}\sigma/ \mathfrak{u}^k\sigma,\ \forall k>r_j+1.\] 

Set $v_1'=\widetilde{v}_1$. Then select 
\[v_2'\in \widetilde{v}_2+ \mathfrak{u}^{r_2+1}\sigma\] 
such that $p_1(v_2'-\widetilde{v}_{1})<\frac{1}{2^1}$. Similarly,  one can select $v_3'\in \widetilde{v}_3+ \mathfrak{u}^{r_3+1}\sigma$ such that $p_2(v_3'-v_2')<\frac{1}{2^2}$.

By such a procedure, one gets $\{v_i'\}_{i\in \mathbb{N}}$ such that $p_{i-1}(v_i'-v_{i-1}')<\frac{1}{2^{i-1}}$.  This is a convergent sequence such that $v_i'-v_i\in \mathfrak{u}^i\sigma$. 

\end{proof}

\begin{remark}\label{CW-ass-weaker}
The assumption~\eqref{CW_assump} can be weakened: it suffices to assume that $\sigma/\fku\sigma$ is a Casselman-Wallach representation of $L$. Indeed, one can show inductively that $\sigma/\fku^k\sigma$ is a Casselman-Wallach representation by using the short exact sequence
\[
0 \lra \fku^{k-1}\sigma \lra \sigma \lra \sigma/\fku^{k-1}\sigma \lra 0
\]
together with Lemma~\ref{exa-hau-2} and Remark~\ref{closed-im-rem}.
\end{remark}

\subsection{Schwartz functions on Nash manifolds}\label{sec-Schwartz fun}
Various Schwartz functions are natural objects in the category of smooth representations. This subsection reviews some facts about Schwartz functions, which will be used freely in the article.

We first recall the definition of \textbf{co-sheaf} on Nash manifold. Let $X$ be a Nash manifold and $\sF$ be a pre-co-sheaf on $X$(for detailed definition, see \cite[Appendix A.4]{AG}. We regard every section of $\mathscr{F}$ on an open subset $\mathscr{U}$ also as a section on every open subset containing $\mathscr{U}$ via the extension map. Then $\sF$ is called a \textbf{co-sheaf} if for every finite open covering $\{\sU_i\}_{1\leq i\leq n}$ of $X$, the following sequence is exact
\begin{equation}\label{co-sheaf exa}
    \bigoplus_{1\leq i<j\leq n}\sF(\sU_i\cap \sU_j)\lra \bigoplus_{1\leq i\leq n}\sF(\sU_i)\lra \sF(X)\lra 0.
\end{equation}
Here, the first map is given by  
\[
(s_{i,j})_{i<j}\mapsto (\sum_{l<k}s_{l,k}-\sum_{k<l}s_{k,l})_{1\leq k\leq n},\ \text{for}\ s_{i,j}\in \sF(\sU_i\cap \sU_j), 
\]
and the second map is given by
\[
(s_i)_{1\leq i\leq n}\mapsto \sum_{1\leq i\leq n}s_i, \ \text{for}\ s_i\in \sF(\sU_i). 
\]
Then Schwartz functions form a co-sheaf.
\begin{proposition}\label{co-sheaf prop}
    Let $X$ be a Nash manifold and $Z\subseteq X$ be a closed submanifold. Let $\CE$ be a tempered bundle over $X$. Then
    \begin{enumerate}
        \item The pre-co-sheaf $\CS(\cdot,\CE):\sU\mapsto \CS(\sU,\CE)$, where $\sU$ is an open subset of $X$, is a co-sheaf.
        \item The pre-co-sheaf 
        \[
        \CS_Z(\cdot,\CE):\sU\mapsto \CS_{\sU\cap Z}(\sU,\CE):=\CS( \sU,\CE)/\CS(\sU\backslash Z,\CE),
        \]
        where $\sU$ is an open subset of $X$, is a co-sheaf.
    \end{enumerate}
\end{proposition}
\begin{proof}
    \begin{enumerate}
        \item The proof is similar as \cite[Proposition 5.1.3]{AG}.
        \item By (1), the second map in \eqref{co-sheaf exa} is surjective. Moreover, it is easy to check the sequence is a complex. Let
        \[(s_i)_{1\leq i\leq n}\in\bigoplus_{1\leq i\leq n}\CS_{\sU_i\cap Z}(\sU_i,\CE)
        \text{ such that } \sum_{1\leq i\leq n} s_i=0.
       \]
       Take $\widetilde{s_i}$ as a lift of $s_i$ in $\CS(\sU_i,\CE)$. Then 
       we have $\widetilde{s}:=\sum_{1\leq i\leq n} \widetilde{s_i}\in \CS(X\backslash Z,\CE)$. By \cite[Theorem 4.4.1]{AG}, there exists a partition of unity by tempered functions $(\alpha_i)_{1\leq i\leq n}$ with $
       \Supp(\alpha_i)\subset \sU_i$. Then we have
  \[
       \alpha_i\cdot \widetilde{s}\in \CS(\sU_i\backslash Z,\CE).
  \]
 It follows that $(\widetilde{s_i}-\alpha_i\cdot \widetilde{s})_{1\leq i\leq n}$ is a lift of $(s_i)_{1\leq i\leq n}$, which maps to zero by the second map in \eqref{co-sheaf exa}. The result now follows from the co-sheaf property in (1).

    \end{enumerate}
\end{proof}

Let $H$ be a subgroup of an almost linear Nash group $G$, the (\textbf{normalized}) Schwartz induction $\SInd_H^G(V_{\sigma})$ of $(\sigma,V_{\sigma})\in\Smod_H$ is defined as the Schwartz sections of the tempered bundle $(V_{\sigma}\otimes (\frac{\delta_H}{\delta_G})^{\frac{1}{2}})\times_H G$ (see \cite{CS} for details).
\begin{proposition}[\cite{Fd}, Proposition 2.2.7]
    The Schwartz induction functor $\SInd_H^G:\Smod_H\to \Smod_G$ is exact.
\end{proposition}
The Schwartz induction also satisfies the following \textbf{Mackey isomorphism}.
\begin{proposition}[\cite{CS}, Proposition 7.4]
    Let $V_0\in\Smod_H$, and $V\in\Smod_G$. If either $V_0$ or $V$ is nuclear, then there is an isomorphism of $G$-representations, 
    \[
    \SInd_H^G(V_0)\widehat{\otimes} V\simeq \SInd_H^G(V_0\widehat{\otimes} V|_H).
    \]
\end{proposition}

Now, let $X$ be a Nash manifold equipped with a Nash group action by $G$, and $\CE$ be a tempered bundle over $X$. Suppose the $G$-action on $X$ has a unique open orbit $O$ or a unique closed orbit $Z$, then we use the following simplified notations:
\begin{equation} \label{open_close}
\CS(X,\CE)_o:=\CS(O,\CE); \qquad\CS(X,\CE)_c:=\CS_Z(X,\CE).\end{equation}
In particular, if the $G$-action on $X$ has exactly two orbits, an open orbit $O$ and a closed orbit $Z$, then there is an exact sequence
\[0\to \CS(X,\CE)_o\to \CS(X,\CE) \to \CS(X,\CE)_c\to 0. \]

We review the smooth Mackey induction theory developed by Du Cloux in \cite{Fd}, which studies the smooth representation theory of non-reductive groups. It is fundamental for the whole Bernstein-Zelevinsky theory that we pursue. 

First of all, \cite[Theorem 5.1]{Fd0} shows that there is a category equivalence between $\Smod_N$ and unitary representations by taking smooth vectors when $N$ is a connected and simply connected nilpotent Lie group. Hence, by Kirillov's orbit method, the irreducible smooth  representations of $N$ are one-one corresponding to the coadjoint orbits on $\Lie(N)^*$. Let $G$ be an almost linear Nash group with a nilpotent normal subgroup $N$. Then given $\nu\in\Lie(N)^*$,  the stabilizer in $G$ of the coadjoint orbit of $\nu$ is equal to $G^{\nu}\cdot N$, denoted by $H^{\nu}$. Let $V^{\nu}$ be the irreducible smooth representation of $N$ corresponding to $\nu$. Let $\CY(\nu)$ be the full subcategory in $\Smod_{H^{\nu}}$ consisting of representations $\sigma$ such that $\sigma|_{N}\simeq V^{\nu}\widehat{\otimes} E$ for some Fr\'echet space $E$. The following theorem is a direct consequence of~\cite[Corollary 4.3.3]{Fd}.
\begin{theorem}\label{Mackey theory}
Let $\sigma\in\CY(\nu)$. If $\pi$ is a subrepresentation of $\SInd_{H^{\nu}}^G(\sigma)$, then there is a subrepresentation $\sigma_0$ of $\sigma$, such that $\pi\simeq \SInd_{H^{\nu}}^G(\sigma_0)$. In particular, when $\sigma$ is irreducible, $\SInd_{H^{\nu}}^G(\sigma)$ is also irreducible.

Similarly, if $\sigma$ is an algebraically irreducible representation in $\CY(\nu)$, then $\SInd_{H^{\nu}}^G(\sigma)$ is also algebraically irreducible.
\end{theorem}
In this article, we shall only require the case where $N$ is abelian. However, the general theory becomes necessary when studying symplectic or unitary groups. 

We recall the theory of abstract Fr\'echet algebra, which will be applied to the rearrangement of coarse spectral filtration. Recall that given two Fréchet spaces $V_1$ and $V_2$,  $V_1\otimes_{\pi} V_2$ is defined as the underlying space $V_1\otimes  V_2$ equipped with strongest Fréchet topology such that the natural map $V_1\times V_2\to V_1\otimes  V_2$ is continuous. The $V_1\widehat{\otimes}_{\pi} V_2$ is defined as the completion of  $V_1\otimes_{\pi} V_2$.
\begin{definition}
    \begin{enumerate}
    \item We call a pair $(\CA, \alpha: \CA \widehat{\otimes}_{\pi} \CA \to \CA)$ a \textbf{Fréchet algebra} if $\CA$ is a Fréchet space and $\alpha$ is a continuous map satisfying the algebra axioms. We will always omit $\alpha$ in the notation.
    \item Let $\CA$ be a Fr\'echet algebra. We call a pair $(M, \alpha: \CA \widehat{\otimes}_{\pi} M \to M)$ a \textbf{Fréchet module} if $M$ is a Fréchet space and $\alpha$ is a continuous map satisfying the module axioms. The image of $\alpha$ equipped with quotient topology is denoted by $\CA\cdot M$.
    \item Let $(M,\alpha)$ be a Fr\'echet module of $\CA$. We call it \textbf{differentiable} if $\alpha$ is surjective and $\{m\in M\mid \CA\cdot m=0\}=\emptyset$. We call a Fr\'echet algebra differentiable if it is differentiable as a left module over itself.
    \item Let $(M,\alpha)$ be a Fr\'echet module of $\CA$. We call it \textbf{factorizable} if $\alpha$ restricting to $\CA\otimes M$ is surjective and $\{m\in M\mid \CA\cdot m=0\}=\emptyset$. We call a Fr\'echet algebra \textbf{factorizable} if it is differentiable and all differentiable modules are factorizable.
    \item We call a Fr\'echet algebra \textbf{hereditary} if any closed submodule of its differentiable module is differentiable.
     \item Let $\CA$ be a Fréchet algebra. A sequence $\{e_n\}_{n=1}^{\infty}$ in $\CA$ is called an \textbf{approximation to the identity} if, for every $a \in \CA$, the sequences $\{e_n \cdot a\}$ and $\{a \cdot e_n\}$ converge to $a$.
    \end{enumerate}
\end{definition}
The following are our prototypes for Fréchet algebras and Fréchet modules.
\begin{example}
 Let $G$ be a Nash group with a fixed Haar measure $dg$. Then $\CS(G)$ is a Fr\'echet algebra under convolution product. It is a standard fact that there is a category equivalence between $\Smod_G$ and the category of differentiable $\CS(G)$-module. Hence, by Diximier-Malliavin theorem, $\CS(G)$ is a hereditary and factorizable Fr\'echet algebra.
\end{example}

\begin{example}
    Let $X$ be a Nash manifold. Then $\CS(X)$ is a Fr\'echet algebra under point-wise multiplication. Let $\{K_n\}_{n=1}^{\infty}$ be an increasing sequence of open and relatively compact subset. Then $\{f_n\}_{n=1}^{\infty}\subset \CC_c^{\infty}(X)$, such that 
    \[
    \supp(f_n)\subset K_{n+1}  \text{ and } f_n(x)=1 \text{ for } x\in K_n,
    \]
     is an approximation to the identity. We can find a finite Nash open covering $\{U_i\}$ such that each $U_i$ is Nash isomorphic to a closed Nash sub-manifold of $\BR^d$ for some $d$. Then each $\CS(U_i)$ is quotient algebra of $\CS(\BR^d)$, which is hereditary and factorizable by Fourier transform. Consequently, $\CS(X)$ is also hereditary and factorizable by partition of unity.
\end{example}

From now on, we assume that all the Fr\'echet algebra is \textbf{nuclear}. We now present a lemma regarding the inverse limit of Fréchet modules.
\begin{lemma}\label{limit diff}
    Let $\CA$ be a hereditary and factorizable Fr\'echet algebra. Let $\{M_k,d_k: M_{k+1}\to M_k\}_{k=0}^{\infty}$ be an inverse system of Fr\'echet $\CA$-module such that $M_0=0$ and $d_k$ is surjective for any integer $k$.
    \begin{enumerate}
        \item Suppose that each $\ker d_k$ is a differentiable $\CA$-module, then $\varprojlim_k M_k$ is a differentiable $\CA$-module as well.
        \item Suppose that $\CA\cdot \Ker d_k=0$, then
        \[
        \CA\cdot \varprojlim_k M_k=0
        \]
    \end{enumerate}
\end{lemma}
\begin{proof}
    To prove the first statement, we first claim that each $M_k$ is differentiable. By induction, it suffices to show that for a short exact sequence of Fr\'echet $\CA$-modules
    \[
    0\lra N_1\lra N\stackrel{\varphi}{\lra} N_2\lra 0,
    \]
    if $N_1$ and $N_2$ are differentiable, then so is $N$. Let $n\in N$. Then $\varphi(n)=\sum_{i=1}^l a_in_i$ for some $a_i\in \CA$ and $n_i\in N_2$. Let $\widetilde{n_i}$ be a pre-image of $n_i$ under $\varphi$. Then 
    \[
    n-\sum_{i=1}^l a_i \widetilde{n_i}\in N_1.
    \]
    Hence, the claim follows from the fact that $N_1$ is differentiable. 

    To prove that $\varprojlim_k M_k$ is differentiable, we first prove that the inverse limit of the action map 
    \[
    \varprojlim_k \CA\otimes M_k\lra \varprojlim_k M_k
    \]
    is surjective. It suffices to show that in the following commutative diagram
    \begin{center}
        \begin{tikzcd}
            \CA\otimes M_k \ar[d,"\alpha_k"] & \CA\otimes M_{k+1} \ar[d,"\alpha_{k+1}"]\ar[l,"\id\otimes d_k"]&\\
            M_k & M_{k+1}\ar[l,"d_k"],
        \end{tikzcd}
    \end{center}
    if $\alpha_k(t)=d_{k}(m)$ for some $t\in  \CA\otimes M_k$ and $m\in M_{k+1}$, then there exists $\widetilde{t}\in \CA\otimes M_{k+1}$ such that $\id\otimes d_k(\widetilde{t})=t$ and $\alpha_{k+1}(\widetilde{t})=m$. Equivalently, it suffices to show $m\in \alpha_{k+1}\left((\id\otimes d_k) ^{-1}(t)\right)$. By definition,
    \[
    (\id\otimes d_k) ^{-1}(t)\supset  t_0 +\CA\otimes \Ker d_k,
    \]
    where $t_0$ is a pre-image of $t$ under $\id\otimes d_k$. Then it follows from the fact that $\Ker d_k $ is differentiable since $\CA$ is hereditary.
    
    Note that there is a natural injective map
    \[
    \varprojlim_k \CA\otimes M_k\lra \CA\widehat{\otimes} \varprojlim_k M_k.
    \]
    Consequently, $\varprojlim_k M_k$ is differentiable.

    To prove the second statement, it suffices to show that for a short exact sequence of Fr\'echet $\CA$-modules
    \[
    0\lra N_1\lra N\lra N_2\lra 0,
    \]
    if $\CA\cdot N_1=0$ and $\CA\cdot N_2=0$, then $\CA\cdot N=0$ as well. It follows from the fact that
    \[
  \CA\cdot N  =(\CA\cdot \CA)\cdot N.
    \]   
\end{proof}

The following lemma is a splitting result of Fr\'echet $\CA$-module.
\begin{lemma}\label{split}
    Let $\CA$ be a differentiable Fr\'echet algebra possessing an approximation to the identity. Let 
   \[
    0\lra M_1\lra M\stackrel{\varphi}{\lra} M_2\lra 0
    \]
     be a short exact sequence of Fr\'echet $\CA$-module. Suppose that $\CA\cdot M_1=0$ and $M_2$ is differentiable. Then there is a splitting of $\varphi$. 
\end{lemma}
\begin{proof}
   By the approximation to the identity, we know that $\CA\cdot M$ is differentiable. Hence, $\CA\cdot M\cap M_1=0$, which implies that the natural continuous map
   \[
   M_1\bigoplus \CA\cdot M\lra M
   \]
   is bijective. Consequently, it is a topological isomorphism by the open mapping theorem of Fr\'echet spaces.
\end{proof}

\subsection{Schwartz homology and Euler-Poincar\'e characteristic}
Let $G$ be an almost linear Nash group. The category $\Smod_G$ admits a homology theory known as Schwartz homology. That is, for $\pi\in\Smod_G$, we take a strong projective resolution $P_{\bullet}$, the homology $\rmh^{\CS}_i(G,\pi)$ is defined as the $i$-th homology of the complex equipped with the subquotient topology
\[
   \dots \lra (P_i)_G\lra (P_{i-1})_G\lra \dots .
\]
Here each $(P_i)_G$ is Fr\'echet, see \cite[Theorem 5.9]{CS}. For another representation $\tau\in \Smod_G$, we define the extension group $\Ext^i_G(\pi,\tau)$ as the $i$-th cohomology group of the complex
\[
  \dots \lra \Hom_G(P_{i-1},\tau)\lra \Hom_G(P_{i},\tau)\lra \dots .
\]
If $\tau$ is the trivial representation, then we equip the cohomology with the subquotient topology of the strong dual topology. As a locally convex topological vector space, it does not depend on the choice of strong projective resolution by the comparison theorem (see \cite[2.2.6]{Wei}). Note that if $\rmh_i^{\CS}(G,\pi)$ is Hausdorff and $\pi$ is nuclear, then we have
\[
\rmh_i^{\CS}(G,\pi)'\simeq \Ext^i_G(\pi,\BC),
\]
see \cite[Proposition 5.3.2]{AGS15b}. 

From now on, in this subsection, unless otherwise specified, we make the further assumption that  $G$ is reductive, 
\begin{equation}\label{cond_CW_nc}\tau\ \textbf{is Casselman-Wallach} \quad \text{and}\quad  \pi \ \textbf{is nuclear}.\end{equation}
Under this assumption, $\pi\widehat{\otimes}\tau^{\vee}$ is also nuclear by \cite[50.9]{Tr}. We say that $\pi$ satisfies the homological finiteness condition with respect to $\tau$ if  $\Ext^i_G(\pi\widehat{\otimes}\tau^{\vee},\BC)$ is a finite-dimensional vector space for every integer $i$. By \cite[Lemma A.1]{CHM}, this implies $\Ext^i_G(\pi\widehat{\otimes}\tau^{\vee},\BC)$ is Hausdorff, hence,
\[ \rmh_i^{\CS}(G,\pi\widehat{\otimes}\tau^{\vee})\simeq \Ext^i_G(\pi\widehat{\otimes}\tau^{\vee},\BC)'
\]
is Hausdorff as well. Note that by Koszul type resolution (see \cite{CS}, 7.3), $\Ext^i_G(\pi\widehat{\otimes}\tau^{\vee},\BC)$ is vanishing for sufficiently large  $i$. We define the \textbf{Euler-Poincar\'e characteristic} of $(\pi,\tau)$ as 
\begin{equation*}
    \EP_G(\pi,\tau):=\sum_i  (-1)^i\dim \Ext^i_G(\pi\widehat{\otimes}\tau^{\vee},\BC).
\end{equation*}
The homological finiteness condition and the vanishing property ensure the  sum above is finite.
\begin{remark}\label{iso-lr}
    Since $\Hom_G(-,\tau)$ is left exact in the category $\Smod_G$, we have 
    \[
        \Ext^0_G(\pi \widehat{\otimes} \tau^{\vee}, \BC) \simeq \Hom_G(\pi \widehat{\otimes} \tau^{\vee}, \BC) \simeq \Hom_G(\pi, \tau),
    \]
 where the second isomorphism comes from the fact that $\tau\simeq (\tau^{\vee})^{\vee}$ coincides with the image of the action map:
    \[
    \CS(G)\widehat{\otimes} (\tau^{\vee})'\lra  (\tau^{\vee})'.
    \]
    This implies the general isomorphism
    \[
        \Ext_G^i(\pi, \tau) \simeq \Ext^i_G(\pi \widehat{\otimes} \tau^{\vee}, \BC)
    \]
    as follows. Let $P_{\bullet}$ be a strong projective resolution of $\pi$. By \cite[Proposition 5.5]{CS}, $P_{\bullet} \widehat{\otimes} \tau^{\vee}$ then forms a strong projective resolution of $\pi \widehat{\otimes} \tau^{\vee}$. The isomorphism consequently follows from
    \[
        \Hom_G(P_i \widehat{\otimes} \tau^{\vee}, \BC) \simeq \Hom_G(P_i, \tau), \quad \forall i \in \BZ.
    \]
    However, unlike the $p$-adic case, Schwartz induction lacks a right adjoint functor. Consequently, defining the Euler-Poincar\'e characteristic in the above form provides greater  flexibility for calculation.
\end{remark}

We need the following result comparing Lie algebra homology and Schwartz homology. Let $K$ be the complexification of a maximal compact subgroup of an almost linear Nash group $G$.
\begin{proposition}[see \cite{CS}, Theorem 7.7]
    Let $\pi\in\Smod_G$. Then there is an isomorphism as topological vector space
    \begin{equation*}
        \rmh_i(\fkg,K;\pi)\simeq \rmh_i^{\CS}(G,\pi).
    \end{equation*}
\end{proposition}
As pointed out by Dipendra Prasad, the Euler-Poincar\'e characteristic is a more natural and flexible invariant than $\dim \Hom_G(\pi,\tau)$ from some points of view. Similar to the $p$-adic case, it has the following basic properties.
\begin{proposition}\label{EP-prop}
    Let $G$ be a reductive almost linear Nash group, and let $\pi,\tau\in\Smod_G$ satisfying condition \eqref{cond_CW_nc}. Then:
    \begin{enumerate}
        \item If  
\[
  0\lra \pi_1\lra \pi\lra \pi_2\lra 0
\]
is an exact sequence in $\Smod$ and $\pi_j$ satisfies the homological finiteness condition with respect to $\tau$ for $j\in\{1,2\}$, then $\pi$ also satisfies the homological finiteness condition with respect to $\tau$ and
\[
\EP_G(\pi,\tau)=\EP_G(\pi_1,\tau)+\EP_G(\pi_2,\tau)
\]
\item The analogous property holds as (1) for variable $\tau$.
\item Assume moreover $\pi$ is Casselman-Wallach, then 
\[
 \Ext_G^i(\pi\widehat{\otimes} \tau^{\vee},\BC)\simeq \Ext_{\fkg,K}^i(\pi^K \otimes (\tau^K)^{\vee},\BC)\simeq\Ext^i_{\fkg,K}(\pi^{K},\tau^{K}).
\]
In particular, $\Ext_G^i(\pi\widehat{\otimes} \tau^{\vee},\BC)$ is finite-dimensional.
\item If $G$ has a non-compact center $Z_G$ and $\pi$ is Casselman-Wallach, then 
\[
\EP_G(\pi,\tau)=0.
\]
    \end{enumerate}
\end{proposition}
\begin{proof}
\noindent {For (1),(2):}    Since the proofs of (1) and (2) have no difference, we only prove (1) for simplicity. By homological finiteness
   \[
   \rmh_i^{\CS}(G,\beta\widehat{\otimes}\tau^{\vee})'\simeq \Ext^i_G(\beta\widehat{\otimes}\tau^{\vee},\BC) \quad \text{for every integer }i \text{ and }\beta=\pi_1,\pi_2 .
   \]
   Consequently, the result follows from the long exact sequence of Schwartz homology associated to the short exact sequence:
   \[
   0\lra \pi_1\widehat{\otimes}\tau^{\vee} \lra \pi\widehat{\otimes}\tau^{\vee}\lra \pi_2\widehat{\otimes}\tau^{\vee}\lra 0.
   \]
\noindent {For (3):} 
   First isomorphism in (3): Consider the map 
   \begin{equation}\label{comp-eq}
        \varsigma^i:\rmh_i(\fkg,K;\pi^K \otimes (\tau^K)^{\vee}) \lra \rmh_i(\fkg,K;\pi\widehat{\otimes} \tau^{\vee}) \simeq \rmh_i^{\CS}(G,\pi\widehat{\otimes} \tau^{\vee}),
   \end{equation}
   we claim this map is an isomorphism for every $i$. This is developed through following two steps.
   
   \noindent {\textbf{Step 1. Reduction to principal series.}} Assume that when $\tau^{\vee}$ is a principal series, $\varsigma^{i}$ is an isomorphism for each $i$. Let us prove the general case. Under this assumption, when $\tau^{\vee}$ is a generalized principal series, $\varsigma^i$ is also an isomorphism for each $i$. 
   
   For general $\tau^{\vee}$, by the Casselman embedding theorem, there exists a short exact sequence
   \[
   0\lra \tau^{\vee}\lra I\lra J\lra 0,
   \]
   where $I$ is a generalized principal series. Consider the commutative diagram of associated long exact sequence
   \[
   \begin{tikzcd}
       \rmh_{i+1}(\fkg,K;\pi^K \otimes J^K)\ar[r]\ar[d,"\varsigma_2^{i+1}"] & \rmh_{i}(\fkg,K;\pi^K \otimes (\tau^K)^{\vee}) \ar[r]\ar[d,"\varsigma^{i}"]& \rmh_{i}(\fkg,K;\pi^K \otimes I^K)\ar[d,"\varsigma_1^{i}"]\\
       \rmh_{i+1}(\fkg,K;\pi\widehat{\otimes} J)\ar[r] & \rmh_{i}(\fkg,K;\pi\widehat{\otimes} \tau^{\vee})\ar[r]& \rmh_{i}(\fkg,K;\pi\widehat{\otimes}I)
   \end{tikzcd}
   \]
   We argue by induction on $i$. When $i$ is large enough, by homology vanishing, $\varsigma^i$ is isomorphic. Assume that $\varsigma^i$ is isomorphic for every Casselman-Wallach representation $\pi,\tau$ when $i\geq k$. For $i=k-1$, by the above commutative diagram, $\varsigma^{i}$ is isomorphic since $\varsigma_2^{i+1}, \varsigma_1^i$ are  both isomorphisms. 
   
   \noindent {\textbf{Step 2. Proof for principal series.}} Let $\tau^{\vee}=\Ind_{P^0}^G(\beta)$, where $\beta$ is an irreducible finite-dimensional representation of $L^0$. Then we have 
   \[
  ( \tau^{\vee})^K \simeq P^{\fkg,K}_{\fkp^0,K^0}(\beta\otimes \delta_{P^0}^{-1/2})
   \]
   by the easy duality theorem \cite[Theorem 3.1]{KV} and the infinitesimal isomorphism theorem \cite[Proposition 11.47]{KV}. By Mackey isomorphism and Shapiro's lemma of both Schwartz homology and $(\fkg,K)$-homology, it suffices to show
   \[
    \varsigma_i:\rmh_i(\fkp^0,K^0;\pi^K\otimes\beta)\lra \rmh_i^{\CS}(P^0,\pi\widehat{\otimes}\beta) 
   \]
   is an isomorphism. We observe two homologies have spectral sequences corresponding through $\varsigma_i$ with following $E^{p,q}_2$-term
   \[
   \varsigma_i: \rmh_p(\fkl^0,K^0,\beta\otimes\rmh_q(\fku^0,\pi^K))\lra\rmh^{\CS}_p(L^0,\beta\widehat{\otimes}\rmh^{\CS}_q(U^0,\pi) ).
   \]
   Consequently, by comparison theorem (see for example \cite[Theorem 5.2]{LLY21}), $\varsigma_i$ is isomorphic at each $E_2^{p,q}$-term, hence, also isomorphic for the map~\eqref{comp-eq}. In particular, homologies in the map~\eqref{comp-eq} are finite dimensional. Therefore, we have
   \[
   \Ext_G^i(\pi\widehat{\otimes} \tau^{\vee},\BC)\simeq \rmh_i^{\CS}(G,\pi\widehat{\otimes} \tau^{\vee})^*\simeq \rmh_i(\fkg,K;\pi^K \otimes (\tau^K)^{\vee})^*\simeq \Ext_{\fkg,K}^i(\pi^K \otimes (\tau^K)^{\vee},\BC).
   \]
   The proof for the second isomorphism in (3) is similar as the smooth representations (see Remark~\ref{iso-lr}) , thus we omit it.
   
\noindent {For (4):} By the additive property in (1), we assume $\tau$ and $\pi$ are irreducible and $\omega_{\pi}=\omega_{\tau}$. Consider the spectral sequence of Schwartz homology
   \[
   E_2^{p,q}:=\rmh_p^{\CS}(G/Z_G,\rmh_q^{\CS}(Z_G,\pi\widehat{\otimes} \tau^{\vee}))\Rightarrow \rmh_{p+q}^{\CS}(G,\pi\widehat{\otimes} \tau^{\vee}),
   \]
   where 
   \[
   \rmh_p^{\CS}(G/Z_G,\rmh_q^{\CS}(Z_G,\pi\widehat{\otimes} \tau^{\vee}))\simeq \rmh_p^{\CS}(G/Z_G,\pi\widehat{\otimes} (\tau^{\vee}))\otimes\rmh_q^{\CS}(Z_G,\omega_{\pi}\otimes\omega_{\tau}^{-1}).
   \]
   Since $Z_G$ is not compact, we have
   \[
   \sum_{q}(-1)^q\dim \rmh_q^{\CS}(Z_G,\mathrm{triv})=0,
   \]
   which implies
\[
   \EP_G(\pi,\tau)= 
    \sum_{p}(-1)^p \dim \rmh_p^{\CS}(G/Z_G,\pi\widehat{\otimes} (\tau^{\vee}) )\sum_{q}(-1)^q \dim \rmh_q^{\CS}(Z_G,\omega_{\pi}\otimes\omega_{\tau}^{-1})=0.
\]

\end{proof}

We also have following Kunneth formula for extension groups and the  Euler-Poincar\'e characteristic.
\begin{proposition}
    Let $G_1,G_2$ be two reductive groups. Suppose that $E_i\in\Smod_{G_i}$ are nuclear and $F_i$ are Casselman-Wallach representation of $G_i$,  for $i=1,2$. Moreover, assume that $E_i$ satisfies the homological finiteness condition with respect to $F_i$, for $i=1,2$. Then we have an isomorphism 
    \[
    \Ext^i_{G_1\times G_2}((E_1\boxtimes E_2)\widehat{\otimes}(F_1^{\vee}\boxtimes F_2^{\vee}),\BC)\simeq \bigoplus_{p+q=i}\Ext^p_{G_1}(E_1\widehat{\otimes}F_1^{\vee},\BC)\otimes\Ext^q_{G_2}(E_2\widehat{\otimes}F_2^{\vee},\BC).
    \]
    Whence, we have
    \[
    \EP_{G_1\times G_2}(E_1\boxtimes E_2,F_1\boxtimes F_2)=\EP_{G_1}(E_1,F_1)\cdot \EP_{G_2}(E_2,F_2).
    \]
\end{proposition}
\begin{proof}
 By the homological finiteness assumptions, we have 
 \[
 \rmh_p^{\CS}(G_j, E_j\widehat{\otimes}F_j^{\vee})\simeq \Ext^p_{G_j}(E_j\widehat{\otimes}F_j^{\vee},\BC)' 
 \]
 is finite-dimensional for every integer $p$ and $j=1,2$. Hence, by \cite[Theorem A.7]{Geng}, we have 
 \[
  \rmh_i^{\CS}(G_1\times G_2, (E_1\widehat{\otimes}F_1^{\vee})\boxtimes (E_2\widehat{\otimes}F_2^{\vee}))\simeq \bigoplus_{p+q=i}\rmh_p^{\CS}(G_1, E_1\widehat{\otimes}F_1^{\vee})\otimes \rmh_q^{\CS}(G_2, E_2\widehat{\otimes}F_2^{\vee})
 \]
 is finite-dimensional as well. The proposition now follows by simple calculation.
\end{proof}

\subsection{Mirabolic induction and Mackey induction}\label{induction section}
Let $\pi$ be a representation of $\GL_k$, and $\sigma$ be a representation of $M_{m}$, where $k+m=n$. Let $\sigma^{\flat}$ be a representation of $\ov{M_m}$. Embed $\GL_{k}$ into $\GL_n$ as the subgroup $\begin{pmatrix}
* & 0\\ 0 & I_{m}
\end{pmatrix}$, and $M_{m}$  as  $\begin{pmatrix}
I_{k} & 0\\ 0 & *
\end{pmatrix}$. The following conventions are freely used throughout the article.
\begin{enumerate}
    \item The \textbf{mirabolic induction} $\pi\times \sigma$ is defined as
    \[
    \SInd_{M_n\cap P_{k,m}}^{M_n} (\pi\boxtimes \sigma),
    \]
    where $\pi\boxtimes \sigma$ is a representation of $\GL_k\times M_m$, viewed as a representation of $M_n\cap P_{k,m}$ by trivial extension.
    \item The \textbf{opposite mirabolic induction} $\pi\mind \sigma$ for $M_n$ is defined as
    \[
    \SInd_{M_n\cap \ov{ P_{k,m}}}^{M_n} (\pi\boxtimes \sigma) 
    \]
    where $\pi\boxtimes \sigma$ is a representation of $\GL_k\times M_m$, and is viewed as a representation of $M_n\cap \ov{ P_{k,m}}$ by trivial extension.
    \item The \textbf{opposite mirabolic induction} $\pi\mind \sigma^{\flat}$ for $\ov{M_n}$ is defined as 
    \[
    \SInd_{\overline{M_n}\cap \ov{P_{k,m}}}^{\ov{M_n}} (\pi\boxtimes \sigma^{\flat}) 
    \]
    where $\pi\boxtimes \sigma^{\flat}$ is a representation of $\GL_k\times \ov{M_m}$, and is viewed as a representation of $\overline{M_n}\cap \ov{P_{k,m}}$ by trivial extension.
    \item The \textbf{Mackey induction} $I(\sigma)$ and \textbf{opposite Mackey induction} $\ov{I}(\sigma^{\flat})$ is defined as 
    \[
    I(\sigma):=\mathcal{S}\mathrm{Ind}_{H_{m+1,2}}^{M_{m+1}}(\sigma\boxtimes \psi_{m+1}) \text{ and } \ov{I}(\sigma^{\flat}):=\mathcal{S}\mathrm{Ind}_{\ov{H_{m+1,2}}}^{\ov{M_{m+1}}}(\sigma^{\flat}\boxtimes \psi_{m+1}).
    \]
    For different choices of $\psi$ in the definition of $\psi_{m+1}$, the (opposite) Mackey inductions are isomorphic. Moreover, when $\sigma$(resp. $\sigma^{\flat}$) is irreducible, $I(\sigma)$(resp. $\ov{I}(\sigma^{\flat})$) is also irreducible by Theorem~\ref{Mackey theory}.
    \item The \textbf{trivial extension} $E(\pi)$ is defined as a $M_{k+1}$-representation trivially extending $\pi$. The \textbf{opposite trivial extension} $\ov{E}(\pi)$ is defined as a $\ov{M_{k+1}}$-representation trivially  extending $\pi$.
\end{enumerate}

Now we turn to other classical groups case. Inside $G_n$, let $\GL_n$ be the standard Levi subgroup of $C_n$, that is,
\[ \left\{ \begin{pmatrix}
    A_n g^{-t}A_n & 0\\
    0 & g
\end{pmatrix}, g\in \GL_n\right\},\ \text{where}\ A_n=\left( \begin{smallmatrix}
& & 1\\ & \begin{sideways}$\ddots$\end{sideways} &\\ 1 &&
\end{smallmatrix}\right).
\]
Hence, $\GL_n\cap R_{n-1,1}\simeq P_{n-1,1}$. The following conventions are freely used throughout the article. Let $\sigma$ be a representation of $P_{n-1,1}$ and $\beta$ be a representation of $R_{n-2,1}$.
\begin{enumerate}
    \item The \textbf{parabolic induction} $M(\sigma)$ is defined as 
    \[
    \SInd_{C_n\cap R_{n-1,1}}^{R_{n-1,1}}(\sigma),
    \]
    where $\sigma$ is viewed as a representation of $C_n\cap R_{n-1,1}$ by trivial extension.
    \item The \textbf{opposite parabolic induction} $\ov{M}(\sigma)$ is defined as 
    \[
    \SInd_{\ov{C_n}\cap R_{n-1,1}}^{R_{n-1,1}}(\sigma),
    \]
    where $\sigma$ is viewed as a representation of $\ov{C_n}\cap R_{n-1,1}$ by trivial extension.
    \item The \textbf{Mackey induction} $I(\beta)$ is defined as
    \[
    \SInd_{R_{n-2,1}\ltimes E_n}^{R_{n-1,1}} (\beta\boxtimes \psi_n),
    \]
    Here, we embed $R_{n-2,1}$ into $R_{n-1,1}$ as the embedding of the subalgebas $\mathfrak{r}_{n-2,1}\subset \mathfrak{r}_{n-1,1}$:
    \[\left\{\begin{pmatrix}  x & * & * & *  \\  0 & * & * & *   \\   0 & * & * & *  \\ 0 & 0 & 0 & x  
    \end{pmatrix}\in \mathfrak{g}_{n-1} \ \middle|\ x\in \mathbf{k}\right\}\hookrightarrow \left\{\begin{pmatrix}
    x & 0 & 0 & 0 & 0 & 0 \\ 0 & x & * & * & * & 0 \\ 0
    & 0 & * & * & * & 0 \\ 0 & 0 & * & * & * & 0 \\ 0 & 0 & 0 & 0 & x & 0 \\ 0 & 0 & 0 & 0 & 0 & x
    \end{pmatrix}\in  \mathfrak{g}_n \ \middle|\ x\in \mathbf{k} \right\}. \]
    Then $R_{n-2,1}\ltimes E_n$ is the group of stabilizers of $\psi_n$ in $R_{n-1,1}$. When $\beta$ is irreducible, then $I(\beta)$ is also irreducible.
    \item Let $\pi$ be a representation of $G_{n-1}\times \GL_1$. The \textbf{trivial extension} $E(\pi)$ is defined as a $R_{n-1,1}$-representation trivially extends from $\pi$.
\end{enumerate}
For every induction, we use script ``u" to indicate the \textbf{un-normalized induction}. We have the following associative law for mirabolic induction and Mackey induction.
\begin{lemma}\label{inductive_arg0}
    Let $\pi$ be a representation of $\GL_n$, $\tau$ be a representation of $\GL_m$, and $\sigma$ be a representation of $M_m$. Then we have natural isomorphisms:
\begin{enumerate}
    \item $\pi\times E(\tau)\simeq E(\pi\times \tau)$ and $\pi\times I(\sigma)\simeq I(\pi\times \sigma)$;
    \item $\pi\mind \ov{I}^k\ov{E}(\tau)\simeq \ov{I}^k\ov{E}(\tau\times \pi)$.
\end{enumerate}
\end{lemma}
\begin{proof}
    (1) follows directly from induction by stages. For (2), by induction in stages, we have
    \[
    \pi\mind \ov{I}^k\ov{E}(\tau)\simeq \ov{I}^k\ov{E}(\pi\mind \tau).
    \]
    Conjugating by $\begin{pmatrix}
        0_{n\times m}& I_n \\
        I_{m} & 0_{m\times n}
    \end{pmatrix}$, we have the identification $\pi\mind \tau\simeq \tau\times \pi$.
\end{proof}

\section{Bernstein-Zelevinsky filtration}

In the study of branching problem about representations of general linear groups over non-Archimedean field, a key ingredient is the Bernstein-Zelevinsky filtration of the smooth representation, that is, as a representation of the mirabolic subgroup, it admits a finite filtration whose successive quotients are inductions from  its derivatives. This section discusses an analogous filtration of the Casselman-Wallach representation, also called Bernstein-Zelevinsky filtration, in the Archimedean case. Compared to the non-Archimedean case, the main difference is that we need to restrict ourselves to the case of Casselman-Wallach representation, and there are infinitely many composition factors as representation of the mirabolic subgroup.

In the context, we will use BZ-filtration to mean Bernstein-Zelevinsky filtration.

\subsection{Bernstein-Zelevinsky filtration for $\GL_n$}
 \label{BZ_gl-sec}
 
Let $\mathbf{k}=\mathbb{R}$ or $\mathbb{C}$, and $n=n_1+n_2$. Let $\pi$ be a representation of $\GL_{n_1}$, $\tau$ be a representation of $\GL_{n_2}$, and $\sigma$ be a representation of $M_{n_2}$. 

To use an inductive argument, we will study the relations between the representations, which are constructed from the same representation of a small subgroup but via different orders of the functors ``$E$" and ``$I$".  The relations are in Lemma \ref{inductive_arg0}, Lemma \ref{inductive_arg}, and Lemma \ref{open_orbit}. The methods to prove these lemmas are inspired by~\cite[Lemma 2.1]{S89}, where unitary Hilbert representations rather than smooth representations were considered. For smooth representations, Lemma~\ref{open_orbit} shows that $I(\pi\times\tau|_{M_{n_2}})\hookrightarrow\tau\overline{\times}E(\pi)$ is an embedding but not a surjection, while they are isomorphic in the setting of unitary Hilbert representations \cite[Lemma 2.1, (v)]{S89}. We emphasize that our proof is more canonical than \textit{loc.cit.} since our proof is independent of the coordinate choice.

We first recall the \textbf{Fourier transform} with respect to the unitary character $\psi$. Let $u$ denote the coordinate vector in the domain $\mathbf{k}^n$, and $\xi$ the coordinate vector in the codomain $(\mathbf{k}^n)^*$. Let $V$ be a Fr\'echet space. The Fourier transform $\CF : \CS(\mathbf{k}^n,V) \to \CS\left((\mathbf{k}^n)^*,V\right)$, defined by
\[
\CF_u(f)(\xi) := \int_{\mathbf{k}^n} f(u) \psi(\langle\xi,u\rangle)  du, \quad f \in \CS(\mathbf{k}^n,V),
\]
is an isomorphism of Fr\'echet spaces, where $du$ denotes the Euclidean measure.

\begin{lemma}\label{inductive_arg}
The representations $\sigma_1:=\pi\mind I(\sigma)$ and $\sigma_2:=I(\pi\mind \sigma)$ of $M_{n+1}$ are isomorphic to each other. 
\end{lemma}
\begin{proof}
Consider the subgroup $Y\subset M_{n+1}$,
\[\left\{\begin{pmatrix} a & 0  & b\\ c & d & e\\ 0 & 0  & 1\end{pmatrix} \ \middle| \ a\in \GL_{n_1}, b\in \mathbf{k}^{n_1\times 1}, c\in \mathbf{k}^{n_2\times n_1}, d\in M_{n_2}, e\in \mathbf{k}^{n_2\times 1}\right\}\]
and the subgroups of $Y$, 
\[Y_1=\left\{\begin{pmatrix} a & 0 & 0 \\ c & d & e\\ 0  & 0  & 1\end{pmatrix}\in Y\right\},\ Y_2=\left\{\begin{pmatrix} a & 0 & b\\ c' & d & e\\ 0  & 0  & 1\end{pmatrix}\in Y \ \middle| \ c'\in \begin{pmatrix} \mathbf{k}^{(n_2-1)\times n_1}\\ 0_{1\times n_1}\end{pmatrix}\right\}.\]
Taking the trivial extension to $c$, and the extension to $e$ by $\psi_{n_2}$,  one obtains the representation $(\pi\boxtimes \sigma \boxtimes \psi_{n_2})$ of $Y_1$. Let \[\gamma_1:=\SInd_{Y_1}^Y(\pi\boxtimes \sigma \boxtimes \psi_{n_2}).\]   
Taking the trivial extension to $c'$, and the extension to the last column by $\psi_{n}$,  one obtains the representation $(\pi\boxtimes \sigma \boxtimes \psi_{n})$ of $Y_2$. Let  \[\gamma_2:=\SInd_{Y_2}^Y(\pi\boxtimes \sigma \boxtimes \psi_{n}).\]
By induction in stages, $\sigma_i\simeq \SInd_{Y}^{M_{n+1}}(\gamma_i)$ for $i=1,2$.  To prove the lemma, it suffices to show $\gamma_1\simeq \gamma_2$.

Let 
\[y=\begin{pmatrix} a & 0 & b\\ c & d & e\\ 0 & 0 & 1\end{pmatrix}\in Y; \ c=\begin{pmatrix}
* \\ \mu
\end{pmatrix}, \mu\in \mathbf{k}^{1\times n_1}; \ e=\begin{pmatrix}
* \\ \lambda
\end{pmatrix}, \lambda\in \mathbf{k}.\]
The $\gamma_1$ can be realized on the space of Schwartz functions from 
\[\Xi_1=\left\{\begin{pmatrix} I_{n_1} & 0 & u\\ 0 & I_{n_2} & 0\\ 0 & 0 & 1\end{pmatrix}\ \middle| \  u\in \mathbf{k}^{n_1\times 1}\right\} \]
to the underlying space of $\pi\boxtimes \sigma$. 
And the action of $y$ on $f\in \mathcal{S}(\Xi_1,\pi\boxtimes\sigma)$ is given by
\[(\gamma_1(y)f)(u)=|\det(a)|_{\mathbf{k}}^{-\frac{1}{2}}\psi(\lambda-\mu\cdot a^{-1}\cdot (b+u))\big(\pi(a)\boxtimes \sigma(d)\big) f(a^{-1}(b+u)),\]

The $\gamma_2$ can be realized on the space of  Schwartz functions from 
\[\Xi_2=\left\{\begin{pmatrix} I_{n_1} & 0 & 0\\ c'' & I_{n_2} & 0\\ 0 & 0 & 1\end{pmatrix}\ \big\vert\  c''=\begin{pmatrix} 0_{(n_2-1)\times n_1}\\ v \end{pmatrix},  v \in \mathbf{k}^{1\times n_1}\right\}\]  
to the underlying space of $\pi\boxtimes \sigma$, and the action of $y$ on $h\in \mathcal{S}(\Xi_2,\pi\boxtimes\sigma)$ is given by
\[(\gamma_2(y)h)(v)=|\det(a)|_{\mathbf{k}}^{\frac{1}{2}}\psi(v\cdot b+\lambda)\big(\pi(a)\boxtimes \sigma(d)\big)h(\mu+v\cdot a).\]

Applying the Fourier transform to the variable $u$, and letting $\xi\in \mathbf{k}^{1\times n_1}$ be the dual variable after the Fourier transform, one has
\[\mathcal{F}_u\big(\gamma_1(y) f\big)(\xi)=|\det(a)|_{\mathbf{k}}^{\frac{1}{2}}\big(\pi(a)\boxtimes \sigma(d)\big)\CF_u(f)(\mu+\xi\cdot a)\cdot \psi(\xi\cdot b+\lambda).\]

This matches $\gamma_2$’s action under $\mathcal{F}_u$, that is, $\gamma_2\simeq \mathcal{F}_u\circ \gamma_1\circ \mathcal{F}_u^{-1}$. So $\gamma_1$ and $\gamma_2$ are isomorphic, and the lemma follows.
\end{proof}

Although our primary focus is on the BZ-filtration of $\pi|_{M_n}$ for a Casselman-Wallach representation $\pi$ of $\mathrm{GL}_n$, it is beneficial to consider representations of $M_n$ in a more general context. We therefore introduce a definition of the BZ-filtration for smooth representations of $M_n$. It will be shown that $\pi|_{M_n}$, for a Casselman-Wallach representation $\pi$ of $\mathrm{GL}_n$, admits a BZ-filtration.

\begin{definition}\label{def-BZ}
Let $\sigma$ be a smooth representation of $M_n$. We call the following datum \textbf{a level $\leq 1$ BZ-filtration} of $\sigma$:
A level $\leq 1$ filtration of $\sigma$ as in Definition \ref{def_fil}
such that 
\begin{itemize}
    \item Each $\sigma_{i,j}/\sigma_{i,j+1}$ is isomorphic to $I^{k_i}E(\pi_{i,j})$ for some $k_i$ (dependent on $i$ but independent of $j$) and irreducible representations $\pi_{i,j}$ of $\mathrm{GL}_{n-k_i-1}$,
    \item The real parts of the central characters satisfy $\mathrm{Re}(\omega_{\pi_{i,j}}) \leq \mathrm{Re}(\omega_{\pi_{i,j+1}})$ for all $j$. Moreover, for every $c \in \mathbb{R}$, there are finitely many $j$ such that $\mathrm{Re}(\omega_{\pi_{i,j}}) \leq c$.
\end{itemize}

For $r \geq 2$, we call the following datum \textbf{a level $\leq r$ BZ-filtration} of $\sigma$: A level $\leq r$ filtration of $\sigma$ as in Definition \ref{def_fil} such that the filtration on each $\sigma_{i,j}/\sigma_{i,j+1}$ is a level $\leq r-1$ BZ-filtration, and
\begin{itemize}
    \item Let $\Omega_{i,j}$ denote the set of $\mathrm{Re}(\omega_{\pi})$ with $I^kE(\pi)$ being the irreducible successive quotient in the filtration of $\sigma_{i,j}/\sigma_{i,j+1}$. Then, for each $i,j$,  $\min \Omega_{i,j} \leq \min \Omega_{i,j+1}$. Moreover, for every $c \in \mathbb{R}$, there are only finitely many $j$ such that $\min \Omega_{i,j} \leq c$.
\end{itemize}

We say that a representation $\sigma$ of $M_n$ admits a BZ-filtration if $\sigma$ admits a level $\leq r$ BZ-filtration for some finite $r \in \mathbb{Z}_{>0}$.
\end{definition}

The following lemma is useful to establish the existence of BZ-filtrations.
\begin{lemma}\label{ext_BZ}
Let $0\to \sigma^{\flat}\to\sigma\to \sigma^{\sharp}\to 0$ be an exact sequence of smooth representations of $M_n$. If both $\sigma^{\flat}$ and $\sigma^{\sharp}$ admit BZ-filtrations, then so does $\sigma$. If $\sigma$ admits a BZ-filtration, then so does $\sigma^{\flat}$. 
\end{lemma}
\begin{proof}
Assume that $\sigma^{\flat}$ (resp. $\sigma^{\sharp}$) admits a level $\leq r'$ (resp. $\leq r''$)  BZ-filtration. By combining these two filtrations together, one obtains a level $\leq \max\{r',r''\}$ BZ-filtration of  $\sigma$. 

Conversely, first assume that $\sigma$ admits level $\leq 1$ BZ-filtration $\{\sigma_i, \sigma_{i,j}\}$. Then $\sigma_i^{\flat}=\sigma_i\cap \sigma^{\flat}$, and $\sigma_{i,j}^{\flat}=\sigma_{i,j}\cap \sigma^{\flat}$ are all closed subrepresentations. We claim that $\{\sigma_i^{\flat}, \sigma_{i,j}^{\flat}\}$ gives a level $\leq 1$  BZ-filtration of $\sigma^{\flat}$.

By Theorem~\ref{Mackey theory}, the successive quotient in the BZ-filtration of $\sigma$ is algebraically irreducible. Hence, it suffices to prove  $\sigma^{\flat}_i/\sigma^{\flat}_{i+1}\simeq \mathop{\varprojlim}\limits_j \sigma_i^{\flat}/\sigma_{i,j}^{\flat}$.  On the one hand, the image of the map 
\[\varsigma_1: \sigma_i^{\flat}/\sigma_{i+1}^{\flat}\hookrightarrow \sigma_i/\sigma_{i+1}\simeq \mathop{\varprojlim}\limits_j \sigma_i/\sigma_{i,j}\] 
is inside $\mathop{\varprojlim}\limits_j \sigma_i^{\flat}/\sigma_{i,j}^{\flat}$. On the other hand, since $\sigma^{\flat}_i/\sigma^{\flat}_{i+1}$ is closed in $\sigma_i/\sigma_{i+1}$, the image of the map
\[\varsigma_2:
\mathop{\varprojlim}\limits_j \sigma^{\flat}_i/\sigma^{\flat}_{i,j}\hookrightarrow \mathop{\varprojlim}\limits_j \sigma_i/\sigma_{i,j}\simeq \sigma_i/\sigma_{i+1}
\]
is inside $\sigma^{\flat}_i/\sigma^{\flat}_{i+1}$. Then the result follows from $\varsigma_1\circ \varsigma_2=\varsigma_2\circ \varsigma_1=id$.

In general, when $\sigma$ admits level $\leq r$ BZ-filtration, by the similar argument as above, one can also obtain a level $\leq r$ BZ-filtration of $\sigma^{\flat}$.

\end{proof} 

Recall that $P_{n_1,n_2}$ denotes the standard parabolic subgroup with Levi subgroup $\GL_{n_1}\times \GL_{n_2}$, and $\overline{P_{n_1,n_2}}$ denotes the opposite parabolic subgroup of $P_{n_1,n_2}$.
\begin{lemma}\label{open_orbit}
Let $\pi_2$ be a Casselman-Wallach representation of $\GL_{n_2}$, which admits a level $\leq r$ BZ-filtration, and let $\pi_1$ be a Casselman-Wallach representation of $\GL_{n_1-1}$.  Then the representation  $\sigma:= \pi_2\overline{\times}E(\pi_1)$ of $M_n$ admits a level $\leq r$ BZ-filtration.
\end{lemma}

\begin{proof}
The argument is similar to that of Lemma \ref{inductive_arg}. Firstly, we use Fourier transform to intertwine two representations of $Y$ which are induced from $Y_1$ and $Y_2$, respectively, where $Y$, $Y_1$ and $Y_2$ are defined as follows,
\begin{itemize} 
\item Subgroup $Y$ of $\GL_n$: 
\[Y:=\left\{\left(\begin{array}{c|c} \begin{matrix} a & b & 0 \\ d & e & 0 \end{matrix} & \begin{matrix} c \end{matrix}  \\ \hline \begin{matrix} g & h & i \\ 0 & 0 & 0 \end{matrix} & \begin{matrix} j \\ 1\end{matrix}\end{array}\right)\in \GL_n\ \middle| \ \begin{array}{l}  a\in \mathbf{k}^{(n_2-1)\times (n_2-1)}, e\in \mathbf{k}^{1\times 1}, \\ i\in \mathbf{k}^{(n_1-1)\times (n_1-1)}, c\in \mathbf{k}^{n_2\times 1} \\  b,d,g,j\ \text{are sub-matrices over}\ \mathbf{k} \end{array} \right\},\] 
\item Subgroups in $Y$: 
\[Y_1=\left\{\left(\begin{array}{c|c} \begin{matrix} a & b & 0 \\ d & e & 0 \end{matrix} & \begin{matrix} 0 \end{matrix}  \\ \hline \begin{matrix} g & h & i \\ 0 & 0 & 0 \end{matrix} & \begin{matrix} j \\ 1\end{matrix}\end{array}\right)\in Y\right\}, \quad Y_2=\left\{\left(\begin{array}{c|c} \begin{matrix} a & b & 0 \\ 0 & 1 & 0 \end{matrix} & \begin{matrix} c \end{matrix}  \\ \hline \begin{matrix} g & h & i \\ 0 & 0 & 0 \end{matrix} & \begin{matrix} j \\ 1\end{matrix}\end{array}\right)\in Y\right\}.\]
\end{itemize}

By trivial extension, one can get the representation $\pi_2\boxtimes E(\pi_1)\boxtimes 1$ of $Y_1$.  Let \[\zeta_1:=\mathcal{S}\mathrm{Ind}_{Y_1}^Y\left(\pi_2\boxtimes E(\pi_1)\boxtimes 1\right).\]
Then by induction in stages, $\sigma\simeq \SInd_{Y}^{M_n}(\zeta_1)$. To prove the statement, let us study $\zeta_1$ in detail.

The underlying space of $\zeta_1$ can be identified with the space of Schwartz functions from
\[\Xi:= \left\{\begin{pmatrix}
I_{n_2} & 0 & u\\ 0 & I_{n_1-1} & 0\\ 0 & 0 & 1
\end{pmatrix}\ \middle|\ u\in \mathbf{k}^{n_2}\right\}\simeq  \mathbf{k}^{n_2}\]
to the underlying space of $\pi_2\boxtimes \pi_1$, and the action of $y=\left(\begin{array}{c|c} \begin{matrix} a & b & 0 \\ d & e & 0 \end{matrix} & \begin{matrix} c \end{matrix}  \\ \hline \begin{matrix} g & h & i \\ 0 & 0 & 0 \end{matrix} & \begin{matrix} j \\ 1\end{matrix}\end{array}\right)\in Y$ on $f\in \mathcal{S}(\Xi,\pi_2\boxtimes \pi_1)$ is given by
\[\zeta_1(y)f(u)=|\det(A)|_{\mathbf{k}}^{-\frac{1}{2}}\cdot (\pi_2(A)\boxtimes \pi_1(i)) f(A^{-1}(u+c)), \ \text{where}\ A=\begin{pmatrix} a & b \\ d & e\end{pmatrix}.\]

Consider the Fourier transform of $\Xi$ with respect to the unitary character $\psi$ over the field $\mathbf{k}$. Let $\xi$ be the dual variables corresponding to $u$.
Then  \[\mathcal{F}_{u}\circ\zeta_1(y)\circ \mathcal{F}_{u}^{-1}(\widehat{f} )(\xi)=|\det(A)|_{\mathbf{k}}^{-\frac{1}{2}+1}\cdot (\pi_2(A)\boxtimes \pi_1(i)) \psi(\xi\cdot c)\widehat{f}(\xi\cdot A).\]  
Let $\widehat{\zeta}_1$ denote $\mathcal{F}_u\circ\zeta_1\circ \mathcal{F}_u^{-1}$. Under the action of  $\widehat{\zeta}_1$, $\mathcal{S}(\mathbf{k}^{n_2}\setminus\{0\},\pi_2\boxtimes \pi_1)$, which is a subspace of $\mathcal{S}(\mathbf{k}^{n_2},\pi_2\boxtimes \pi_1)\simeq  \mathcal{S}(\Xi,\pi_2\boxtimes \pi_1)$, is $Y$-stable, and the corresponding subrepresentation is denoted by $\widehat{\zeta}_1^{\flat}$. Hence, $\widehat{\zeta}_1$ admits a filtration 
\[0\to \widehat{\zeta}_1^{\flat}\to \widehat{\zeta}_1\to \widehat{\zeta}_1^{\sharp}\to 0,\]
where $\widehat{\zeta}_1^{\sharp}:=\widehat{\zeta}_1/\widehat{\zeta}_1^{\flat}$.
By Lemma \ref{ext_BZ}, it suffices to prove the statement for $\SInd_Y^{M_n}(\widehat{\zeta}_1^{\flat})$ and $\SInd_Y^{M_n}(\widehat{\zeta}_1^{\sharp})$. Let us deal with them separately.

\noindent \textbf{(i). BZ-filtration of $\SInd_Y^{M_n}(\widehat{\zeta}_1^{\flat})$.} We will show that it is isomorphic to the representation $\SInd_{Y}^{M_n}({\zeta}_2)$, which admits a BZ-filtration.  Let $w=\begin{pmatrix} 0 &I_{n_2} & 0\\ I_{n_1-1} & 0 & 0 \\ 0 & 0 & 1\end{pmatrix}$. Taking the extension to the last column of $Y$ by $^w\psi_n$, and trivial extension to the remaining part, one obtains the representation $(\pi_2|_{M_{n_2}}\boxtimes \pi_1)\boxtimes {}^w\psi_n$ of $Y_2$. Let \[\zeta_{2}:=\mathcal{S}\mathrm{Ind}_{Y_2}^{Y}((\pi_2|_{M_{n_2}}\boxtimes \pi_1)\boxtimes {}^w\psi_n).\]   
The underlying space of $\zeta_2$ is the space of Schwartz sections of $({\pi_2|_{M_{n_2}}}\boxtimes \pi_{1}\boxtimes {}^w\psi_n\otimes (\frac{\delta_{Y_2}}{\delta_{Y}})^{\frac{1}{2}})\times_{Y_2} Y$. 

Under conjugation by $w$, one has $ \mathcal{S}\mathrm{Ind}_{Y}^{M_n}\zeta_2\simeq I(\pi_1\times \pi_2|_{M_{n_2}})$. By assumption, $\pi_2|_{M_{n_2}}$ admits a level $\leq r$ BZ-filtration. By induction, one obtains a level $\leq r$  BZ-filtration of $I(\pi_1\times \pi_2|_{M_{n_2}})$, which is given by applying $I(\pi_1\times \bullet)$ to the BZ-filtration $\bullet$ of  $\pi_2|_{M_{n_2}}$.

It remains to show that $\widehat{\zeta}_1^{\flat}$ is isomorphic to $\zeta_{2}$ as $Y$-representations.

The underlying space of $\zeta_2$ can be described more explicitly as follows. Consider
the affine subsets of $\GL_{n_2}\simeq \begin{pmatrix} \GL_{n_2} & \\  & I_{n_1}\end{pmatrix}\subset Y$: 
\[\Omega_r:=\left\{a_{r,\xi}:=\begin{pmatrix}\ \begin{matrix} 
   0 & 0 & I_{n_2-r-1}\\ I_{r} & 0 & 0  \\  \hline \multicolumn{3}{c}{\xi} \end{matrix}\ 
\end{pmatrix}\ \middle|\ \xi=[\xi_1,\dots,\xi_{n_2}],\ \text{with} \ \xi_{r+1}\neq 0\right\},\]
for $0\leq  r\leq n_2-1$, which satisfies $Y=\bigcup_rY_2\cdot \Omega_r$. The Schwartz sections of the underlying space of $\zeta_2$ supported on $Y_2\cdot \Omega_r$ can be identified with $\mathcal{S}(\Omega_r,\pi_2\boxtimes \pi_1)$, and $\zeta_2$ is spanned by $\sum_r\mathcal{S}(\Omega_r,\pi_2\boxtimes \pi_1)$.

Define the intertwining operator $\mathcal{T}_r$ as follows,
\[\begin{aligned}\mathcal{T}_r:\mathcal{S}(\Omega_r,\pi_2&\boxtimes \pi_1) \to \mathcal{S}(\mathbf{k}^r\times \mathbf{k}^{\times} \times \mathbf{k}^{n_2-r-1},\pi_2\boxtimes \pi_{1}) \\ \CT_r( f)(\xi) &:= 
|\xi_{r+1}|_{\mathbf{k}}^{-\frac{1}{2}}\pi_2(a_{r,\xi})^{-1}f(a_{r,\xi}).\end{aligned}\]
Since $\pi_2$ is a moderate growth representation, this map is a well-defined.
Moreover, $\mathcal{T}_r=\mathcal{T}_{r'}$ over the intersection of Schwartz sections over $\Omega_r\cdot Y_2$ and $\Omega_{r'}\cdot Y_2$. Therefore, by the co-sheaf property of Schwartz functions, there is a well-defined topological linear isomorphism
\[
\bigcup_r \CT_r : \zeta_2 \to \mathcal{S}(\mathbf{k}^{n_2}\setminus\{ 0\}, \pi_2\boxtimes \pi_1).
\]
Furthermore,  the  intertwining operator satisfies
\begin{align*} 
& \mathcal{T}_r\circ \zeta_2(y)\circ \mathcal{T}_r^{-1}(f)(\xi)\\ = &\ \ |\xi|_{\mathbf{k}}^{-\frac{1}{2}}|\det(A')|_{\mathbf{k}}^{\frac{1}{2}}|\xi'|_{\mathbf{k}}^{\frac{1}{2}}
\cdot \psi(\xi\cdot c)\pi_2(a_{r,\xi})^{-1}(\pi_2(A')\boxtimes \pi_1(i))\pi_2(a_{r,\xi'}) f(\xi'),\end{align*}
where $\xi'=\xi\cdot A$ and $A'\in M_{n_2}$ satisfy $A'\cdot a_{r,\xi'}=a_{r,\xi}\cdot A$. Since \[
\pi_2(a_{r,\xi})^{-1}\pi_2(A')\pi_2(a_{r,\xi'})=\pi_2(A),\]
one has 
\[\mathcal{T}_r\circ \zeta_2(y)\circ \mathcal{T}_r^{-1}=\widehat{\zeta}_1(y),\forall y\in Y\] over $\mathcal{S}(\mathbf{k}^r\times  \mathbf{k}^{\times} \times \mathbf{k}^{n_2-r-1},\pi_2\boxtimes \pi_1)$. Hence, $\bigcup_r \mathcal{T}_r$ intertwines $\zeta_{2}$ and $\widehat{\zeta}_{1}^{\flat}$.

\noindent \textbf{(ii). BZ-filtration of $\SInd_{Y}^{M_n}(\widehat{\zeta}_1^{\sharp}$).} The underlying space of $\widehat{\zeta}_1^{\sharp}$  is $$\mathcal{S}_{\{0\}}(\mathbf{k}^{n_2},\pi_2\boxtimes \pi_1).$$ 
By Borel's lemma, it admits a natural decreasing filtration \[\widehat{\zeta}_1^{\sharp}=\widehat{\zeta}_{1,0}\supset \widehat{\zeta}_{1,1}\supset \widehat{\zeta}_{1,2}\supset \dots\] 
with $\widehat{\zeta}_1^{\sharp}\simeq \varprojlim\limits_j \widehat{\zeta}_1^{\sharp}/\widehat{\zeta}_{1,j}$ and  \[\widehat{\zeta}_{1,j}/\widehat{\zeta}_{1,j+1}\simeq (|\det|_{\mathbf{k}}^{\frac{1}{2}}\cdot \pi_2\otimes_{\BR} \mathrm{Sym}^j(\mathbf{k}^{ n_2}))\boxtimes E(\pi_1)\boxtimes 1,\]
where $\mathbf{k}^{n_2}$ is the natural representation of $\GL_{n_2}$. 

It follows that $\SInd_{Y}^{M_n}(\widehat{\zeta}_1^{\sharp})$ admits a decreasing filtration with successive quotients
\[E\left((|\det|_{\mathbf{k}}^{\frac{1}{2}}\cdot \pi_2\otimes_{\BR} \mathrm{Sym}^j(\mathbf{k}^{ n_2}))\overline{\times}  \pi_1\right).\]
Notice that $(|\det|_{\mathbf{k}}^{\frac{1}{2}}\cdot \pi_2\otimes_{\BR} \mathrm{Sym}^j(\mathbf{k}^{ n_2}))\overline{\times}  \pi_1$ has finite length. Taking a finer filtration if needed, one can get a level $\leq 1$ BZ-filtration of $\SInd_{Y}^{M_n}(\widehat{\zeta}_1^{\sharp})$. This finishes the proof of the Lemma.
\end{proof} 

\begin{theorem}\label{ind_BZ}
Let $\pi$ be the parabolically induced representation $\mathrm{Ind}_{P_{n_1,n_2}}^{\GL_n}(\pi_1\boxtimes \pi_2)$, where $\pi_i$'s are  Casselman-Wallach representations of $\GL_{n_i}$ such that $\pi_i|_{M_{n_i}}$ admits a level $\leq r_i$ BZ-filtration. Then $\pi|_{M_n}$ admits a level  $\leq \max\{r_1+r_2,r_2+1\}$ BZ-filtration. 
\end{theorem}

\begin{proof}   
By 
\[P_{n_1,n_2}\backslash \GL_n/M_n\simeq \left\{I_n,w=\begin{pmatrix}
0 & I_{n_2}\\ I_{n_1} & 0
\end{pmatrix}\right\},\] 
there are two $M_n$-orbits on $P_{n_1,n_2}\backslash G$, the open orbit of $w$ and the closed orbit of $I_n$.
This yields a filtration of $\pi|_{M_n}$ under the notation of \eqref{open_close}
\[0\to \pi_o\to \pi\to \pi_c\to 0,\] 
where $\pi_o\simeq \pi_2\overline{\times} \pi_1|_{M_{n_1}},$
and by Borel's lemma, $\pi_c$ admits an infinite decreasing filtration 
$$\pi_c=(\pi_c)_0\supset (\pi_c)_1\supset (\pi_c)_2\supset\dots$$
with $\pi_c=\varprojlim\limits_i (\pi_c)_0/(\pi_c)_i$, and
\begin{equation}\label{closed_cal}
    \begin{aligned} (\pi_c)_i/(\pi_c)_{i+1} & \simeq \big(|\det|_{\mathbf{k}}^{\frac{1}{2}}\cdot \pi_1 \otimes  \mathrm{Sym}^i(\mathfrak{s}^{\vee})\big)\times \pi_2|_{M_{n_2}}\end{aligned}
\end{equation}
where $\mathfrak{s}=\mathfrak{g}/(\mathfrak{p}_n+\mathfrak{p})$,  and $\mathfrak{s}^{\vee}$ carries the adjoint action of $\GL_{n_1}$.

To prove the statement, it suffices to show that both $\pi_o$ and $\pi_c$ admit BZ-filtrations of  such level. 

For $\pi_c$, by the assumption, $\pi_2|_{M_{n_2}}$ admits a level $\leq r_2$ BZ-filtration. By applying $ \big(|\det|_{\mathbf{k}}^{\frac{1}{2}}\cdot \pi_1 \otimes \mathrm{Sym}^i(\mathfrak{s}^{\vee})\big)\times \bullet$ to the filtration $\bullet$, then using Lemma \ref{inductive_arg0} and taking refinement if needed, one obtains a level $\leq r_2+1$ BZ-filtration of $\pi_c$.

For $\pi_o \simeq \pi_2\overline{\times}\pi_1|_{M_{n_1}}$, by the assumption, $\pi_1|_{M_{n_1}}$ admits level $\leq r_1$ BZ-filtration. For every successive quotient $I^{k}E(\widetilde{\pi})$ of this filtration of $\pi_1|_{M_{n_1}}$, 
we claim that $\pi_2\overline{\times} I^{k}E(\widetilde{\pi}) $ admits a level $\leq r_2$ BZ-filtration, which implies that $\pi_o$ admits a level $\leq r_1+r_2$ BZ-filtration.

To prove the claim, we use Lemma \ref{inductive_arg},  
\[\pi_2\overline{\times} I^{k}E(\widetilde{\pi})  \simeq  
I^k(\pi_2\overline{\times} E(\widetilde{\pi})).\]
By Lemma \ref{open_orbit}, $\pi_2\overline{\times} E(\widetilde{\pi})$
admits a level $\leq r_2$ BZ-filtration, since $\pi_2|_{M_{n_2}}$ admits a level $\leq r_2$ BZ-filtration. So does $I^k(\pi_2\overline{\times} E(\widetilde{\pi}))$. The claim now follows.

\end{proof}

\begin{theorem}\label{have_B-Z_fil}
Let $\pi$ be a Casselman-Wallach representation of $\GL_n$. Then $\pi|_{M_n}$ admits a BZ-filtration of level $\leq n$.
\end{theorem}

\begin{proof}
By Lemma \ref{ext_BZ}, we can assume $\pi$ is irreducible. In this case, $\pi$ can be embedded into a principal series which is induced from Borel subgroup of $\GL_n$. By Theorem \ref{ind_BZ}, an inductive argument implies that the principal series has BZ-filtration of level $\leq n$. So does $\pi|_{M_n}$ by Lemma~\ref{ext_BZ}.
\end{proof}

\subsection{The restriction to maximal parabolic subgroup $P_{n-1,1}$}

Let us consider the Mackey induction theory for the group $P_{n-1,1}$. Note the nilpotent radical of $P_{n-1,1}$ is $V_n$, which is the same as that of $M_n$.
For $\psi_n\in \mathfrak{v}_n^*$,
 the stabilizer of $\psi_n$ in $P_{n-1,1}$ is 
 \[\left\{\begin{pmatrix} * & * & *\\ 0_{1\times (n-2)} & x & *\\ 0_{1\times (n-2)} & 0 & x\end{pmatrix}\ \middle|\  x\in \mathbf{k}^{\times}\right\}.\]
Embedding $P_{n-2,1}$ into $P_{n-1,1}$ as the subgroup
\[ \left\{\begin{pmatrix} * & * & 0_{(n-2)\times 1}\\ 0_{1\times (n-2)} & x & 0\\ 0_{1\times (n-2)} & 0 & x\end{pmatrix}\ \middle|\  x\in \mathbf{k}^{\times}\right\},\]
one has $\mathrm{Stab}_{P_{n-1,1}}(\psi_n)\simeq P_{n-2,1}\ltimes V_n$.

For a smooth representation $\sigma$ of $P_{n-2,1}$, let $I(\sigma)$ denote $\SInd_{P_{n-2,1}\ltimes V_n}^{P_{n-1,1}}\left( \sigma\otimes \psi_n\right)$. By Lemma~\ref{Mackey theory}, if $\sigma$ is (algebraically) irreducible, then so is $I(\sigma)$. For a smooth representation $\tau$ of $\GL_{n-1}\times \GL_1$, let $E(\tau)$ denote the representation of $P_{n-1,1}$ by trivially extending to $V_n$.

Given an irreducible representation $\pi$ of $\GL_n$, the center $\{c \cdot I_n \mid c \in \mathbf{k}^{\times}\}$ acts by the character $\omega_{\pi}$. Consequently, a filtration of $\pi|{M_n}$ induces a filtration of $\pi|{P_{n-1,1}}$ whose successive quotients are of the form $I^{k-1}E(\tau)$ for some positive integer $k$ and irreducible representation $\tau$ of $\GL_{n-k} \times \GL_1$.

\subsection{Bernstein-Zelevinsky filtration of quasi-split classical groups}\label{BZ-cl-sec}
Let $G_{n}$ be the quasi-split classical group defined  previously in Subsection~\ref{derivative-sec}. Recall that $G_{n-1}\times \GL_1 \subset G_n$ is the standard Levi subgroup of $R_{n-1,1}$, and let $\mathrm{GL}_{n}$ be the Levi subgroup of $C_{n}\subset G_{n}$. For the notation used for various representations of $R_{n-1,1}$, we refer the reader to Subsection~\ref{induction section}.

In order to study the filtration of Casselman-Wallach representation of $G_n$, it suffices to study the principal series by Casselman embedding theorem. Let $I$ be a principal series of $G_n$, which is viewed as parabolic induction $\Ind_{C_n}^{G_n}\pi$, where $\pi$ is a principal series of the Levi subgroup $\GL_n$ of $C_n$.

In this article, we deal with $G_n=\SO(n,n)$, whose $E_n$ is abelian. Hence, $\mathbf{k}=\mathbb{R}$ in the remaining of this subsection. In order to obtain a BZ-filtration of $I$, we observe that $R_{n-1,1}$ has a unique open orbit and a unique closed orbit on $C_n\backslash G_n$, which leads to 
\[
0\lra I_o\lra I|_{R_{n-1,1}}\lra I_c\lra 0.
\]
Here, by Borel's lemma, $I_c$ admits a decreasing filtration indexed by non-negative integers $j$ with successive quotients
\[
 M\left(\pi|_{P_{n-1,1}}\cdot |(\det)_{\GL_{n-1}}|^{-1/2}\cdot |(\det)_{\GL_1}^{n-1}|^{-1/2}\otimes\Sym^j (\mathfrak{f}_n)\right),
\]
where $\GL_{n-1}\times \GL_1$ is the Levi subgroup of $P_{n-1,1}$, and $\mathfrak{f}_n$ is the complexified Lie algebra of $F_n$ which carries the adjoint action of $P_{n-1,1}\subset R_{n-1,1}$. By Theorem~\ref{have_B-Z_fil}, $\pi|_{P_{n-1,1}}$ admits a BZ-filtration, which gives rise to a BZ-filtration of $I_c$. The successive quotients of this filtration have the form
\[
 I^{k-1}E( \Ind_{Q_{n-k}}^{G_{n-k}} (\beta))
\]
for some positive integer $k$ and irreducible $\GL_{n-k}$-representation $\beta$. 

On the other hand, the open orbit $I_o\simeq \ov{M} (\pi|_{P_{n-1,1}})$. The BZ-filtration of $\pi|_{P_{n-1,1}}$ leads to a filtration of $I_o$ with successive quotients of two types:
\begin{itemize} 
\item[(i)] $\ov{M}(E(\tau))$ for some irreducible $\GL_{n-1}\times \GL_1$-representation $\tau$, or
\item[(ii)] $\ov{M}(I(\sigma))$ for some $P_{n-2,1}$-representation $\sigma$.
\end{itemize}
Let us discuss these two cases separately.

\subsubsection{Spectrum decomposition of  $\ov{M}(E(\tau))$}

Recall that $A_{m}$ denotes the $m\times m$ anti-diagonal matrix with $1$ in entries. By conjugation of $w=\begin{pmatrix}
1 & 0 & 0  \\  0 & A_{2n-2} & 0\\  0 & 0 & 1
\end{pmatrix}$, since $w(g)=(g^{-1})^t$ for $g\in \GL_{n-1}$, $\ov{M}(E(\tau))$ is isomorphic to $\SInd_{C_{n-1}\cdot \GL_1\cdot F_n}^{R_{n-1,1}}(\tau^{\vee}\otimes 1)$. For convenience, we replace $\tau^{\vee}$ by $\tau$, and will work with $\sigma:=\SInd_{C_{n-1}\cdot \GL_1\cdot F_n}^{R_{n-1,1}}(\tau\otimes 1)$ from now on. 
\begin{proposition}\label{open-dis}
Let 
$$\sigma^{\flat}=I(M(\tau|_{P_{n-2,1}}))=\mathcal{S}\mathrm{Ind}_{P_{n-2,1}\cdot U_{n-1}\cdot E_{n}}^{R_{n-1,1}}(\tau|_{P_{n-2,1}}\otimes 1\otimes \psi_n).$$
Then one has an exact sequence as $R_{n-1,1}$-representation
\[0\to \sigma^{\flat}\to \sigma\to \sigma^{\sharp}\to 0\]
with $\sigma^{\sharp}$ having a decreasing filtration as $M_n$-representation
\[\sigma^{\sharp}=\sigma^{\sharp}_0\supset \sigma^{\sharp}_1\supset \sigma^{\sharp}_2\supset \dots\]
such that $\sigma^{\sharp}\simeq \varprojlim\limits_{i}\sigma^{\sharp}/\sigma_i^{\sharp}$, and $\sigma^{\sharp}_i/\sigma^{\sharp}_{i+1}$ is isomorphic to the trivial extension of $ \mathrm{Ind}_{C_{n-1}\cdot \GL_1}^{G_{n-1}\cdot \GL_1}(\tau_i)$,  where $\tau_i$ is a Casselman-Wallach representation of $\mathrm{GL}_{n-1}\times \GL_1$.

\end{proposition}

\begin{proof}
Consider the subgroup $Y=R_{n-1,1}\cap C_n$ of $R_{n-1,1}$. Since $Y\supset P_{n-2,1}\cdot U_{n-1}\cdot E_{n}$, by the induction in stages, it suffices to find a short exact sequence of representations of $Y$,
\[
 0\lra \gamma^{\flat}\lra \gamma\lra \gamma^{\sharp}\lra 0
\]
where $\gamma:=\mathcal{S}\mathrm{Ind}_{C_{n-1}\cdot (E_n\cap U_{n})}^{Y}(\tau\otimes 1)$,
\[\gamma^{\flat}:=\mathcal{S}\mathrm{Ind}_{M_{n-1}\cdot U_{n-1}\cdot E_{n}}^{Y}(\tau|_{M_{n-1}}\otimes 1\otimes \psi_n),\] 
and $\gamma^{\sharp}$ admits a filtration from which the filtration of $\sigma^{\sharp}$ can be induced. 

The $\gamma$ can be realized as the space of Schwartz functions from $E_{n}\cap \GL_n$ to the underlying space of $\tau$. Let $U_Y$ be the unipotent radical of $Y$, then  
$Y=\mathrm{GL}_{n-1}\cdot U_Y$.  Decompose 
\[U_Y= \left(F_n\cdot U_{n-1}\right)\cdot V_n.\] 
For $y\in Y$, one can write $y=y_1\cdot y_2\cdot y_3$ with $y_1\in \mathrm{GL}_{n-1}\times \GL_1$, $y_2\in F_n\cdot U_{n-1}$, and  $y_3\in V_n$. Note that the group $V_n$ is abelian and is isomorphic to $\mathbb{R}^{n-1}$ via the exponential map. Hence, one can apply the Fourier transform on $V_n$. 

Let $f \in \mathcal{S}(V_n,\tau)$. The action of $\gamma(y)$ is given by 
\[(\gamma(y)f)(x)=\big\vert\det\big(\mathrm{Ad}(y_1')|_{V_n}\big)\big\vert^{-\frac{1}{2}}\tau(y_1')f(x')\]
where $y_1'$ and $x'$ are determined by 
 $xy=y_1'\cdot y_2'\cdot x'$ with 
 \[x'\in V_n,\  y_1\in (\mathrm{GL}_{n-1}\times \GL_1),\ \text{and}\ y_2'\in   \left(F_n\cdot U_{n-1}\right).\]  
By direct computation, one has $y_1'=y_1$, and $x'=(y_1^{-1}xy_1)\cdot y_3$.

Apply the Fourier transform to $x\in V_n$, assume that $\xi\in (\mathbb{R}^*)^n$ is the dual variable of $x$.  One can get the representation $\widehat{\gamma}$ on $\mathcal{S}((\mathbb{R}^*)^n,\tau)$, which is given by
\[(\widehat{\gamma}(y)h)(\xi)=\big\vert\det\big(\mathrm{Ad}(y_1)|_{V_n}\big\vert^{\frac{1}{2}} \tau(y_1) h(\mathrm{Ad}(y_1)^{-1}\xi)\psi\big(\xi(\mathrm{Ad}(y_1)(y_3))\big).\]

Note that the $0\in (\mathbb{R}^*)^n$ is fixed by the action of $Y$ under $\widehat{\gamma}$.  Let us consider the subrepresentation $\widehat{\gamma}|_{(\mathbb{R}^*)^n\setminus \{0\}}$ of $Y$ consisting of Schwartz sections supported on $(\mathbb{R}^*)^n\setminus \{0\}$. We \textbf{claim} that it is isomorphic to $\gamma^{\flat}$.

Let $\Omega_i$ be the set of $\mathrm{GL}_{n-1}\times 1$ defined by
\[\Omega_i=\left\{ a_{i,\xi}:=\begin{pmatrix}\ \begin{matrix} 
   0 & 0 & I_{n-i-2}\\ I_{i} & 0 & 0  \\  \hline \multicolumn{3}{c}{\xi} \end{matrix}\ 
\end{pmatrix}\ \middle|\ \xi=[\xi_1,\dots,\xi_{n-1}], \ \text{with}\ \xi_{i+1}\neq 0\right\}\]
for $ 0\leq i\leq n-2$. Then $\mathrm{GL}_{n-1}\times \GL_1=\mathop{\bigcup}_i P_{n-2,1}\cdot \Omega_i$, and the underlying space of $\gamma^{\flat}$ is spanned by $\CS(\Omega_i,\tau), 0\leq i\leq n-2$.

Over $\Omega_i$, we define the isomorphism $\CT_i$ from $\mathcal{S}(\Omega_i,\tau)$ to the Schwartz functions $\mathcal{S}((\mathbb{R}^*)^i\times (\mathbb{R}^*\backslash \{0\})\times (\mathbb{R}^*)^{n-i-2},\tau)$ as follows,
\[\mathcal{T}_i(f)(\xi):=|\xi_{i+1}|^{-\frac{1}{2}}\tau(a_{i,\xi})^{-1}f(a_{i,\xi}).\]
The map $\mathcal{T}:=\bigcup_i \mathcal{T}_i$ defines an isomorphism from the underlying space of $\gamma^{\flat}$ to the underlying space of $\widehat{\gamma}|_{(\mathbb{R}^*)^n\setminus \{0\}}$.

Let us verify that $\mathcal{T}$ is actually a $Y$-morphism. Over $\Omega_i$,
\begin{align*} (\mathcal{T}\circ \gamma^{\flat}(y)\circ \mathcal{T}^{-1}f)(\xi) = &\  |\xi_{i+1}|^{-\frac{1}{2}} \big\vert\det\big(\mathrm{Ad}(y_1'')|_{(\mathfrak{gl}_{n-1}\times \mathfrak{gl}_1)/\mathfrak{p}_{n-2,1}}\big\vert^{-\frac{1}{2}} |\xi_{i+1}''|^{\frac{1}{2}} \\ & \cdot \tau(a_{i,\xi})^{-1}\tau( y_1'' )\tau(a_{i,\xi''})  \cdot \psi_n(\mathrm{Ad}(a_{i,\xi}\cdot y_1) y_3)f(a_{i,\xi''})\end{align*}
where $y_1''$ and $\xi''$ are determined by $y_1''\cdot a_{i,\xi''}=a_{r,\xi}\cdot y_1$ for some $y_1''\in P_{n-2,1}$. Hence, $\tau(a_{i,\xi})^{-1}\tau( y_1'' )\tau(a_{i,\xi''})=\tau(y_1)$, and $\xi''=\mathrm{Ad}(y_1)(\xi)$.  Therefore, 
$$\gamma^{\flat}(y)=\mathcal{T}^{-1}\circ\widehat{\gamma}(y)\circ\mathcal{T}$$
for every $y\in Y$.

At the point $0\in (\mathbb{R}^*)^n$, by Borel's lemma, one can get a filtration of $\gamma^{\sharp}$ with successive quotients as $(|(\det)_{\GL_{n-1}}|^{\frac{1}{2}}\cdot (\det)_{\GL_1}|^{-\frac{1}{2}} )\tau\otimes  \mathrm{Sym}^i(\mathfrak{v}_n)^{\vee}$ of $Y$, where $\mathfrak{v}_n$ carries the adjoint representation of $Y$. By induction in stages, one can get the statement.
\end{proof}

Inductively, we can apply the BZ-filtration of $\tau|_{M_n}$ and get a filtration of $I(M(\tau|_{M_n}))$.

\subsubsection{Spectrum decomposition of $\overline{M}(I(\sigma))$}

Recall that $A_n$ denotes the $n\times n$-matrix with $1$ on the anti-diagonal entries and $0$ elsewhere. To simplify the notation, for $g\in \GL_n$, let $\widetilde{g}$ denote $A_n\cdot (g^{-1})^t\cdot A_n$; for $y\in \mathfrak{gl}_n$, let $\widetilde{y}$ denote $A_n\cdot (-y^t)\cdot A_n$; for $x=[x_1,\dots,x_n]^t\in \mathbb{R}^n$, let  $\widetilde{x}$ denote $[-x_n,\dots,-x_1]$.

Recall that $\phi_n$ is the unitary character of $E_n$ defined as $\phi_n(x):=\psi(\langle x\cdot f_1,e_2+f_2\rangle)$. 
We define $\psi_n^{(i)}(x):=\psi(\langle x\cdot f_1,f_i\rangle)$ for $2\leq i\leq n$.

Let us introduce several useful intermediate subgroups. Let $Y=\overline{U}_{n-1}\cdot \GL_{n-1}\cdot\GL_1^{\triangle}\cdot E_n$  be the subgroup of $R_n$, where $\GL_1^{\triangle}:=\left\{\left(\begin{smallmatrix} t^{-1}I_n &  \\ & tI_n \end{smallmatrix}\right)\ \middle|\ t\in \mathbb{R}^*\right\}$. Consider the action of $Y$ on a subset of $E_n^*$,
\[(V_n^*\setminus\{0\})\times F_n^*,\] 
let $Y_1=\overline{U}_{n-1}\cdot M_{n-1}\cdot \GL_1^{\triangle}\cdot E_n$. Under the standard basis, its Lie algebra 
\[\mathfrak{y}_1=\left\{\begin{pmatrix}
-s & \widetilde{a}  & \widetilde{b} & 0\\ 0 & \widetilde{c} & 0 & b\\ 0 & d & c & a \\ 0 & 0 & 0 & s
\end{pmatrix}\in \mathfrak{g}_n\ \middle| \ c=\begin{pmatrix}
* & * \\ 0 & s
\end{pmatrix}, s\in \mathbb{R}  \right\}.\] 
One has
\begin{equation}\label{R-space}
   (V_n^*\setminus\{0\})\times F_n^*\simeq (\psi_n+F_n^*)\times_{Y_1} Y. 
\end{equation}

Let $\Xi$ denote the set $\psi_n+F_n^*$. Then $Y_1$ has two open orbits in $\Xi$ and one closed orbit. 
\begin{itemize}
\item[(i)] Open orbits: the orbit of $\phi_n$,  denoted by $\Xi_o^+$,  and the orbit of $\phi_n^{-}:=-\psi_n^{(2)}+\psi_n$, denoted by $\Xi_o^-$. Let $Y_2$ (resp. $Y_2^-)$ be the stabilizer of $\phi_n$(resp. $\phi_n^-)$ in $Y_1$.
\item[(ii)] Closed orbit: the orbit of $\psi_n$. Let $Y_3$ be the stabilizer of $\psi_n$ in $Y_1$. Then $Y_3=\overline{U}_{n-2}\cdot M_{n-1}\cdot \GL_1^{\triangle} \cdot V_n$.
\end{itemize}

The Lie subalgebra $\mathfrak{y}_2$ of $Y_2$ is 
\[ \left\{\begin{pmatrix}
0 & \widetilde{a} & \widetilde{b} & 0\\ 0 & \widetilde{c} & 0 & b\\ 0 & d & c & a \\ 0 & 0 & 0 & 0
\end{pmatrix}\in \mathfrak{y}_1\ \middle| \ c=\begin{pmatrix}
* & e \\ 0 & 0
\end{pmatrix}, e\in \mathbb{R}^{n-2}, d=\begin{pmatrix}
-e & *\\ 0 & -\widetilde{e}
\end{pmatrix}\right\}.\]
We embed $P_{n-2,1}$ into $Y_2$ as the Lie subgroup corresponding to the Lie subalgebra
\[ \left\{\begin{pmatrix}
-s & 0 & 0 & 0\\ 0 & \widetilde{c} & 0 & 0\\ 0 & d & c & 0 \\ 0 & 0 & 0 & s
\end{pmatrix}\in \mathfrak{y}_1\ \middle| \ c=\begin{pmatrix}
* & e \\ 0 & s
\end{pmatrix}, e\in \mathbb{R}^{n-2}, d=\begin{pmatrix}
-e & 0\\ 0 & -\widetilde{e}
\end{pmatrix}\right\}.\]

Given a representation $\sigma$ of $P_{n-2,1}$, by extension of $\phi_n$ and then trivial extension, one can obtain a $Y_2$-representation $\sigma\otimes \phi_n$, and let
\[\SInd_{Y_2}^{Y_1}(\sigma\otimes\phi_n)\] be the Mackey induction of $Y_1$. 

Similarly, one can also define $=\SInd_{Y_2^-}^{Y_1}(\sigma\otimes\phi_n^-)$.

Let $Y_0=\overline{U}_{n-1}\cdot M_{n-1}\cdot \GL_1^{\triangle} \cdot V_n$. By induction in stages, one has  
\[\ov{M}(I(\sigma))\simeq \SInd_{Y_1}^{M_n}\big(\SInd_{Y_0}^{Y_1}(\sigma\otimes \psi_n)\big),\]
where $\psi_n$ is a character of $V_n\simeq V_n$, and $\sigma\otimes \psi_n$ is regarded as the representation of $Y_0$ by trivial extension to $\overline{U}_{n-1}$.

\begin{proposition}\label{open-cts}
Retain the notation as above, the representation $\eta:=\SInd_{Y_0}^{Y_1}(\sigma\otimes \psi_n)$  of $Y_1$ can be realized as the space of Schwartz sections of a tempered bundle over the $Y_1$-space $\Xi$.

Moreover, the $Y_1$-representation $\eta$ has a filtration 
\[0\to \eta^{\flat}\to \eta\to \eta^{\sharp}\to 0,\] such that $\eta^{\flat}$ is isomorphic to  the direct sum of Mackey induction 
\[\SInd_{Y_2}^{Y_1}(\sigma\otimes\phi_n)\oplus \SInd_{Y_2^-}^{Y_1}(\sigma\otimes\phi_n^-),\] and $\eta^{\sharp}$ has an infinite filtration with successive quotients
\[\SInd_{Y_3}^{Y_1}\left(\sigma\otimes  \mathrm{Sym}^i(\mathfrak{s})\otimes \psi_n\right),\ \forall i\in \mathbb{Z}_{\geq 0}\]
where $\mathfrak{s}$ is the $Y_3$-representation on the complexified of the Lie subalgebra 
\[{\tiny \left\{\begin{pmatrix}
0 & 0  & \widetilde{b} & 0\\ 0 & 0 & 0 & b\\ 0 & 0 & 0 & 0 \\ 0 & 0 & 0 & 0
\end{pmatrix}\in \mathfrak{y}_1\ \middle| \ b=\begin{pmatrix}
*  \\ 0_{(n-2)\times 1}
\end{pmatrix} \right\}.\tiny}\]
\end{proposition} 

\begin{proof} 
The representation $\eta$ can be realized as the space of Schwartz functions from $F_n$ to the underlying space of $\sigma$. Let $g\in Y_1$, and write $g=upvt$ with $u\in \overline{U_{n-1}}$, $p\in M_{n-1}\cdot \GL_1^{\triangle}$, $v\in V_n$, $t\in F_n$. Given $f\in \mathcal{S}(F_n, \sigma)$,  the action of $\eta(g)$ is given  by
\begin{align*}(\eta(g)f)(x) = & \ \big\vert\det\left(\mathrm{Ad}(p)|_{\mathfrak{y}_1/\mathfrak{y}_0}\right)\big\vert^{-\frac{1}{2}}\cdot \psi_n\left(-\mathrm{Ad}(u)(x)+x+\mathrm{Ad}(up)(v)\right)\\ & \cdot \sigma(p)f(\mathrm{Ad}(p)^{-1}(x)+t).\end{align*}

Applying the Fourier transform to $x$ with respect to $\psi^{-1}$, letting $\xi$ be the real dual variable of $x$ and identify $\xi$ with the  unitary character $\left(x\mapsto \psi(\langle \xi,x\rangle)\right)$ of $F_n$, one has
\begin{align*}& \mathcal{F}_x\circ \eta(g)\circ \mathcal{F}_x^{-1}(\widehat{f})(\xi)\\ = &\  \big\vert\det\left(\mathrm{Ad}(p)|_{\mathfrak{y}_1/\mathfrak{y}_0}\right)\big\vert^{\frac{1}{2}}\cdot \psi\left(\langle\xi,\mathrm{Ad}(p)(t)\rangle\right)\cdot \psi_n\left(\mathrm{Ad}(up)(v)+\mathrm{Ad}(up)(t)-\mathrm{Ad}(p)(t) \right) \\ &
\cdot \sigma(p)\widehat{f}\left(\mathrm{Ad}(p)^{-1}(\xi)-\mathrm{Ad}(p)^{-1}(\psi_n)+\mathrm{Ad}(up)^{-1}(\psi_n)\right).\end{align*}
Since the action of  $\widehat{\eta}:= \mathcal{F}_x\circ \eta\circ \mathcal{F}_x^{-1}$ on the variable $\xi\in F_n^*$ matches with the action of $Y_1$ on $\psi_n+F_n^*$, the $\widehat{\eta}$ can be realized as the Schwartz sections of the bundle over $\psi_n+F_n^*$, that is $\mathcal{S}(\Xi,\sigma)$.

The representation $\beta=\SInd_{Y_2}^{Y_1}(\sigma\otimes\phi_n)$ can be realized as the space of Schwartz functions from 
\[{\tiny \Xi_Y:=\left\{a(s,y):= \begin{pmatrix}
1 & 0 & 0 & 0 \\ 0 & I_{n-1} & 0 & 0\\ 0 & B_y & I_{n-1} & 0 \\ 0 & 0 & 0 & 1
\end{pmatrix}\cdot  \begin{pmatrix}  \frac{1}{s}I_n & 0 \\ 0 &   sI_n \end{pmatrix} \ \middle|\  \renewcommand{\arraystretch}{1.3} \begin{array}{l}   B_y= \begin{pmatrix}y & 0_{(n-2)\times (n-2)} \\ 0 & \widetilde{y} \end{pmatrix}\\ s\in \mathbb{R}^{\times}, y=[y_1 \dots y_{n-2}]^t\in \mathbb{R}^{n-2}\end{array}\right\}}\] to the underlying space of $\sigma$.

Given $g=upvt\in  Y_1$ as above,  decompose $p=l_0\cdot u_0\cdot a(s_0,0)$, where $l_0\in \GL_{n-2}$, $u_0\in F_{n-1}$ and $s_0\in \mathbb{R}^{\times}$. Write $u=a(1,y_{u})\cdot u'$ with $y_u\in \mathbb{R}^{n-2}$ and $u'\in \overline{U_{n-2}}$. Write 
\[{\tiny u_0=\begin{pmatrix}
1 & 0 & 0 & 0 & 0 & 0\\ 0 & 1 & \widetilde{y_{u_0}} & 0 & 0 & 0\\   0 & 0 & I_{n-2}  & 0 & 0 & 0\\  0 & 0 &  0 & I_{n-2} & y_{u_0} & 0\\  0 &  0 & 0 & 0 & 1  & 0 \\ 0 &  0 & 0 & 0 &  0 & 1
\end{pmatrix}}.\]
The action of $\beta(g)$ is given by
\begin{align*}(\beta(g)h)(a(s,y))=& \big\vert\det\left(\mathrm{Ad}(p)|_{\mathfrak{y}_1/\mathfrak{y}_0}\right)\big\vert^{\frac{1}{2}} \phi_n\left(\mathrm{Ad}(a(s,y)g)(vt)\right)\\ & 
\cdot  \sigma(l_0u_0)h\left(  a(ss_0,l_0^{-1}(y+y_u)+y_{u_0})\right) \end{align*}

we define a map $d: \Xi_Y\to \Xi_o^+$ by 
\[d(a(s,y)):=a(s,y)^{-1}(\phi_n)=\psi_n+\frac{1}{s^2}\left(\psi_n^{(2)}-\sum_{i=1}^{n-2}y_{n-i-1}\psi_n^{(i+2)}\right).\] 
Define $\mathcal{T}:\mathcal{S}(\Xi_o^+,\sigma)\to \mathcal{S}(\Xi_Y,\sigma)$ by 
\[\mathcal{T}(f)(a(s,y)):=\sigma(a(s,0)) f(d(a(s,y))).\] 
Then $\mathcal{T}$ is an intertwining operator from $\widehat{\eta}|_{\Xi_o^+}$ to $\SInd_{Y_2}^{Y_1}(\sigma\otimes \phi_n)$. Similarly, $\widehat{\eta}|_{\Xi_o^-}\simeq \SInd_{Y_2^-}^{Y_1}(\sigma\otimes \phi_n)$. Let $\eta^{\flat}=\widehat{\eta}|_{\Xi_o^+\cup \Xi_o^-}$. 
By Borel's lemma, $\eta/\eta^{\flat}$ has an infinite decreasing filtration with successive quotients as the statement.

\end{proof}

Consequently, by isomorphism~\eqref{R-space}, $\ov{M}(I(\sigma))$ can be realized as Schwartz sections of a tempered bundle $\CE$ over 
\[
X:=\big((V_n^*\setminus\{0\})\times F_n^*\big)\times_Y R_{n-1,1}.
\]
Note that $R_{n-1,1}$ has a right action on $E_n^*$. This action will induce a Nash submersion
\[
\varphi:\big((V_n^*\setminus\{0\})\times F_n^*\big)\times_Y R_{n-1,1}\lra E_n^*\setminus \{0\}.
\]
In addition, by Proposition~\ref{open-cts}, $\fke_n$-action on $\CS(X,\CE)$ is given by
\begin{equation}\label{E_n-act}
  (\xi\cdot f)(x):=d\psi(1)\varphi(x)(\xi)\cdot f(x),\xi\in \fke_n
\text{ and } f\in\CS(X,\CE).   
\end{equation}


\subsection{Bernstein-Zelevinsky filtration of the degenerate principal series}
Let $n_i, 1\leq i\leq m$ be positive integers such that $\sum_{ i=1}^m n_i=n$. The following theorem concerns the infinitesimal characters of irreducible subquotients occurring in the BZ-filtration of some degenerate principal series. It will be used in Theorem~\ref{ext-van-thm}.
\begin{theorem}\label{central_char_of_fil}
Let $\pi=\prod_{i=1}^m\chi_{r_i,s_i}$ be a representation
of $\GL_n$, where $\chi_{r_i,s_i}$ is a character of $\GL_{n_i}$. Then there exists a  filtration of $\pi|_{M_n}$ with successive quotients being isomorphic to  $I^{k-1}E(\prod_{i=1}^m\tau_i)$ for some non-negative integer $k$, where $\tau_i$ is either 
 \begin{itemize}
\item[(i)] a $\GL_{n_i}$-representation $|\det|_{\mathbf{k}}^{\frac{1}{2}}\cdot \chi_{r_i,s_i}\otimes_{\BR} \mathrm{Sym}^l(\mathbf{k}^{n_i})$ for some $l\in \mathbb{N}$, where $\mathbf{k}^{n_i}$ is the standard representation of $\GL_{n_i}$; or
\item[(ii)] a $\GL_{n_i-1}$-representation $\chi_{r_i,s_i}|_{\GL_{n_i-1}}$. 
\end{itemize}
And each $\prod_{i=1}^m\tau_i$ satisfying (i) and (ii) shows up exactly once in the successive quotients. 
Moreover, by taking refinement to break the finite length representation $\prod_{i=1}^m\tau_i$ into irreducible ones, one can get the BZ-filtration of $\pi|_{M_n}$.
\end{theorem}

\begin{proof}

We prove it by induction on $m$.
For $m=1$, it is trivially true. Assume that it holds for $m-1$, let us show it for $m$. Write $\pi$ as \[\mathrm{Ind}_{P_{n-n_m,n_m}}^{\GL_n}(\widetilde{\pi}\boxtimes \chi_{r_m,s_m}), \text{ where } \widetilde{\pi}:=\prod_{i=1}^{m-1}\chi_{r_i,s_i}.\]

As the proof of Theorem \ref{ind_BZ}, $\pi|_{M_n}$ admits a filtration 
$$0\to \pi^{\flat}\to \pi|_{M_n}\to \pi^{\sharp}\to 0.$$ 
By induction on $m$, one can get a filtration of $\widetilde{\pi}|_{M_{n-n_m}}$ as the statement. Therefore, one can get a filtration of 
$$\pi^{\flat}=|\det|_{\mathbf{k}}^{\frac{1}{2}}\cdot\chi_{r_i,s_i}\mind \widetilde{\pi}|_{M_{n-n_m}}$$ 
as the statement by Lemma~\ref{inductive_arg} and Lemma~\ref{open_orbit}. Moreover, as the proof~\eqref{closed_cal} of Theorem \ref{ind_BZ}, one can get a filtration of $\pi^{\sharp}$ as the statement. This finishes the proof of the statement.
\end{proof}

We give a concrete example, which will be used in the Example \ref{ext_nonzero}.

\begin{example}\label{exa-BZ}
Let $\pi$ be the representation $(\chi_{r_1,s_1})_{\GL_2}\times (\chi_{r_2,s_2})_{\GL_2}$ of $\GL_4(\mathbb{C})$. Then by Theorem \ref{central_char_of_fil}, $\pi$ admits a level $\leq 1$ BZ-filtration,
\[\pi=\sigma_0\supset \sigma_1 \supset \sigma_2 \supset 0 \]
with 
\begin{itemize}
\item[(i)] an infinite decreasing filtration  \begin{align*}\sigma_0=\sigma_{0,0}^{0}\supset\sigma_{0,0}^{1}\supset \dots\supset \sigma_{0,0}^{i_0}=\sigma_{0,1}^{0}\supset\sigma_{0,1}^{1}\supset\dots\supset \sigma_{0,1}^{i_1}=\sigma_{0,2}^{0}\supset\sigma_{0,2}^{1}\supset \dots  \supset \sigma_1,\end{align*}
where $\sigma_0/\sigma_1\simeq \varprojlim_j  \sigma_0/\sigma_{0,j}^{0}$, and 
\[\sigma_{0,j}^{0}/\sigma_{0,j+1}^{0}\simeq E\left(\big((\chi_{r_1,s_1})_{\GL_2} |\det|\otimes_{\mathbb{R}} \mathrm{Sym}^j(\mathbb{C}^2)\big)\times (\chi_{r_2,s_2})_{\GL_1}\right),\] and  $\sigma_{0,j}^{0}\supset\sigma_{0,j}^{1}\supset\dots\supset\sigma_{0,j}^{i_j}=\sigma_{0,j+1}^{0}$ is a finite refinement with irreducible successive quotients.
\item[(ii)] Similar to (i), an infinite decreasing filtration  
\[\sigma_1=\sigma_{1,0}\supset\sigma_{1,0}^{1}\supset \dots\supset \sigma_{1,0}^{*}=\sigma_{1,1}^{0}\supset\sigma_{1,1}^{1}\supset\dots\supset \sigma_{1,1}^{*}=\sigma_{1,2}^{0}\supset\sigma_{1,2}^{1}\supset \dots  \supset \sigma_1,\]
where $\sigma_1/\sigma_2\simeq \varprojlim_j  \sigma_1/\sigma_{1,j}^{0}$, and \[\sigma_{1,j}^{0}/\sigma_{1,j+1}^{0}\simeq E\left((\chi_{r_1,s_1})_{\GL_1}\times \big((\chi_{r_2,s_2})_{\GL_2} |\det|\otimes_{\mathbb{R}} \mathrm{Sym}^j(\mathbb{C}^2)\big)\right),\] 
and  $\sigma_{1,j}^{0}\supset\sigma_{1,j}^{1}\supset\dots\supset\sigma_{1,j}^{*}=\sigma_{1,j+1}^{0}$ is a finite refinement with irreducible successive quotients.
 
\item[(iii)] $\sigma_2\simeq IE\big((\chi_{r_1,s_1})_{\GL_1}\times (\chi_{r_2,s_2})_{\GL_1}\big)$, and a finite refinement of  $\sigma_2\supset 0$ with irreducible successive quotients.

\end{itemize}

\end{example}

\subsection{Opposite Bernstein-Zelevinsky filtration}
The group $\GL_n$ has an outer automorphism given by conjugate inversion, which leads us to consider restricting Casselman-Wallach representations to the opposite mirabolic subgroup $\ov{M_n}$. This should give us more information than just considering the restriction to the mirabolic subgroup. On the other hand, we observe that such an outer automorphism will induce an involution on the category of Casselman-Wallach representations, which is called MVW-involution. By Harish-Chandra character theory, we have the following well-known fact.
\begin{lemma}\label{lem-MVW}
    Let $\pi$ be an irreducible representation of $\GL_n$. Its MVW-involution is given by $\pi^{MVW}(g):=\pi(g^{-t})$. Then $\pi^{\vee}\simeq \pi^{MVW}$.
\end{lemma}
\begin{remark}
    Observe that for other classical groups $G_n$, the opposite mirabolic subgroup $\ov{M_n}$ is inner conjugate to $M_n$. Hence, restriction to the opposite mirabolic subgroup will not provide more information.
\end{remark}

An axiomatic definition about \textbf{opposite BZ-filtration} is also one that we appreciate.
\begin{definition}\label{def-oppo-BZ}
Let $\sigma$ be a representation of $\ov{M_n}$. We call the following datum \textbf{a level $\leq 1$ opposite BZ-filtration} of $\sigma$: A level $\leq 1$ filtration of $\sigma$ as \ref{def_fil}
such that  
\begin{itemize}
    \item $\sigma_{i,j}/\sigma_{i,j+1}$ is isomorphic to  $\ov{I}^{k_i}\ov{E}(\pi_{i,j})$ for some $k_i$ (dependent on $i$ but independent on $j$) and irreducible representations $\pi_{i,j}$ of $\GL_{n-k_i-1}$, and 
     \item The real part of the central characters satisfies $\mathrm{Re}(\omega_{\pi_{i,j}})\geq \mathrm{Re}(\omega_{\pi_{i,j+1}})$ for every $j$. And for every $c\in \mathbb{R}$, there are finitely many $j$ such that $\mathrm{Re}(\omega_{\pi_{i,j}})\geq c$. 
\end{itemize} 
 
For $r\geq 2$, we call the following datum \textbf{a level $\leq r$ opposite BZ-filtration} of $\sigma$: A level $\leq r$ filtration of $\sigma$ as \ref{def_fil} such that the filtration on each $\sigma_{i,j}/\sigma_{i,j+1}$ is a level $\leq r-1$ opposite BZ-filtration, and

\begin{itemize} 
     \item Let $\Omega_{i,j}$ denote the set of real parts of the central characters of the irreducible successive quotients in $\sigma_{i,j}/\sigma_{i,j+1}$. Then for each $i,j$, $\max \Omega_{i,j}\geq \max \Omega_{i,j+1}$. Moreover, for every $c\in \mathbb{R}$, there exists only finitely many element $j$ with $\max \Omega_{i,j}\geq c$.
\end{itemize}

We say that a representation $\sigma$ of $\ov{M_n}$ admits an opposite BZ-filtration if $\sigma$ admits a level $\leq r$ opposite BZ-filtration for some finite $r\in \mathbb{Z}_{>0}$.
\end{definition}

Let $\pi$ be a Casselman-Wallach representation of $\GL_n$. Thus, by Theorem~\ref{have_B-Z_fil}, $\pi^{\vee}|_{M_n}$ admits a BZ-filtration of level $\leq n$. We realize $\pi$ and $\pi^{\vee}$ on  the same vector space by Lemma~\ref{lem-MVW}. Then the filtration is stable under $\pi(\ov{M_n})$-action since $(M_n)^{-t}=\ov{M_n}$. Moreover, suppose some successive subquotient is isomorphic to $I^{k-1}E(\tau)$ for some positive integer $k$ and irreducible $\GL_{n-k}$-representation $\tau$ under $\pi^{\vee}$-action. Then under $\pi$-action, it is isomorphic to $\ov{I}^{k-1}\ov{E}(\tau^{\vee})$. Consequently, we get the following result.

\begin{proposition}\label{have_oppoBZ-fil}
    Let $\pi$ be a Casselman-Wallach representation of $G_n$. Then $\pi|_{\ov{M_n}}$ admits an opposite BZ-filtration of level $\leq n$.
\end{proposition}
\begin{remark}\label{order_rem}
    Besides MVW-involution, one can write down an opposite BZ-filtration by a similar argument in section~\ref{BZ_gl-sec}. However, the order of these two filtrations is not always identical, see the example of self-dual discrete series in section~\ref{LLC-sec}.
    
\end{remark}
Observe the above proposition, we get a remark on Theorem~\ref{have_B-Z_fil}.
\begin{remark}
    In the proof of Theorem~\ref{have_B-Z_fil}, the BZ-filtration of irreducible representation comes from a subrepresentation structure. Actually, we can also realize the irreducible representation $\pi$ as a quotient of some Casselman-Wallach representation $I$ equipped with a BZ-filtration. Then $\pi^{\vee}\hookrightarrow I^{\vee}$ will inherit an opposite BZ-filtration. Thus we get a BZ-filtration on $\pi$. These two filtrations do not always coincide.
\end{remark}

\section{Twisted homology and Highest Derivative}\label{highest der}
The highest derivative is an important tool to study the branching law of general linear groups, see \cite[Theorem 4.3]{PWZ} for example. In this section, we show that the highest derivative of $\pi$ coincides with the bottom layer of the BZ-filtration. With such an idea, we can calculate the highest derivative for parabolically induced representations and prove some results similar to the $p$-adic case, which significantly extends the result of \cite[Theorem B]{AGS15b}.
\subsection{Twisted homology and highest derivative of $\GL_n$}
The following result is fundamental to the entire article. Let $\sigma$ be a representation of $M_{n-1}$, hence $I(\sigma)$ is a representation of $M_n$. We interpret $I(\sigma)$ as Schwartz sections of a tempered bundle $\CE$ over $H_{n,2}\backslash M_n$. 
\begin{proposition}\label{homology prop}
    We have the following result about the Lie algebra homology of $I(\sigma)$.
    $$
    \mathrm{H}_i(\fkv_n,I(\sigma)\otimes (-\psi_n))=\left\{\begin{aligned}
       & \sigma, \quad \text{ if } i=0\\
        & 0, \quad \text{ otherwise}
    \end{aligned}
    \right.
   $$
\end{proposition}
In the proof, we will use a variant of the following lemma concerning the homology of a family of representations.
\begin{lemma}[\cite{AGS15b}, Lemma 6.2.2]\label{lem:SimpHo}
Let $X$ be a Nash manifold and $\fkv$ be a complex abelian Lie algebra. Let $\varphi:X \to \fkv^*$ be a Nash map. This defines a map $\chi: \fkv \to \CT(X)$, where $\chi(v)(x)= \varphi(x)(v)$. Consider an action of $\fkv$ on $\CS(X)$ defined by $\pi(v)(f):= \chi(v) \cdot f$. Suppose that $0 \in \fkv^*$ is a regular value of $\varphi$. Then
\begin{enumerate}
\item $\rmh_i(\fkv,\CS(X))=0$ for $i>0$.
\item  Let $X_0:=\varphi^{-1}(0)$, which is smooth. Let $r$ denote the restriction map $r:\CS(X) \to \CS(X_0)$. Then $r$ induces an isomorphism $\rmh_0(\fkv,\CS(X)) \stackrel{\sim}{\lra} \CS(X_0)$.
\end{enumerate}
\end{lemma}
Now we introduce a variant of this lemma, which will be used in the forthcoming proof.
\begin{corollary}[Bundle version]\label{bun ver}
Let $X$ be a Nash manifold and $\CE$ be a tempered Fr\'echet bundle over $X$. Let $\fkv$ be a complex abelian Lie algebra, and $\varphi:X\to \fkv^*$ be a Nash map. Assume either $0\notin \varphi(X)$ or $0 \in \fkv^*$ is a regular value of $\varphi$. Then, considering the $\fkv$-action on $\CS(X)$ as in Lemma~\ref{lem:SimpHo}, we have:
    \begin{enumerate}
\item $\rmh_i(\fkv,\CS(X,\CE))=0$ for $i>0$.
\item  Let $X_0:=\varphi^{-1}(0)$, which is smooth. Let $r$ denote the restriction map $r:\CS(X,\CE) \to \CS(X_0,\CE)$. Then $r$ induces an isomorphism $\rmh_0(\fkv,\CS(X,\CE)) \stackrel{\sim}{\lra} \CS(X_0,\CE)$.
\end{enumerate}
\end{corollary}
\begin{proof}
    Consider a finite open covering $\{U_\alpha\}_{\alpha\in I}$ of $X$ which trivializes $\CE$. 
    \begin{itemize}
        \item \textbf{Case 1:} $0\in\fkv^*$ is a regular value of $\varphi$.
    \end{itemize}
    When $J$ is a subset of $I$, we define
    \[ \CS_J:=\CS(\bigcap_{j\in J}U_j,\CE) \text{ and } S_J':=\CS\big(\bigcap_{j\in J}(U_j\cap X_0),\CE|_{X_0}\big).
    \]
    Since the Schwartz sections compose a co-sheaf by Proposition~\ref{co-sheaf prop}, we have the following \v{C}ech resolution of $\CS(X,\CE)$ and $\CS(X_0,\CE|_{X_0})$
    \[
    \begin{tikzcd}
      \ar[r]  & \bigoplus_{|J|=k} \CS_J\ar[r]\ar[d,"r_{k-1}"]& \dots\ar[r]& \bigoplus_{|J|=1} \CS_J\ar[r]\ar[d,"r_1"]& \CS(X,\CE)\ar[r]\ar[d,"r"]& 0\\
        \ar[r]  &\bigoplus_{|J|=k} \CS_J'\ar[r]& \dots\ar[r]& \bigoplus_{|J|=1} \CS_J'\ar[r]& \CS(X_0,\CE|_{X_0})\ar[r]& 0,
    \end{tikzcd}
    \]
    where $r_k$ is the sum of restriction map $r_J$ on each $S_J$.  For each subset $J$, by Lemma~\ref{lem:SimpHo}, we have 
    \[
    \rmh_i(\fkv,S_J)=0, i>0 \text{ and } r_J:\rmh_0(\fkv,S_J) \stackrel{\sim}{\lra} S_J'.
    \]
    Thereby, the upper horizontal line is an acyclic resolution, and after taking $\rmh_0$ on the upper horizontal line, we get the bottom horizontal line. The corollary hence follows.
        \begin{itemize}
        \item \textbf{Case 2:} The image of $\varphi$ does not contain $0\in\fkv^*$.
    \end{itemize}
Take a direct sum decomposition of $\fkv = \oplus_{j=1}^{\ell} \fkv_j$, where $\fkv_j$ is a one-dimensional subalgebra of $\fkv$. Let $\varphi_j$ be the composition of $\varphi$ and the restriction map:
    \[
     \varphi_j: X \stackrel{\varphi}{\lra} \fkv^* \lra \fkv_j^*.
    \]
Since $0 \notin \varphi(X)$, $\{\varphi_j^{-1}(\fkv_j \setminus \{0\})\}_{j=1}^{\ell}$ is an open covering of $X$. Hence, we can assume that for each $\alpha \in I$, there is an integer $1 \leq j(\alpha) \leq \ell$ such that
\[
U_{\alpha} \subset \varphi_{j(\alpha)}^{-1}(\fkv_{j(\alpha)} \setminus \{0\})
\]
by taking a suitable refinement of $\{U_{\alpha}\}_{\alpha \in I}$. By a similar argument in \textbf{Case 1}, it suffices to prove
\[
\rmh_i(\fkv, \CS(U_{\alpha}, \CE)) = 0
\]
for every $\alpha \in I$ and integer $i$. Furthermore, by the Hochschild-Serre spectral sequence, it suffices to prove
\[
\rmh_i(\fkv_{j(\alpha)}, \CS(U_{\alpha}, \CE)) = 0
\]
for every $\alpha \in I$ and integer $i$. Choosing a trivialization $\CE|_{U_\alpha} \simeq U_{\alpha} \times E$, where $E$ is a Fr\'echet space. It is equivalent to show the Koszul complex is exact:
\[
0 \lra (\fkv_{j(\alpha)} \otimes \CS(U_{\alpha})) \hat{\otimes} E \stackrel{m \otimes \id}{\lra} \CS(U_{\alpha}) \hat{\otimes} E \lra 0,
\]
where $m$ is defined by
\[
m(\xi \otimes f) := \varphi_{j(\alpha)}(\xi) \cdot f \text{ for } \xi \in \fkv_{j(\alpha)}, f \in \CS(U_{\alpha}).
\]
Consequently, $m \otimes \id$ is an isomorphism since for $\xi \neq 0$, $\varphi_{j(\alpha)}(\xi)$ is an everywhere non-zero Nash function.
\end{proof}

\begin{proof}[proof of Proposition~\ref{homology prop}]
   Let $X= H_{n,2}\backslash M_n$. Then we have an open embedding as Nash manifolds
   \[
   \varphi: X\lra \fkv_n^*\qquad x\mapsto {}^x\psi_n.
   \]
   Let $\widetilde{\varphi}:=\varphi-\psi_n$. Then $\widetilde{\varphi}$ is a Nash map on $X$ such that $0$ is a regular value. Moreover, $\widetilde{\varphi}^{-1}(0)=\{\ov{e}\}$, where $\ov{e}$ is the image of the identity element in $H_{n,2}\backslash M_n$. Since $V_n$ is a normal subgroup of $M_n$, the action of $\fkv_n$ on $I(\sigma)=\CS(X,\CE)$ is given by
   \[
   (\xi\cdot f)(x)=\widetilde{\varphi}(x)(\xi)f(x)\text{ for }\xi\in \fkv_n,f\in\CS(X,\CE) .
   \]
   Consequently, by Lemma~\ref{bun ver}, the Proposition follows.
\end{proof}
Proposition~\ref{homology prop} can be applied to show that the twisted homologies of irreducible representations occurring in the BZ-filtration are vanishing. 
\begin{corollary}\label{vanishing}Let $\tau$ be a representation of $\GL_{n-d}$, then
\begin{enumerate}
    \item    $L^iD^k(I^{d-1}E(\tau))=0 $
    for every integer $k$ and $i\geq 1$; 
    \item   $D^d(I^{d-1}E(\tau))=\tau$.
\end{enumerate}
Hence, if $\sigma$ is a representation of $M_n$ having a BZ-filtration, then $L^iD^k(\sigma)=0$ for every integer $k$ and $i\geq 1$.
\end{corollary}
\begin{proof}
    Assertion (2) follows directly from Proposition~\ref{homology prop}. For assertion (1), note that $\Phi$ is an exact functor, which implies $L^iD^k=\Phi\circ L^i\Psi^{k-1},k\geq 2$ and $L^iD^1=0$. Note that when $k=2$, $L^i\Psi(I^{d-1}E(\tau))=0$ for every $i\geq 1$ by Proposition~\ref{homology prop}. We use induction on $k$. Assume the statement holds for $k$, then for $k+1$, consider the Hochschild-Serre spectral sequence
    \[
  E_2^{p,q}:=  L^p\Psi \circ L^q\Psi^{k-1}\Rightarrow L^{p+q}\Psi^k.
    \]
    When $q\neq 0$, it follows from the induction. For $q=0$, we have
    \[
    \Psi^{k-1}(I^{d-1}E(\tau))=\begin{cases}
        I^{d-k}E(\tau)&d>k\\
        E(\tau)&d=k\\
        0& d<k,
    \end{cases}
    \]
    thus $L^p\Psi( \Psi^{k-1}(I^{d-1}E(\tau)))=0$ for $p\geq 1$.
\end{proof}

Recall the depth of $M_n$-representation in Definition~\ref{def-der-gl}. By Proposition~\ref{homology prop}, the representation $I^{k-1}E(\tau)$ is of depth $k$, where $\tau$ is a representation of $\GL_{n-k}$. Let $\sigma$ be a representation of $M_n$ having BZ-filtration. Then the depth of $\sigma$ is the maximal integer $k$ such that $I^{k-1}E(\tau)$ appears in the subquotient of the filtration for some $\tau$. Let $\pi$ be a Casselman-Wallach representation of $\GL_n$ with depth $d$, then \cite[Corollary 3.0.9(1)]{AGS15a} shows that $D^d(\pi)$ is a Casselman-Wallach representation of $\GL_{n-d}$. The following lemma is a direct consequence of Corollary~\ref{vanishing}.

\begin{lemma}\label{bottom fin}
Let $\pi$ be a Casselman-Wallach representation of $\GL_n$ such that $\text{depth}(\pi)=k_0$. Then the number of depth $k_0$ successive quotients in every BZ-filtration equals the length of $D^{k_0}(\pi)$.
\end{lemma}

In the following context, we denote the highest depth terms in the BZ-filtration as \textbf{bottom layer}. The following is the main theorem of this section, showing the highest derivative of product representations. Let $n_i,1\leq i\leq k$ be positive integers and $n:=\sum_{i=1}^k n_i$.
\begin{theorem}\label{high-thm}
    Let $\pi_i$ be Casselman-Wallach representations of $\GL_{n_i}$, then 
    \[
   \operatorname{s.s.} (\pi_1\times\dots\times\pi_k)^-\simeq\operatorname{s.s.} (\p_1^-\times\dots\times\pi_k^-)
    \]
\end{theorem}
\begin{proof}
Set $\pi=\pi_1\times\dots\times\pi_k$. Let $m_i$ be the depth of $\pi_i$ and let $m$ be the depth of $\pi$. Then $m=\sum m_i$ by Lemma \ref{bottom fin}. Assume that the depth $m_i$ successive quotients in BZ-filtration of $\pi_i$ are
$$\left\{I^{m_i-1}E(\tau_{i,j})\ |\ 1\leq j\leq r_i\right\}.$$
By Theorem \ref{ind_BZ}, the depth $m$ successive quotients in BZ-filtration of $\pi$ are $I^{m-1}E(\tau)$, where $\tau$ runs through all the irreducible composition factors of 
\[\left\{\prod_{i=1}^k\tau_{i,j_i}\ |\ 1\leq j_i\leq r_i\right\}.\] 

By Corollary \ref{vanishing}, one has 
\[\operatorname{s.s.} (\pi_1\times\dots\times\pi_k)^-\simeq (\operatorname{s.s.} \p_1^-)\times \dots\times(\operatorname{s.s.} \pi_k^-).\]
Since $(\operatorname{s.s.} \p_1^-)\times \dots\times(\operatorname{s.s.} \pi_k^-)\simeq\operatorname{s.s.} (\p_1^-\times\dots\times\pi_k^-)$, this finishes the proof of the statement.
\end{proof}
\subsection{Twisted homology of $\SO(n,n)$} Like the Casselman-Wallach representation of $\GL_n$, one always hopes that the twisted homology of the nilradical is Hausdorff and its higher homology vanishes. In addition, the information on the derivatives is helpful in obtaining the Euler-Poincar'e characteristic formula of $G_n$, especially when a powerful Kunneth formula is not available in the Archimedean case. Based on the BZ-filtration of $\SO(n,n)$, we have the following result.
\begin{theorem}\label{tw-homo-thm}
    Let $\pi$ be a Casselman-Wallach representation of $\SO(n,n)$. Then $\Upsilon(\pi)$ is Hausdorff and $L^i\Upsilon(\pi)=0$ for every integer $i\geq 1$.
\end{theorem}
\begin{proof}
    By the same argument as \cite[Proof of Theorem A, p50]{AGS15b}, the statement is reduced to the case of principal series.
    
    Let $I$ be a principal series of $G_n$. Following the notation of Section~\ref{BZ-cl-sec}, we realize $I$ as a representation induced from the Siegel parabolic subgroup $C_n$. The $R_{n-1,1}$-orbits on $C_n\backslash G_n$ lead to following short exact sequence
    \[
    0\lra I_0\lra I|_{R_{n-1,1}}\lra I_c\lra 0.
    \]
    We prove the statement for both $I_o$ and $I_c$. By Proposition~\ref{open-dis} and Proposition~\ref{open-cts}, both $I_c$ and $I_o$ admit filtrations of $M_n$. By Lemma~\ref{fil-lem}, it suffices to prove the statement for successive quotients in the filtration. 
    
   \noindent{\textbf{Case 1.} } The successive quotient is isomorphic to $I^{d-1}E(\Ind_{Q_{n-d}}^{G_{n-d}} \tau)$ for some irreducible representation $\tau$ of $\GL_{n-d}$. By Corollary~\ref{bun ver}, we have
      $$
    \mathrm{H}_i(\fke_n,I^{d-1}E(\Ind_{Q_{n-d}}^{G_{n-d}} \tau)\otimes (-\phi_n))=0
   $$
   for every integer $i$ and positive integer $d$.  
   
  \noindent{ \textbf{Case 2.}} The successive quotient is isomorphic to $\ov{M}(I(\sigma))$, where $\sigma$ is a representation of $M_{n-1}$. We follow the notation in subsection~\ref{BZ-cl-sec}. Since the $\fke_n$-action is given by~\eqref{E_n-act}, by Corollary~\ref{bun ver}, we have
   \[
  \rmh_i(\fke_n,\ov{M}(I(\sigma))\otimes(-\phi_n))= \rmh_i(\CS(X,\CE)\otimes(-\phi_n))\simeq\begin{cases}
      \CS(X_0,\CE|_{X_o}) &\text{ for }i=0\\
      0 &\text{ for }i\neq 0,
  \end{cases}
   \]
   where $X_0:=\varphi^{-1}(\phi_n)$. We observe that $X_0=\Omega_0$ as the proof of Proposition \ref{open-cts}, so 
   \[\rmh_0(\mathfrak{e}_n,\ov{M}(I(\sigma))\otimes(-\phi_n))=\SInd_{P_{n-2,1}\cdot \overline{U_{n-2}}}^{P_{n-2,1}\cdot G_{n-2}}(\sigma).\]
\end{proof}

\section{Generic representations of general linear groups}
In this section, we study the irreducible generic representation under Langlands parameterization and the BZ-filtration of relative discrete series.
\subsection{Local Langlands correspondence for $\GL_n$}\label{LLC-sec}
Let $W_{\mathbf{k}}$ be the Weil group of the Archimedean local field $\mathbf{k}$, that is
\[
W_{\mathbf{k}}:=\begin{cases}
    \BC^{\times} &\text{ for } \mathbf{k}=\BC\\
    \BC^{\times}\bigsqcup j\BC^{\times} &\text{ for } \mathbf{k}=\BR,
\end{cases} 
\]
where $j^2=-1$ and $jzj^{-1}=\ov{z}$. The local Langlands correspondence states that there is a one-to-one correspondence between the irreducible representations of $\GL_n(\mathbf{k})$ and isomorphism classes of $n$-dimensional semi-simple $W_{\mathbf{k}}$-representations. 
\begin{itemize}
    \item Let $\mathbf{k}=\BC$. Each irreducible representation of $W_{\mathbf{k}}$ is a character. Its corresponding $\GL_1(\BC)$-representation is the character itself.
    \item Let $\mathbf{k}=\BR$. Each irreducible representation $W_{\mathbf{k}}$ is either a character or a two-dimensional representation taking following form
    \[
    \kappa_{k,s}:=\Ind_{\BC^{\times}}^{W_{\BR}} \chi_{k,s},k\in\BZ_{\geq 1},s\in\BC.
    \] 
    Such a two dimensional representation will correspond to a relative discrete series of $\GL_2(\BR)$ which we will describe more concretely. For a character of $W_{\BR}$, it always descends to a character of $\BR^{\times}$, and its corresponding $\GL_1(\BR)$-representation is the character itself.
\end{itemize}
Let $k\geq 1$ be a positive integer. Consider the reducible principal series $I(k):=\chi_{\epsilon_{k-1},-\frac{k}{2}}\times \chi_{0,\frac{k}{2}} $ of $\GL_2(\BR)$, where $\epsilon_{k-1}$ is the parity of $k-1$. It fits into a short exact sequence
\begin{equation}\label{DS equ}
     0\lra V_k\lra \chi_{\epsilon_{k-1},-\frac{k}{2}}\times \chi_{0,\frac{k}{2}}\lra D_k\lra 0,
\end{equation}
where $V_k$ consisting of degree $<k$ polynomial functions when restricted to $\ov{N_2}$. Hence $V_k\simeq \Sym^{k-1} V_{\mathrm{std}}\cdot|\det|^{-\frac{k-1}{2}}$, where $V_{\mathrm{std}}$ refers to the standard representation of $\GL_2(\BR)$. Furthermore, $D_k$ is the unique relative discrete series of $\GL_2(\BR)$ with central character $\chi_{\epsilon_{k-1},0}$ and infinitesimal character $(-k,k)$. 

Under the local Langlands correspondence, the Weil group representation $\kappa_{k,s}$ will correspond to $D_{k,s}:=D_{k}\cdot |\det|^s$. We define the \textbf{real part} of $D_{k,s}$ as 
\[
\mathrm{Re}\, D_{k,s}:=\mathrm{Re}\,  s.
\]
Let $(\pi_i)_{1\leq i\leq r}$ be a set consisting of characters of $\GL_1$ and relative discrete series of $\GL_2$. The \textbf{standard module} is the parabolic induction $\pi_1\times \dots\times \pi_r$ such that 
\[
\mathrm{Re}(\pi_1)\geq \dots\geq \mathrm{Re}(\pi_r).
\]
The standard module has a unique irreducible quotient, which is called the \textbf{Langlands quotient}.
Let $\kappa=\oplus_{i=1}^r\kappa_i$ be an $n$-dimensional semi-simple $W_{\mathbf{k}}$-representation such that $\kappa_i$ is irreducible. Then the irreducible $\GL_n$ representation corresponding to $\kappa$ is the Langlands quotient of $\pi:=\pi_1\times\dots\times \pi_r$, where $\pi_i$ corresponds to $\kappa_i$ under some rearrangement of $(\kappa_i)$ making $\pi$ to be a standard module.

For applications to extension vanishing results for irreducible generic representations, it is desirable to describe the BZ-filtration of relative discrete series. By short exact sequence~\ref{DS equ}, we first describe the BZ-filtration of $I(k)$. It admits a level $\leq 1$ BZ-filtration:
\[
 I(k)|_{P_2}=\sigma_0\supset \sigma_1\supset\sigma_2\supset 0,
\]
such that $\sigma_0/\sigma_1=I(k)_c$ corresponds to the unique closed orbit of $P_2$ on $B_2\backslash \GL_2$, and $\sigma_1=I(k)_o$ corresponds to the unique open orbit. Moreover, each subquotient $\sigma_i/\sigma_{i+1}$ admits a decreasing filtration.
\begin{itemize}
    \item $\sigma_2$ is irreducible and isomorphic to the Gelfand-Graev representation $I(\BC)$. 
    \item $\sigma_1/\sigma_2$ admits an infinite decreasing filtration
    \[
    \sigma_1=\sigma_{1,0}\supset \sigma_{1,1}\supset \dots \supset \sigma_2
    \]
    such that $\sigma_{1,i}/\sigma_{1,i+1}\simeq E(|\det|^{\frac{k+1}{2}}\cdot (\det)^i)$ for $i\in\BZ_{\geq 0}$.
    \item $\sigma_0/\sigma_1$ admits an infinite decreasing filtration 
    \[
    \sigma_0=\sigma_{0,0}\supset \sigma_{0,1}\supset \dots \supset \sigma_1
    \]
    such that $\sigma_{0,i}/\sigma_{0,i+1}\simeq E(|\det|^{\frac{k-1}{2}}\cdot (\det)^{i-(k-1)})$ for $i\in\BZ_{\geq 0}$.
\end{itemize}
Consequently, the relative discrete series $D_k$ admits a level $\leq 1$ BZ-filtration:
\begin{equation}\label{DS_pos_fil}
     D_k|_{P_2}=\sigma_0'\supset \sigma_1\supset\sigma_2\supset 0,
\end{equation}
where $\sigma_1$ coincides with the subrepresentation occurring in $I(k)$.
In addition,  $\sigma'_0/\sigma_1$ admits an infinite decreasing filtration 
    \[
    \sigma_0'=\sigma'_{0,0}\supset \sigma'_{0,1}\supset \dots \supset \sigma_1
    \]
    such that $\sigma'_{0,i}/\sigma'_{0,i+1}\simeq E(|\det|^{\frac{k-1}{2}}\cdot (\det)^{i+1})$ for $i\in\BZ_{\geq 0}$.
Likewise, $D_k$ also admits a level $\leq 1$ opposite BZ-filtration:
\[
D_k|_{\ov{P_2}}= \ov{\sigma_0}\supset \ov{\sigma_1}\supset \ov{\sigma_2}\supset 0,
\]
where $\ov{\sigma_1}= I(k)_o$ corresponds to the unique open orbit of $\ov{P_2}$ on $B_2\backslash\GL_2$ and $\ov{\sigma_0}/\ov{\sigma_1}$ admits an infinite decreasing filtration
\[
\ov{\sigma_0}=\ov{\sigma_{0,0}}\supset \ov{\sigma_{0,1}}\supset \dots \supset \ov{\sigma_1}
\]
such that $\ov{\sigma_{0,i}}/\ov{\sigma_{0,i+1}} \simeq E(|\det|^{\frac{k-1}{2}}\cdot (\det)^{-i-k})$ for $i\in\BZ_{\geq 0}$. Note that here we directly use the Mackey theory of $\ov{P_2}$ on $B_2\backslash\GL_2$, see also Remark~\ref{order_rem}.

\subsection{Irreducibility of standard module}\label{irr_stand}

In this subsection, notation follows from section~\ref{notation_sec} and section~\ref{LLC-sec}. It is well-known that an irreducible representation of $\GL_n$ is generic if and only if its standard module is irreducible. In this subsection, we describe these irreducible standard modules, see \cite{S77} for details.
\begin{itemize}
    \item Let $\mathbf{k}=\BC$. For principal series of $\GL_n(\mathbb{C})$
    \[\pi=\prod_{i=1}^n\chi_{m_i,s_i},m_i\in \mathbb{Z}, s_i\in \mathbb{C},\]
    it is irreducible if and only if 
\[s_i-s_j\notin \frac{|m_i-m_j|}{2}+\mathbb{Z}_{>0}, \ \forall  i\neq j.\]
\item Let $\mathbf{k}=\BR$. For parabolic induction of $\GL_n(\mathbb{R})$
\[\pi=\prod_{i=1}^m \chi_{\epsilon_i,s_i}\times \prod_{j=1}^lD_{k_j,t_j}, n=m+2l,\]
where $\epsilon_i\in\{0,1\},s_i\in\BC$ and $D_{k_j,t_j}$ is the relative discrete series defined in section~\ref{LLC-sec}, it is irreducible if and only if 
\begin{enumerate}
    \item $s_i-s_{i'}\notin  |\epsilon_i-\epsilon_{i'}|-1+2\mathbb{Z}_{>0}$, \ $\forall i\neq i'$;
    \item  $|s_i-t_j|\notin \frac{k_j}{2}+\mathbb{Z}_{>0}$, \ $\forall 1\leq i\leq m,1\leq j\leq l$;
    \item  $t_j-t_{j'}\notin \frac{|k_j-k_{j'}|}{2}+\mathbb{Z}_{>0}$, $\forall j\neq j'$.
\end{enumerate}

\end{itemize}

We need the following lemma in the proof of Theorem~\ref{ext-van-thm}.
\begin{lemma}\label{irr-lem}
    Let $D_{k_1,t_1},D_{k_2,t_2}$ be two discrete series of $\GL_2(\BR)$, and $\chi_{\epsilon,s}$ be a character of $\GL_1(\BR)$.
    \begin{enumerate}
        \item Assume $\frac{k_1}{2}+t_1-s\in \BZ_{>0}$, then $D_{k_1,t_1}\times \chi_{\epsilon,s}$ is irreducible only if $t_1-\frac{k_1}{2}\leq s$.
        \item Assume $\frac{k_1}{2}+t_1-(t_2-\frac{k_2}{2})\in \BZ_{>0}$, then $D_{k_1,t_1}\times D_{k_2,t_2}$ is irreducible only if $t_1-\frac{k_1}{2}\leq t_2-\frac{k_2}{2}$ or $t_2+\frac{k_2}{2}\geq \frac{k_1}{2}+t_1$.
        \item Assume $t_1-\frac{k_1}{2}-(t_2+\frac{k_2}{2})\in \BZ_{>0}$, then $D_{k_1,t_1}\times D_{k_2,t_2}$ is reducible.
    \end{enumerate}
\end{lemma}
\begin{proof}
    \begin{enumerate}
        \item Let $\frac{k_1}{2}+t_1-s=p$ be a positive integer. Suppose $t_1-\frac{k_1}{2}>s$, then 
        \[t_1-s=p-\frac{k_1}{2}=\frac{k_1}{2}+(p-k_1),
        \]
        and $(p-k_1)>0$. Hence the result follows from irreducibility criterion.
        \item Let $\frac{k_1}{2}+t_1-(t_2-\frac{k_2}{2})=p$ be a positive integer. Suppose 
        \begin{equation}\label{ineq}
                    t_1-\frac{k_1}{2}> t_2-\frac{k_2}{2} \text{ and } t_2+\frac{k_2}{2}< \frac{k_1}{2}+t_1.
        \end{equation}
        Without loss of generality, we assume $k_1>k_2$. Then 
        \[
        t_1-t_2=p-\frac{k_1+k_2}{2}=\frac{k_1-k_2}{2}+(p-k_1) .
        \]
        The first inequality in~\eqref{ineq} shows that $t_1-t_2>\frac{k_1-k_2}{2}$, which implies $p-k_1$ is a positive integer. Thus the result follows from irreducibility criterion.
        \item The third statement follows directly from irreducibility criterion.
    \end{enumerate}
\end{proof}

\section{Bernstein-Zelevinsky filtration of unitary representations}
For general linear groups over a $p$-adic field, Bernstein proposed a unitarity criterion for irreducible representations, see \cite[section 7.3]{Ber84}. In this section, our main result is a similar necessary condition in the Archimedean case, see Theorem~\ref{uni-crit}. Our approach is based on the classification of unitary dual, which is different from the $p$-adic case since the theory of $\ell$-sheaves is not available. We first recall the classification of the unitary dual due to D. Vogan, see \cite{Vog86} as well.
\begin{itemize}
    \item Let $\mathbf{k}=\BC$. Every irreducible unitary representation of $\GL_n(\BC)$ is a product of following two kinds of representations:
    \begin{enumerate}
        \item unitary characters $\chi_{k,s}$, where $k\in\BZ,s\in\sqrt{-1}\BR$, and
        \item complementary series $\chi(|\det|^s\times |\det|^{-s})$, where $\chi$ is a unitary character and $0< s<1$.
    \end{enumerate}
    \item Let $\mathbf{k}=\BR$. Every irreducible unitary representation of $\GL_n(\BR)$ is a product of following four kinds of representations:
    \begin{enumerate}
        \item The Speh representations $\chi\cdot \delta(m)$ indexed by an unitary character $\chi$ and an integer $m$, which we will explain in more detail;
        \item The unitary characters $\chi_{k,s}$, where $k\in\{0,1\}$ and $s\in\sqrt{-1}\BR$;
        \item The Stein complementary series $\chi(|\det|^s\times |\det|^{-s})$, where $\chi$ is a unitary character and $0< s<\frac{1}{2}$;
        \item The Speh complementary series $\chi(\delta(m)|\det|^s\times\delta(m)|\det|^{-s}) $, where $\chi\delta(m)$ is a Speh representation and $0< s<\frac{1}{2}$.
    \end{enumerate}
\end{itemize}

To prove Theorem~\ref{uni-crit}, we realize an irreducible unitary representation as a product of characters and Speh representations. We first use Theorem~\ref{central_char_of_fil} to prove the case when the irreducible unitary representation is a product of characters. 

\begin{lemma}\label{uni-cen-char}
Let $\pi$ be an irreducible unitary $\GL_n$-representation of depth $d$, which is also a degenerate principal series. For every irreducible subquotient $I^{k-1}E(\tau)$ in the BZ-filtration of $\pi|_{M_n}$ satisfying $k \neq d$ (where $\tau$ denotes an irreducible representation of $\GL_{n-k}$), we have $\mathrm{Re}\,   \omega_{\tau} >0$.
\end{lemma}

\begin{proof}
Let $m_1$ be the number of unitary characters in $\pi$, and $m_2$ be the number of complementary series in $\pi$. By rearranging the characters, we write $\pi$ as $\prod_{i=1}^m \chi_{r_i,s_i}$ such that 
\begin{itemize}
\item $s_i\in \sqrt{-1}\mathbb{R}$ when $1\leq i\leq m_1$,
\item $r_{m_1+2j-1}=r_{m_1+2j}$ and $s_{m_1+2j-1}-t_j=s_{m_1+2j}+t_j\in \sqrt{-1}\mathbb{R}$ for some $0<t_j<\frac{1}{2}$, when $1\leq j\leq m_2$.
\end{itemize}
Following the notations of Theorem \ref{central_char_of_fil}, when $1\leq i\leq m_1$, 
\[
\mathrm{Re}\,   \omega_{\tau_i} =0 \text{ or } \mathrm{Re}\,  \omega_{\tau_i} \in n_i\cdot\frac{1}{2}+\BZ_{\geq 0}
\]
since $\chi_{r_i,s_i}$ is a unitary character. When $1\leq j\leq m_2$, there are three cases about $\tau_{m_1+2j-1}\times \tau_{m_1+2j}$. 
\begin{itemize}
\item Both $\tau_{m_1+2j-1}$ and $\tau_{m_1+2j}$ are in case (i) of Theorem~\ref{central_char_of_fil}. Then 
\[
\mathrm{Re}\,   \omega_{\tau_{m_1+2j-1}}+\mathrm{Re}\,   \omega_{\tau_{m_1+2j}}  \in \BZ_{\geq 0}.
\]
\item One of $\tau_{m_1+2j-1}$ and $\tau_{m_1+2j}$ is in case (i) of Theorem \ref{central_char_of_fil}. Then 
\[
\mathrm{Re}\,   \omega_{\tau_{m_1+2j-1}}+\mathrm{Re}\,   \omega_{\tau_{m_1+2j}} \geq  n_{m_1+2j}\cdot \frac{1}{2}-t_{j} >0
\]
since $t_j<
\frac{1}{2}$.
\item Both $\tau_{m_1+2j-1}$ and $\tau_{m_1+2j}$ are in case (ii) of Theorem \ref{central_char_of_fil}. Then
\[
\mathrm{Re}\,   \omega_{\tau_{m_1+2j-1}}+\mathrm{Re}\,   \omega_{\tau_{m_1+2j}}=0.
\]
\end{itemize}

\end{proof}

In the rest of this chapter, we will describe the Speh representations and their BZ-filtration, which are also the building blocks for representations in Arthur type. The Speh representation $\delta(m,n)$ of $\GL_{2n}(\BR)$ is the unique irreducible submodule of
\[
\chi_{\epsilon_{m-1},\frac{m}{2}}\times \chi_{0,-\frac{m}{2}},
\]
where $\chi_{\epsilon_{m-1},\frac{m}{2}}$ and $\chi_{0,-\frac{m}{2}}$ are characters of $\GL_n(\BR)$, see \cite{SaSt90}. When $n$ is clear from context or is not important, we will simply denote $\delta(m)$.  Observing the associated variety, \cite[section 4]{AGS15a} proves that
\begin{equation}\label{speh der}
      \delta(m,n)^-=\delta(m,n-1) 
\end{equation}
and for $ 0< s<\frac{1}{2}$,
\[
(\delta(m,n)|\det|^s\times\delta(m,n) |\det|^{-s})^-=\delta(m,n-1)|\det|^s\times\delta(m,n-1) |\det|^{-s}.
\]
Actually, by Theorem~\ref{high-thm}, the second point follows from the first point. In order to better investigate the positivity in the BZ-filtration of Speh representations, we prefer another inductive realization. Like $p$-adic case, the Speh representation $\delta(m,n)$ is the unique irreducible submodule of 
\[
 D_m|\det|^{\frac{n-1}{2}}\times D_m|\det|^{\frac{n-3}{2}}\times \dots \times D_m|\det|^{\frac{1-n}{2}},
\]
hence is the unique irreducible submodule of
\begin{equation}\label{speh}
   \Pi:=\delta(m,n-1)|\det|^{\frac{1}{2}}\times D_m|\det|^{\frac{1-n}{2}}. 
\end{equation}
We realize $\Pi$ as a tempered bundle on $P_{2n-2,2}\backslash \GL_{2n}$ and describe the irreducible subquotient in its BZ-filtration that lies in $\delta(m,n)$ as well. Since $M_{2n}$-action has a unique open orbit and a unique closed orbit on $P_{2,2n-2}\backslash \GL_{2n}$, 
\[
 0\lra \Pi_o\lra \Pi|_{M_{2n}}\lra \Pi_c\lra 0,
\]
where 
\[
\Pi_o= D_m|\det|^{\frac{1-n}{2}} \mind |\det|^{\frac{1}{2}}\delta(m,n-1)|_{M_{2n-2}}.
\]
By~\eqref{speh der}, $\delta(m,n-1)|_{M_{2n-2}}$ admits a BZ-filtration with bottom layer $IE(\delta(m,n-2))$. the Bernstein–Zelevinsky filtration of $\Pi_0$ contains a subquotient isomorphic to
\[
 I(D_m|\det|^{\frac{1-n}{2}} \mind E(|\det|^{\frac{1}{2}}\delta(m,n-2)).
\]
By Lemma~\ref{open_orbit}, it has a BZ-filtration with a successive quotient isomorphic to 
\[
 IE(\delta(m,n-1)\hookrightarrow IE(|\det|^{\frac{1}{2}}\delta(m,n-2)\times D_m|\det|^{\frac{2-n}{2}}).
\]
This is the bottom layer in the BZ-filtration of $\delta(m,n)$, and other terms in the BZ-filtration of $\delta(m,n)$ have depth one. 

Now we use an inductive argument to show that Theorem~\ref{uni-crit} holds for Speh representations $\delta(m,n)$. When $n=1$, the Speh representations are discrete series; hence, the result follows from the discussion in section~\ref{LLC-sec}. We assume Theorem~\ref{uni-crit} holds for $\delta(m,n-1)$, and proceed to prove the statement for $\delta(m,n)$. Each depth one term in $\Pi_c$ has the form
\[
E\left((\delta(m,n-1)|\det|\otimes_{\BR} \Sym^j(\BR^{2n-2}))\times \chi|\det|^{\frac{1-n}{2}}\right)
\]
for some positive integer $j$ such that $\chi$ is a character of $\GL_1$ and $E(\chi)$ is a successive quotient in the BZ-filtration of $D_m$. Here, $\BR^{2n-2}$ is the standard representation of $\GL_{2n-2}(\BR)$. Consequently,
\[
\mathrm{Re}\,  \omega_{\widetilde{\tau}}\geq 2n-2+\frac{1-n}{2}+\mathrm{Re}\,  \omega_{\tau} >0
\]
since $\mathrm{Re}\,  \omega_{\tau}\geq 1$ by argument in section~\ref{LLC-sec}, where 
$$\widetilde{\tau}=\left(\delta(m,n-1)|\det|\otimes_{\BR} \Sym^j(\BR^{2n-2})\right)\times \chi|\det|^{\frac{1-n}{2}}.$$
On the other hand, the depth one term in $\Pi_o$ has the form
\[
E\left(\tau|\det|^{\frac{1}{2}}\times( D_m|\det|^{\frac{2-n}{2}}\otimes_{\BR}\Sym^j(\BR^{2} ))\right),
\]
for some positive integer $j$ such that $E(\tau)$ is a depth one term in the BZ-filtration of $\delta(m,n-1)$. Here, $\BR^{2}$ is the standard representation of $\GL_{2}(\BR)$. Consequently,
\[
\mathrm{Re}\,  \omega_{\widetilde{\tau}}= \mathrm{Re}\,  \omega_{\tau} +\frac{2n-3}{2}+ \frac{2-n}{2}\times 2 >0
\]
by the inductive assumption, where $\widetilde{\tau}=\tau|\det|^{\frac{1}{2}}\times\left( D_m|\det|^{\frac{2-n}{2}}\otimes_{\BR}\Sym^j(\BR^{2} )\right)$. Therefore, by a similar argument to Lemma~\ref{uni-cen-char}, we get the following Lemma, which completes the proof of Theorem~\ref{uni-crit} together with Lemma~\ref{uni-cen-char}.
\begin{lemma}
    Let $\pi$ be an irreducible unitary $\GL_n$-representation of depth $d$, which is also a product of Speh representations and Speh complementary series. For every irreducible subquotient $I^{k-1}E(\tau)$ in the BZ-filtration of $\pi|_{M_n}$ satisfying $k \neq d$ (where $\tau$ denotes an irreducible representation of $\GL_{n-k}$), we have $\mathrm{Re}\,   \omega_{\tau} >0$.
\end{lemma}

\section{The restriction to maximal parabolic subgroup}\label{res-max-sec}

Given a Casselman-Wallach representation $\pi$ of $\GL_n$, by BZ-filtration, the restriction of $\pi$ to the mirabolic subgroup (or the parabolic subgroup $P_{n-1,1}$) decomposes discretely. 
In general,  it will be shown that the restriction of $\pi$ to every maximal parabolic subgroup admits a filtration with successive quotients being the Mackey inductions. 

\subsection{Coarse spectral filtration}
Let $G$ be a real reductive group, and $P=LU$ be a parabolic subgroup with Levi decomposition, such that $U$ is abelian. We first define a family of representations of $P$ that generalizes the trivial extension and Mackey induction of $M_n$. For every $\phi\in\widehat{U}$, let $S_{\phi}$
 be the stabilizer subgroup of $P$-action on $\phi$. Then we have decomposition
 \[
 S_{\phi}=(S_{\phi}\cap L)\ltimes U.
 \]
 For a representation $\sigma$ of $S_{\phi}\cap L$, define the induction
 \[
 I_{\phi}(\sigma):=\SInd_{S_{\phi}}^{P}(\sigma\boxtimes \phi).
 \]
 We call such representations \textbf{geometrical Mackey inductions}.
 
 In particular, when $P=P_{n-k,k}$, we have $U_{n-k,k}\simeq\Hom_\mathbf{k}( \mathbf{k}^{k},\mathbf{k}^{n-k})$ and
  \[\Hom_\mathbf{k}( \mathbf{k}^{n-k},\mathbf{k}^{k})\stackrel{\simeq}{\lra}\widehat{ U_{n-k,k}} \quad  x\mapsto \left(u\mapsto \psi(\tr (x\circ u)\right).
 \]
 Hence, the $L$-orbit $\widehat{ U_{n-k,k}}$ on is determined by rank. Specifically, for
\[ x=\begin{pmatrix}\ \begin{matrix} I_{n-k}\ \Big| & \begin{matrix} A & C\\ B & D\end{matrix} \\ \hline 0 & I_k\end{matrix}\ \end{pmatrix}, \begin{array}{ll} A\in \mathbf{k}^{(n-k-l)\times k}, & C\in \mathbf{k}^{(n-k-l)\times (k-l)},\\ B\in \mathbf{k}^{l\times l}, & D\in \mathbf{k}^{l\times (k-l)},\end{array}\]   we choose the standard $\psi^{n,k}_l(x):=\psi(\mathrm{trace}(B))$. Let $S^{n,k}_l$ be the stabilizer of $\psi^{n,k}_l$. When $k$ is clear from the context, we will omit $k$ in the above notations for simplicity. Moreover, for a representation $\sigma$ of $L\cap S_l^n$, the geometric Mackey induction $I_{\psi_l^n}(\sigma)$ will simply be denoted as $I_l(\sigma)$. For the trivial extension $I_0(\sigma)$, if $\begin{pmatrix}
aI_{n-k} & \\ & a^{-1}I_k
\end{pmatrix}$ acts on $\sigma$ by scalar $a^c$ for every $a\in \mathbb{R}_{>0}$, then let $\omega_{\sigma}$ denote this exponent $c$.

\begin{proposition}\label{rest_max}
Given a principal series $\pi$ of $\GL_n$, for every maximal parabolic subgroup $P_{n-k,k}$, the restriction $\pi|_{P_{n-k,k}}$ admits a filtration as in definition~\ref{def_fil}, where each successive quotient is a geometrical Mackey induction satisfying the following:
\begin{itemize}
\item[(i)] When a successive quotient of the filtration is of the form $I_0(\sigma)$, then $\sigma$ is an irreducible Casselman-Wallach representation.
\item[(ii)] The set $\{\mathrm{Re}(\omega_{\sigma})\ |\ I_0(\sigma)\ \text{is a successive quotient of the filtration}\}$ has a finite minimal value.
\end{itemize}
\end{proposition}

The filtration described in the above proposition will be called the \textbf{coarse spectral filtration}. The term ``coarse" indicates that the successive quotients in the filtration are not necessarily irreducible. 

For simplifying the notation, we introduce the following inductions which will be freely used in the following two sections. Let $\chi$ be a character of $\GL_1$, $\sigma$ be a representation of $P_{n-m-1,m}$ and $\tau$ be a representation of $P_{n-m,m-1}$, where $m$ is a positive integer such that $m<n-1$. Let $\beta$ be a representation of $S_l^{n-1}$.
\begin{enumerate}
    \item Induction $\chi\mind\sigma $ is defined as
    \[
    \SInd_{\ov{P_{1,n-1}}\cap P_{n-m,m}}^{P_{n-m,m}}(\chi\boxtimes\sigma)
    \]
    where $\chi\boxtimes\sigma$ is a representation of $\GL_1\times P_{n-m-1,m}$ and is viewed as a representation of $\ov{P_{1,n-1}}\cap P_{n-m,m}$ by trivial extension.
    \item Induction $\tau\times \chi$ is defined as 
    \[
    \SInd_{P_{n-1,1}\cap P_{n-m,m}}^{P_{n-m,m}}(\tau\boxtimes\chi)
    \]
     where $\tau\boxtimes\chi$ is a representation of $P_{n-m,m-1}\times \GL_1$ and is viewed as a representation of $P_{n-1,1}\cap P_{n-m,m}$ by trivial extension.
     \item We view $\beta\boxtimes \chi$ as a representation of $S_l^n\cap P_{n-1,1}$ by trivial extension. Then we define the induction $\beta\times 
     \chi$ as
     \[
     \SInd^{S_l^n}_{S_l^n\cap P_{n-1,1}}( \beta\boxtimes \chi).
     \]
\end{enumerate}

\begin{proof}[Proof of Proposition~\ref{rest_max}]
Let us prove Proposition \ref{rest_max} by induction on $k$ and $n$. For $k=1$, it is by BZ-filtration.

Let $\pi=\tau\times\chi$, where $\tau$ is a principal series of $\GL_{n-1}$. Considering the  $P_{n-k,k}$-orbit on $P_{n-1,1}\backslash \GL_n$, as before, one has the exact sequence 
\[0\to \pi_o\to \pi|_{P_{n-k,k}} \to \pi_c\to 0.\] 

By Borel filtration, one obtains a filtration of $\pi_c$ such that each successive quotient is of the form 
$$\SInd_{P_{n-1,1}\cap P_{n-k,k}}^{P_{n-k,k}}\left((\tau|_{P_{n-k,k-1}} \boxtimes \chi)\otimes \mathrm{Sym}^i\big(\mathfrak{gl}_n/(\mathfrak{p}_{n-1,1}+\mathfrak{p}_{n-k,k})\big)^{\vee}\right),i\in \mathbb{Z}_{\geq 0},$$
 where $\mathfrak{gl}_n/(\mathfrak{p}_{n-1,1}+\mathfrak{p}_{n-k,k})$ admits the adjoint action of $P_{n-1,1}\cap P_{n-k,k}$.

By induction on $k$, $\tau|_{P_{n-k,k-1}}$ admits a filtration with successive quotients being geometrical Mackey inductions, $I_l(\sigma)$. Then the result follows from the following isomorphism
\[
I_l(\sigma)\times \chi\simeq  I_l(\sigma\times \chi).
\]
The properties (i) and (ii) follow from the inductive assumption and the property of the Borel's filtration. 

On the other hand, one has $\pi_o\simeq \chi\overline{\times}(\tau|_{P_{n-k-1,k}})$. By induction on $n$, $\tau|_{P_{n-k-1,k}}$ admits a filtration with successive quotients being geometrical Mackey inductions and satisfying (i) and (ii). Let us first show that the representation $\chi\overline{\times}I_l(\beta)$ admits a filtration with successive quotients being geometrical Mackey inductions. Here, $\beta$ is a representation of $S_{n-1}^l\cap (\GL_{n-k-1}\times\GL_k)$.
Write $P_{n-k,k}$ as
\[\left\{\begin{pmatrix} a & b & c \\ d & e & f \\ 0 & 0 & g 
\end{pmatrix}\in P_{n-k,k}\ \middle| \ a\in \mathbf{k}, e\in \mathbf{k}^{(n-k-1)\times (n-k-1)}, g\in \mathbf{k}^{k\times k}\right\} \]

Consider the subgroup
\[Y:=\left\{\begin{pmatrix} a & 0 & c \\ d & e & f \\ 0 & 0 & g 
\end{pmatrix}\ \middle|\ e=\begin{pmatrix} * & *\\ 0_{l\times (n-k-1-l)} & t\end{pmatrix}, g=\begin{pmatrix} t & *\\ 0_{(k-l)\times l} & *\end{pmatrix}\right\}, \]
Then $Y\cap (1\times P_{n-k-1,k})=S_{n-1}^l$.
Consider the representation of $Y$ induced from the subgroup $Y_1:=\left\{\begin{pmatrix} a & 0 & 0 \\ d & e & f \\ 0 & 0 & g 
\end{pmatrix}\in Y\right\}$, $\gamma:=\SInd_{Y_1}^Y(\chi\otimes (\beta\otimes \psi_l^{n-1}))$. The representation can be realized as the space of Schwartz functions from 
\[Y_1\backslash Y\simeq \left\{\begin{pmatrix} 1 & 0 & x \\ 0 & I_{n-k-1} & 0 \\ 0 & 0 & I_k
\end{pmatrix}\ \middle|\ x\in \mathbf{k}^{k}\right\}\] 
to the underlying space of  $\beta$.
By
\[\begin{pmatrix} 1 & 0 & x \\ 0 & I_{n-k-1} & 0 \\ 0 & 0 & I_{k}
\end{pmatrix}\begin{pmatrix} a & 0 & c \\ d & e & f \\ 0 & 0 & g 
\end{pmatrix}=\begin{pmatrix} a & 0 & 0 \\ d & e & f-da^{-1}(c+xg) \\ 0 & 0 & g 
\end{pmatrix}\begin{pmatrix} 1 & 0 & a^{-1}(c+xg) \\ 0 & I_{n-k-1} & 0 \\ 0 & 0 & I_{k}
\end{pmatrix},\]
the action of $p=\begin{pmatrix} a & 0 & c \\ d & e & f \\ 0 & 0 & g 
\end{pmatrix}\in Y$ is given by
\[(\gamma(p)h)(x)=|a|_{\bf{k}}^{-\frac{k}{2}}\chi(a)\cdot \beta(\begin{pmatrix}e & 0 \\  0 & g 
\end{pmatrix})\cdot  \psi_l^{n-1}(\begin{pmatrix}I_{n-k-1} & e^{-1}(f-da^{-1}(c+xg)) \\  0 & I_k 
\end{pmatrix}) h(a^{-1}(c+xg)).\] 
Denote the action of $p$ after applying the Fourier transform (using $\psi^{-1}$) to the variable $x$ by
\[
\widehat{\gamma}(p)(\widehat{h}):=\mathcal{F}_x\circ\gamma(p)\circ\mathcal{F}_x^{-1}(\widehat{h}),
\]
where $\widehat{h}$ is a Schwartz function on the Fourier domain. Namely, we have
\[(\widehat{\gamma}(p)\widehat{h})(y)=\chi(a)\cdot \beta(\begin{pmatrix}e & 0 \\  0 & g 
\end{pmatrix})\psi_{l}^{n-1}(fg^{-1})\psi(cg^{-1}y)\cdot |a|_{\mathbf{k}}^{\frac{k}{2}}|g|_{\mathbf{k}}^{-\frac{1}{2}}\widehat{h}(g^{-1}ya+d''),\] 
where $d''\in \mathbf{k}^{k}$ with $d''=g^{-1}\begin{pmatrix}
d' \\  0_{(k-l)\times 1} 
\end{pmatrix}$ and $d'\in \mathbf{k}^{l}$ such that $d=\begin{pmatrix}
* \\ d'
\end{pmatrix}$. This action keeps the closed subspace $\mathbf{k}^l\times0^{(k-l)}$ of $\mathbf{k}^k$. Thus, there exists a short exact sequence
\[
0\lra \widehat{\gamma}|_{\mathbf{k}^k\setminus (\mathbf{k}^l\times 0^{k-l})} \lra \widehat{\gamma}\lra \widehat{\gamma}^{\sharp}\lra 0,
\]
where $ \widehat{\gamma}|_{\mathbf{k}^k\setminus (\mathbf{k}^l\times 0^{k-l})}$ consists of Schwartz sections supported on $\mathbf{k}^k\setminus (\mathbf{k}^l\times 0^{k-l})$.
\medskip

\textbf{(1). Filtration of $\SInd_{Y}^{P_{n-k,k}}( \widehat{\gamma}|_{\mathbf{k}^k\setminus (\mathbf{k}^l\times 0^{k-l})})$.} When $0\leq l\leq k-1$,
consider another representation $\eta$ of $Y$ which is induced from $\chi\boxtimes \beta\cdot |a|_{\mathbf{k}}^{\frac{k}{2}}\otimes(\psi_l^n\cdot \widetilde{\psi})$ of 
\[Y_2:=\left\{\begin{pmatrix} a & 0 & c \\ d & e & f \\ 0 & 0 & g 
\end{pmatrix}\ \middle|\ e=\begin{pmatrix} * & *\\ 0_{l\times (n-k-1-l)} & t\end{pmatrix}, g=\begin{pmatrix} t & d' & *\\ 0_{1\times l} &  a &  *\\ 0_{(k-l-1)\times l} &  0_{(k-l-1)\times 1} &  *\end{pmatrix}\right\} \]
where $ \widetilde{\psi}$ takes the value $\psi(c_{l+1})$ if $cg^{-1}=(c_1,\dots,c_k)$.
Then the induced representation $\eta$ can be realized as the sum of  the Schwartz functions from the affine spaces $A_r$ to  the underlying space of $\beta$, where 
\[A_r=\left\{\begin{pmatrix}
I_{n-k} & 0\\ 0 & a_r
\end{pmatrix}\ \middle|\ a_r= w_l^{-1}\cdot \begin{pmatrix} z \Bigg\vert \begin{matrix}  I_r & 0_{r\times (k-1-r)} \\ 0_{1\times r} &  0_{1\times (k-1-r)} \\  0_{(k-1-r)\times r} & I_{k-1-r}\end{matrix}\end{pmatrix}^{-1}, z\in \mathbf{k}^{r}\times \mathbf{k}^{\times}\times \mathbf{k}^{k-1-r}\right\},\] for $l\leq r\leq k-1$, where
 $w_l=\begin{pmatrix}
& 1  & \\  I_{l-1} & & \\ && I_{k-1-l}
\end{pmatrix}$. Over $A_r$, the action of $p\in Y$ is given by
\[(\eta(p)\widetilde{h})(z)=\chi(a)\cdot \beta(\begin{pmatrix}e & 0 \\  0 & \widetilde{g} 
\end{pmatrix})\psi_l^{n-1}(\widetilde{f}\widetilde{g}^{-1}) \widetilde{\psi}(\widetilde{c})\widetilde{h}(\widetilde{z}),\]
where $\widetilde{\cdot}$ are determined by
$\begin{pmatrix}
I_{n-k} & 0\\ 0 & a_r
\end{pmatrix}p=p_2\begin{pmatrix}
I_{n-k} & 0\\ 0 & \widetilde{a}_r
\end{pmatrix}$ with $p_2=\begin{pmatrix} a & 0 & \widetilde{c} \\ d & e & \widetilde{f} \\ 0 & 0 & \widetilde{g} 
\end{pmatrix}\in Y_2$ and $\widetilde{a}_r$ corresponds to $ \widetilde{z}$.

Define the intertwining operator $\mathcal{T}_r$ on $\CS(A_r,\beta)$ by
\[\CT_r(\widetilde{h})(z):= \beta(a_r)^{-1}\widetilde{h}(z).\]
One can verify that $\mathcal{T}_r$ and $\CT_{s}$ coincide on $\CS(A_r\cap A_s,\beta)$ for $r\neq s$. Moreover, $\CT:=\bigcup_r\CT_r$ intertwines $\eta$ and $ \widehat{\gamma}|_{\mathbf{k}^k\setminus (\mathbf{k}^l\times 0^{k-l})}$, that is, $\mathcal{T}\circ \widehat{\gamma}|_{\mathbf{k}^k\setminus (\mathbf{k}^l\times 0^{k-l})}= \eta\circ \mathcal{T}$.

Let  $w=\begin{pmatrix}
0& I_{n-k-1}& 0 \\ 1 & 0 & 0\\ 0 & 0 & I_k 
\end{pmatrix}$, then $\SInd_{{}^wY}^{P_{n-k,k}}({}^w\eta)$ is of the form $I_{l+1}(\cdot )$, so is  $\SInd_{Y}^{P_{n-k,k}}( \widehat{\gamma}|_{\mathbf{k}^k\setminus (\mathbf{k}^l\times 0^{k-l})})$.

When $l=k$, 
consider the subgroup $Y_3$ of $Y$
\[\left\{\begin{pmatrix} a & 0 & c \\ d & e & f \\ 0 & 0 & g 
\end{pmatrix}\ \middle|\ d=\begin{pmatrix} * \\ 0_{k\times 1}\end{pmatrix}\right\}. \]
The representation $\eta':=\SInd_{Y_3}^Y(|g|_{\mathbf{k}}^{-1}\cdot \chi\boxtimes \beta\cdot \psi_l^n)$ can be realized as the space of Schwartz functions from  
\[\left\{\begin{pmatrix} 1 & 0 & 0 \\ * & I_{n-k-1} & 0 \\ 0 & 0 & I_k 
\end{pmatrix}\ \middle|\ *=\begin{pmatrix} 0_{n-2k-1} \\ z\end{pmatrix}, z\in \mathbf{k}^{k}\right\} \]
to the underlying space of $\beta$. 
One can check directly that $\widehat{\gamma}$ is isomorphic to $\eta'$. Hence, the $\SInd_{Y}^{P_{n-k,k}}( \widehat{\gamma})$ also is of the form $I_k(\cdot)$ since $\SInd_{Y}^{P_{n-k,k}}(\eta')$ is.  

\textbf{(2). Filtration of $\SInd_{Y}^{P_{n-k,k}}\widehat{\gamma}^{\flat}$.} Over $\mathbf{k}^l\times 0^{k-l}$, by Borel's Lemma, one can get a filtration of $\widehat{\gamma}^{\sharp}$  with successive quotients of the form 
\[I_l(\SInd_{S_l^n\cap Y_2}^{S_l^n}(|a|_{\mathbf{k}}^{-\frac{l}{2}}\chi\otimes \beta\otimes_{\mathbb{R}} (\mathrm{Sym}^i(\mathbf{k}^l)^{\vee})),\]
where $\mathbf{k}^l$ is equipped with the adjoint representation of $S_l^n\cap Y_2$ on the Lie subalgebra 
\[\begin{pmatrix}
0_{1\times (n-l)} & \mathbf{k}^l\\ 0_{(n-1)\times (n-l)} & 0_{(n-1)\times l}
\end{pmatrix}.\]

For the properties (i) and (ii) of $\pi_o$, note that the term $I_0(\cdot)$ only shows up in the Borel's filtration of $ \widehat{\gamma}$ at $y=0$ in the case of $l=0$. Now (i) and (ii) follow from the induction and the property of  Borel's filtration.
\end{proof}

\subsection{Nilpotent invariants and comparison to $L^2$-theory}
In the next section, we will show that Proposition~\ref{rest_max} is a fundamental step in proving the Casselman-Wallach property of homology of the Jacquet functor. On the other hand, in this subsection, we will give some evidence and propose some conjectures stating that the coarse spectral filtration is related to some nilpotent invariants of the representation. Hence, it is desirable to prove its existence in a general setting. 

From now on, in this subsection, $P=LU$ is a parabolic subgroup of a real reductive group $G$ such that its unipotent radical $U$ is abelian and $\widehat{U}$ has finitely many $L$-orbits.
\begin{conjecture}\label{coa-conj}
Let $\pi$ be a Casselman-Wallach representation of $G$,  the restriction of $\pi$ to $P$ admits a filtration with each successive quotient being geometrical Mackey induction. 
\end{conjecture}
For example, when $G=G_n$, the Siegel subgroup $P=Q_n$ satisfies the condition of the above conjecture. We hope that the conjecture will indicate some aspects of the branching law of the symmetric pair $(G_n,\GL_n)$.

We introduce various nilpotent invariants attached to a representation $\pi$ of $P$.
\begin{itemize}
    \item \textbf{Spectral orbits.}
\end{itemize}
The set of spectral orbits consists of the $L$-orbits on $\widehat{U}$ for which there exists (and thus for any) an element $\phi$ in the orbit satisfying
\[
\pi / \langle u \cdot v - \phi(u)v \mid u \in U, v \in \pi \rangle \neq 0.
\]
When $\pi$ is a Casselman-Wallach representation of $G$ and admits a coarse spectral filtration, \textit{i.e.} Conjecture~\ref{coa-conj} holds, then the spectral orbits coincide with the orbits appearing in the successive quotients of the filtration (see the proof of Corollary~\ref{eq-spec-sup}). Moreover, when $\pi$ is a Casselman-Wallach representation of $G$, the zero orbit is always a spectral orbit since the surjective map
\[
 \pi/\fku\pi \lra \pi/\fku^0\pi, \ \text{where}\ \pi/\fku^0\pi\neq 0.
\]
We use $\mathrm{SO}_U(\pi)$ to denote the union of spectral orbits of $\pi$. We have the following basic conjecture.
\begin{conjecture}
    Let $\pi$ be a Casselman-Wallach representation of $G$. Then $\SO_U(\pi)$ is a closed subset of $\widehat{U}$.
\end{conjecture}
For instance, when $G=\GL_n$ and $\pi$ is a degenerate principal series, the conjecture follows from a calculation similar to that in Proposition~\ref{rest_max}.

\begin{itemize}
    \item \textbf{Smooth support $\supp_{U}(\pi)$.}
\end{itemize}
The Fr\'echet space $\CS(U)$ has two natural Fr\'echet algebra structures: one is the convolution $(\CS(U),*)$, the other one is the pointwise multiplication $(\CS(U),\bullet)$. Moreover, under the Fourier transform, we have a natural isomorphism of Fr\'echet algebras
\[
(\CS(U),*)\simeq (\CS(\widehat{U}),\bullet).
\]
Fix a Haar measure $du$ on $U$. Then the smooth moderate growth representation $\pi$ is a differentiable $(\CS(U),*)$-module by
\[
 f\cdot v:=\int_U f(u)u\cdot vdu \text{ for } f\in \CS(U) \text{ and } v\in\pi.
\]
Therefore, it is a module of $(\CS(\widehat{U}),\bullet)$ as well. Let $\CI_{\pi}\subset \CS(\widehat{U})$ be the closed annihilated ideal of $\pi$. Then we define the smooth support of $\pi$ as the complementary subset of the maximal open subset $\Omega\subset \widehat{U}$ such that $\CS(\Omega)\subset \CI_{\pi}$. Since $\pi$ is a $P$-representation, $\supp_{U}(\pi)$ is $L$-invariant. The following lemma can be proven by directly verifying the definition, and we leave the details to the reader.
\begin{lemma}\label{supp_Mac}
    Let $\sigma$ be a representation of $S_{\phi}\cap L$, where $\phi\in\widehat{U}$. Then 
    \[
    \supp_U(I_{\phi}(\sigma))=\ov{\CO_{\phi}},
    \]
    where $\CO_{\phi}$ is the $L$-orbit of $\phi$.
\end{lemma}

This lemma has a direct corollary.
\begin{corollary}\label{eq-spec-sup}
    Let $\pi$ be a Casselman-Wallach representation of $G$. Assume that Conjecture~\ref{coa-conj} holds, then 
    \[
    \ov{\SO_U(\pi)}=\supp_U(\pi).
    \]
\end{corollary}
\begin{proof}
By Lemma~\ref{supp_Mac}, it suffices to prove that 
   \begin{equation}\label{geo_mac_homo}
        \rmh_0(\fku,  I_{\phi}(\sigma)\otimes (-\phi))\simeq \sigma \text{ and } \rmh_i(\fku,I_{\phi}(\sigma)\otimes (-\phi'))=0
   \end{equation}
   for every integer $i$ and character $\phi'\notin \CO_{\phi}$, where $\phi,\phi'\in \widehat{U}$ and $\sigma$ is a representation of $S_{\phi}\cap L$. Consider the embedding of Nash manifolds:
   \[
   \varphi: S_{\phi}\backslash P\lra \fku^* \quad x\mapsto {}^x\phi.
   \]
   Note that $I_\phi(\sigma)$ can be realized as Schwartz sections of a tempered bundle $\CE$ over $S_{\phi}\backslash P$ such that the $\fku$-action is given by
   \[
  ( \xi\cdot f)(x):=\varphi(x)(\xi)\cdot f(x), \quad f\in\CS(S_{\phi}\backslash P,\CE)  \text{ and } \xi\in\fku.
   \]
   Therefore, the second assertion in~\eqref{geo_mac_homo} follows from Corollary~\ref{bun ver}. 

   On the other hand, by the covering technique, it suffices to prove for an open neighborhood $U$ of $\ov{e}\in S_{\phi}\backslash P$ that
   \[
   \rmh_0(\fku,\CS(U,\CE)\otimes (-\phi))\simeq\sigma.
   \]
   Here, $\ov{e}$ is the image of the identity element. Let $\fkw$ be the subalgebra of $\fku$ such that $\fkw^*$ is the image of $d\varphi_{\ov{e}}$. Consequently, there exists an open neighborhood $U$ of $\ov{e}$ such that $0\in \fkw^*$ is a regular value of
   \[
   \varphi_{\fkw}:U \xrightarrow{\varphi-\phi} \fku^* \lra \fkw^*.
   \]
   By Corollary~\ref{bun ver}, this implies
   \[
   \rmh_i(\fkw,\CS(U,\CE)\otimes (-\phi))\simeq \begin{cases}
      \CE_0= \sigma &\text{ for } i=0,\\
      0 &\text{ for }i\neq 0.
   \end{cases}
   \]
   Thus, we get 
   \[\rmh_0(\fku,\CS(U,\CE)\otimes (-\phi))=\rmh_0\left(\fku/\fkw,\rmh_0(\fkw,\CS(U,\CE)\otimes (-\phi))\right)\simeq \sigma.
   \]
\end{proof}

\begin{itemize}
    \item \textbf{Wavefront set}.
\end{itemize}

Let $\pi$ be a Casselman-Wallach representation of $G$. By the Casselman's embedding theorem, $\pi$ can be continuously embedded into a Hilbert generalized principal series. Take the closure $\ov{\pi}$ of $\pi$ in this Hilbert space. It is a Hilbert globalization of $\pi$. In the sense of \cite{How81}, we view $\ov{\pi}$ as a $U$-representation and define the wavefront set of $\pi$ as $\WF_U(\ov{\pi})$. It is a $L$-invariant closed subset of $\widehat{U}$, and is independent of the choice of Hilbert globalization. We have the following comparison between the wavefront set and the smooth support.
\begin{lemma}
    Let $\ov{\sigma}$ be a separable Hilbert globalization of $\sigma\in \Smod_P$. Then 
    \[
    \Supp_U(\sigma)= \WF_U(\ov{\sigma}).
    \]
\end{lemma}
\begin{proof}
    Let $\CL_1(\ov{\sigma})$ be the Banach ideal of bounded operators consisting of trace class operators. Define the continuous bounded function $\tr(T)$ on $U$ by
    \[
    \tr(T)(u):=\tr(T\circ \ov{\sigma}(u)) \text{ for } T\in \CL_1(\ov{\sigma}) .   \]
    We regard it as a distribution as well. Define a closed subset of $\widehat{U}$ as follows:
    \[
    \BS_{\ov{\sigma}}:=\ov{\bigcup_{T\in \CL_1(\ov{\sigma})} \supp \widehat{\tr(T)}},
    \]
    where $\widehat{\tr(T)}$ refers to the Fourier transform of $\tr(T)$. Note that $\BS_{\ov{\sigma}}$ is $L$-invariant. In particular, it is conic. Hence, by the argument in \cite[Proposition 2.1]{How81}, we have
    \[    \BS_{\ov{\sigma}} =\WF_U(\ov{\sigma}).
    \]
    Therefore, it suffices to prove 
    \[
    \BS_{\ov{\sigma}}=\supp_U(\sigma).
    \]
    On the one hand, since $\sigma$ is dense in $\ov{\sigma}$, $\BS_{\ov{\sigma}}\subset \supp_U(\sigma)$ by the definition. On the other hand, for every nonzero bounded operator $\varphi$, by taking a specific orthonormal basis such that $\langle \varphi(v),v\rangle\neq 0$ for some $v$ in the chosen basis , we can find a positive trace class operator $T$ such that $\tr(T\circ \varphi)\neq 0$. Consequently, we have the inverse containment $\supp_U(\sigma)\subset \BS_{\ov{\sigma}}$.
    \end{proof}
    
We would like to mention that these invariants have a close relation to the nilpotent invariants of $G$. On the one hand, a spectral orbit corresponds to a degenerate Whittaker model in the sense of \cite{GGS17}. For $\phi \in \widehat{U}$, we can find a semisimple element $h \in \Lie(G)$ such that $L=Z_G(h)$, the eigenvalues of $\ad(h)$ lie in $\mathbb{Q}$ and $\ad^*(h)(\phi) = -2\phi$. Such an element is unique modulo $\Lie(Z_G)$. We regard $\phi$ as an element of $\Lie(G)^*$ that is trivial on $\Lie(\ov{P})$; hence, $(h, \phi)$ is a Whittaker pair. The celebrated result \cite[Theorem A]{GGS17} establishes the connection between degenerate Whittaker models and generalized Whittaker models. 

On the other hand, in general, the wavefront set of $U$ and the wavefront set of $G$ have a partial relation, which leads to the following corollary; see \cite[Proposition 1.5]{How81} for details.

\begin{corollary}
    Let $\pi$ be a Casselman-Wallach representation of $G$. Let $\mathrm{pr}$ be the natural projection map from the linear dual of Lie algebra $\Lie(G)^*$ to $\widehat{U}$. Then
    \[
    \mathrm{pr}(\WF_G(\ov{\pi})) \subset \supp(\sigma).
    \]
\end{corollary}

\begin{example}
For an irreducible representation $\pi$ of $G$ with a specific wavefront set, the above relation together with Corollary~\ref{eq-spec-sup} can be employed to show that the wavefront set is contained within the generalized Whittaker model. For example, let $G = \GL_{2n}$ and let $\pi$ be an irreducible representation such that the maximal orbit in its wavefront set is $(\underbrace{2,\dots,2}_\ell,\underbrace{1,\dots,1}_{2n-2\ell})$ for some $0 \leq \ell \leq n$. Let $P = P_{n,n}$. Then the maximal $L$-orbit in $\mathrm{pr}(\WF_G(\ov{\pi}))$ is of rank $\ell$. Consequently, $\pi$ has a degenerate Whittaker model corresponding to this $L$-orbit. By \cite[Theorem A]{GGS17}, $\pi$ has a generalized Whittaker model corresponding to the nilpotent orbit of an element in this $L$-orbit, which is exactly $(\underbrace{2,\dots,2}_\ell,\underbrace{1,\dots,1}_{2n-2\ell})$.
\end{example}

From now on, we assume $\pi \in \widehat{G}$ and $G=G_n$ is a classical group introduced in subsection~\ref{derivative-sec}. Consider the Siegel parabolic subgroup $P$ when $G$ is of Type II and $P = P_{[\frac{n}{2}], n-[\frac{n}{2}]}$ when $G = \mathrm{GL}_n$. By direct integral theory, $\ov{\pi}$ as a $U$-representation is determined by a projection-valued measure $\mu_{\pi}$ on $\widehat{U}$.

Let $\beta$ be an $L$-invariant subset of $\widehat{U}$. Define $\widehat{G}_{\beta}$ to be the subset of $\widehat{U}$ consisting of representations $\pi$ such that
\[
\supp(\mu_{\pi}) \subset \beta.
\]
If $\beta$ is the union of open $L$-orbits on $\widehat{U}$, then we will simply denote $\widehat{G}_o$. We use $\mathscr{S}_c$ to denote the set of non-open orbits. Then we have the following disjoint union decomposition for $\widehat{G}$:
\[
\widehat{G} = \widehat{G}_o \bigsqcup \left( \bigsqcup_{\beta \in \mathscr{S}_c} \widehat{G}_{\beta} \right),
\]
see \cite[Theorem 3.1]{Li} for Type II classical groups and \cite[Theorem 3.6]{Sca90} for $\GL_n$. By the correspondence of the unitary representation and the imprimitive system, we have $\ov{\supp(\mu_{\pi})}=\supp_U(\pi)$. 
\begin{definition}
Let $\pi$ be an irreducible representation of $G$. We call it a low rank representation if $\supp_U(\pi)$ does not contain any open $L$-orbit.
\end{definition}

In \textit{loc. cit.}, the unitary low rank representations are explicitly constructed as theta lifts from a dual pair in the stable range. However, for general low rank representations, such a straightforward classification does not hold. For instance, there exist irreducible finite-dimensional representations of $\GL_n$ that cannot be realized as theta lifts from a dual pair in the stable range. However, it is still hopeful to prove some partial results.
\begin{conjecture}
Let $(G', G)$ be a Type I dual pair in the stable range with $G'$ being the smaller one. If $\tau$ is an irreducible representation of $G'$, then $\theta(\tau)$ is a low rank representation of $G$.
\end{conjecture}

\subsection{Rearrangement of the filtration}
This subsection applies the abstract results of Subsection~\ref{sec-Schwartz fun} to rearrange the coarse spectral filtration. This process will separate the successive quotients, isolating those of the form $I_0$ as a quotient and those of the form $I_l$ (for $l \neq 0$) as a subrepresentation. It is a crucial ingredient for the proof of Theorem~\ref{intro-CW} and the general canonical filtration in~\cite{CWYZ}.
\begin{proposition}\label{rearrange prop}
    Let $\pi$ be a principal series representation of $\GL_n$. Then the restriction $\pi|_{P_{n-k,k}}$ admits a filtration
    \[
    \pi|_{P_{n-k,k}} = \sigma_0 \supset \sigma_1 \supset \dots \supset \sigma_m = 0
    \]
    as in Definition~\ref{def_fil} such that:
    \begin{enumerate}
        \item it satisfies the condition in Proposition~\ref{rest_max};
        \item there exists an integer $s$ with $0 \leq s \leq m$ such that the successive quotients in $\sigma_0/\sigma_s$ are of the form $I_0$, while the successive quotients in $\sigma_s/\sigma_m$ are of the form $I_l$ for some $l \neq 0$.
    \end{enumerate}
\end{proposition}
\begin{proof}
    Let $\CA$ be the commutative Fr\'echet algebra $\CS(\widehat{U_{n-k,k}}\setminus\{0\})$. It is stable under the co-adjoint action of $P_{n-k,k}$. Hence, the result follows from Lemma~\ref{limit diff} and Lemma~\ref{split}.
\end{proof}

\section{Casselman-Wallach property of functor $B^k$}
In this section, we apply the BZ-filtration to give an affirmative answer to an open question in \cite[3.1.(1)]{AGS15a}. In fact, our result is a generalization of the open question. We show that for a Casselman-Wallach representation $\pi$ of $\GL_n$, the derived functor $L^iB^k(\pi)$ is a Casselman-Wallach representation of $\GL_{n-k}$. This property is critical for proving the homological branching law in the next section. We provide a proof sketch in the first subsection.
\begin{theorem}\label{haus thm}
    Let $\pi$ be a Casselman-Wallach representation of $\GL_n$, then $L^iB^k(\pi)$ is a Casselman-Wallach representation of $\GL_{n-k}$ for every integer $0\leq k\leq n$ and integer $i$. In particular, $L^iB^k(\pi)$ is Hausdorff.
\end{theorem}

\subsection{Sketch of the proof}\label{sketch subsection}
Our proof proceeds in the following steps.
\begin{itemize}
    \item \textbf{Step 1}: We reduce the problem to prove that for every principal series $\pi$, $\rmh_i(\fku_{n-k,k},\pi)$ is Casselman-Wallach for every integer $i$ and $0<k<n$. From now on, following the notation in subsection~\ref{category C}, the parabolic subgroup $P$ we concern in this section is $P_{n-k,k}$, with the standard Levi subgroup $L=\GL_{n-k}\times\GL_k$ and the unipotent radical $U=U_{n-k,k}$.
    \item \textbf{Step 2}: We realize the principal series as $\pi=\tau\times_u\chi$, where $\tau$ is a principal series of $\GL_{n-1}$ and $\chi$ is a character. Equivalently, it is the space of Schwartz sections of a tempered bundle on $P_{n-1,1}\backslash \GL_n$. The $P_{n-k,k}$ has a unique open orbit and a unique closed orbit on $P_{n-1,1}\backslash \GL_n$, which leads to a short exact sequence
    \[
    0\lra \pi_o\lra \pi|_{P_{n-k,k}}\lra \pi_c\lra 0.
    \]
     By Lemma~\ref{exa-haus-lem}, it suffices to prove that $\rmh_i(\fku_{n-k,k},\pi_o)$ and $\rmh_i(\fku_{n-k,k},\pi_c)$ are Casselman-Wallach $\GL_{n-k}\times \GL_k$-representations.  We proceed by induction on $n$. The base case $n=2$ follows from the comparison theorem for minimal parabolic subgroups, see for example \cite[Theorem 5.2]{LLY21}. Assuming the statement for $n-1$, we prove it for $n$. 
    \item \textbf{Step 3}: For $\pi_o\simeq \chi\mind_u \tau|_{P_{n-k-1,k}}$, we establish the isomorphism 
    \[
    \rmh_i(\fku_{n-k,k},\pi_o)\simeq \chi\mind_u\rmh_i(\fku_{n-k-1,k},\tau).
    \]
    The Casselman-Wallach property for $\rmh_i(\fku_{n-k,k},\pi_o)$ then follows from the induction hypothesis on $n$.
    \item \textbf{Step 4}: We analyze $\pi_c$. Note that $\pi_c$ and $\pi$ have the same infinitesimal character. We first establish the result for the case $k=1$, although this case can also be proved using the argument for general $k$. In this case, we can directly compute $\pi_c$ via its strong dual and demonstrate that $\pi_c \in \CC(\mathfrak{g}, L)_f$. The underlying reason that $\pi_c\in\CC(\fkg,L)$ is that the BZ-filtration of $\pi_c$ is composed of trivial extension spectrum. For general $k$, we apply the Casselman-Jacquet functor to eliminate the non-trivial extension spectrum. 
    
   For general $k$, note that $\pi_c$ admits a decreasing Borel filtration indexed by non-negative integers
    \[
    \pi_c=(\pi_c)_0\supset (\pi_c)_1\supset \dots .
    \]
    We show that $\rmh_i(\fku_{n-k,k},(\pi_c)_j/(\pi_c)_{j+1})$ is Casselman-Wallach for $i=0$ and every non-negative integer $j$ by the induction hypothesis on $n$. Therefore, $\rmh_0(\fku_{n-k,k},\pi_c)$ is Hausdorff by Lemma~\ref{ext_Hausd}. Inductively, we demonstrate that $\pi_c/\fku_{n-k,k}^\ell\pi_c$ is Casselman-Wallach for every positive integer $\ell$. Then by Proposition~\ref{surj_CJ}, we have short exact sequence 
    \[
     0\lra \Ker \varphi \lra \pi_c \stackrel{\varphi}{\lra} \CJF(\pi_c)\lra 0.
    \]
    By Proposition~\ref{rest_max}, $\pi_c$ admits a coarse spectral filtration. We show that the induced filtration on $\Ker\varphi$ does not contain trivial extension spectrum since the weight of trivial extension terms has a lower bound.
    Consequently, we have 
    \[
    \rmh_i(\fku,\pi_c)\simeq  \rmh_i(\fku,\CJF(\pi_c))
    \]
    for every integer $i$, and $\rmh_i(\fku,\CJF(\pi_c))$ is Casselman-Wallach since $\CJF(\pi_c)$ belongs to $\CC(\fkg,L)_f$, see subsection~\ref{category C}.
\end{itemize}

\subsection{Reduction}
We start proving Theorem~\ref{haus thm} by reducing the problem to the case of principal series. Fix an integer $k$ with $1\leq k\leq n$.
\begin{lemma}\label{red to prin}
    Suppose that for every principal series $I$ of $\GL_n$ and every integer $i$, $L^iB^k(I)$ is Casselman-Wallach. Then for every Casselman-Wallach representation $\pi$ of $\GL_n$ and every integer $i$, $L^iB^k(\pi)$ is Casselman-Wallach.
\end{lemma}
\begin{proof}
    Every generalized principal series admits a filtration whose successive quotients are principal series. Thus, by Lemma~\ref{exa-haus-lem},  $L^iB^k(J)$ is Casselman-Wallach for every generalized principal series $J$ of $\GL_n$ and every integer $i$. By the Casselman embedding theorem, we have a resolution of $\pi$ by generalized principal series:
    \[
    \pi \lra J_0\lra J_1\lra \dots
    \]
    Consider the $\fkv_{n-k+1}$-Koszul resolution $P_{i,\bullet}$ of each $\Psi^{k-1}(J_i)$, we get a double complex $P_{\bullet,\bullet}$. Since $\Psi$ is exact, we have
    \[
    \rmh_i(\Tot(P_{\bullet,\bullet}))\simeq L^iB^k(\pi).
    \]
     On the other hand, we have a decreasing filtration $\CF^{\bullet}$ of total complex 
    \[
    \CF^j=F^j(\Tot(P_{\bullet,\bullet})):= \Tot(P_{\geq j,\bullet}).
    \]
    Since the Koszul resolution is finite length, we can find a large enough $m$, such that for every $i$,
    \[
     \rmh_i(\CF^0/\CF^m)\simeq L^i B^k(\pi).
    \]
    
Note that $\rmh_i(\CF^j/\CF^{j+1})\simeq L^iB^k(J_{j})$ is Casselman-Wallach for every $j$. Thus, inductively, we consider the exact sequence of complexes:
    \begin{equation*}
        0\lra \CF^j/\CF^{j+1}\lra \CF^{j-r+1}/\CF^{j+1}\lra\CF^{j-r+1}/\CF^{j}\lra 0 .
    \end{equation*} 
    By Remark~\ref{rem_haus}, we can show $\rmh_i(\CF^{j-r+1}/\CF^{j+1})$ is Casselman-Wallach for every integer $i,j$ and $r\geq 1$.
\end{proof}

On the other hand, by Remark~\ref{B^k rem}, we have $L^iB^k_0(\pi)= \Psi_0^{k-1}\rmh_i(\fku_{n-k,k},\pi)$ since $\Psi_0$ is exact. Let $\beta$ be an irreducible representation of $\GL_{n-k}\times\GL_k$, then $\beta\simeq \beta_1\widehat{\otimes}\beta_2$, where $\beta_1$ is an irreducible representation of $\GL_{n-k}$ and $\beta_2$ is an irreducible representation of $\GL_k$. Therefore, 
\[
 \Psi^{k-1}_0(\beta)=\beta_1\otimes \Psi_0^{k-1}(\beta_2),
\]
which is a Casselman-Wallach representation of $\GL_{n-k}$ since $\Psi_0^{k-1}(\beta_2)$ is finite dimensional. Consequently, to prove Theorem~\ref{haus thm}, we need only to prove that $\rmh_i(\fku_{n-k,k},\pi)$ is a Casselman-Wallach representation for every integer $i$ and every principal series $\pi$.

\subsection{Open orbit}
In this subsection, we prove \textbf{step 3} in the subsection~\ref{sketch subsection}. Let $m$ and $k$ be two positive integers such that $n=k+m$ and $k<n-1$.
\begin{proposition}\label{Jacquet-open}
    Let $\chi$ be a character of $\GL_1$, and $\sigma$ be a representation of $P_{m-1,k}$. If $\rmh_i(\fku_{m-1,k},\sigma)$ is Hausdorff for every integer $i$, then for every integer $i$, we have a natural isomorphism as $\GL_{n-k}\times \GL_k$-representations
    \[
    \rmh_i(\fku_{n-k,k},\chi\mind_u\sigma)\simeq \chi\mind_u \rmh_i(\fku_{m-1,k},\sigma).
    \]
\end{proposition}
\begin{proof}
\textbf{Step 1.} This statement can be reduced to $i=0$. Assume that it holds for $i=0$, let us show the statement for $i>0$. We first show if $P_{\bullet}$ is a $U_{m-1,k}$-strong projective resolution of $\sigma$, then $\chi\mind_u P_{\bullet}$ is a $\fku_{n-k,k}$-acyclic resolution of $\chi\mind_u \sigma$. The $\chi\mind_u \sigma$ is realized as Schwartz sections of some tempered bundle $\CE$ over 
$$X:=\ov{P_{1,n-1}}\cap P_{m,k}\backslash P_{m,k}.$$ 
Note that $X$ is a fiber bundle 
\[
 X\lra X/U_{m,k}\simeq \ov{P_{1,m-1}}\backslash \GL_{m}
\]
such that the fiber is isomorphic to $\mathbf{k}^{k}$. Let $U$ be the unipotent radical of $P_{1,n-1}$. Then $\{U\cdot w_i\mid w_i=(1,i),1\leq i\leq m\}$ is an affine open covering of $X$, such that $X$ trivializes over each $U\cdot w_i$. Here $(1,i)$ is the corresponding permutation matrix. Let $I$ be a subset of $\{1,\dots,m\} $, we define
\[
\CS_I:= \CS(\bigcap_{i\in I} U\cdot w_i,\CE).
\]
Since the Schwartz functions over a Nash manifold compose a co-sheaf, we have \v{C}ech resolution
\[
\lra  \bigoplus_{|I|=k} \CS_I\lra \bigoplus_{|I|=k-1} \CS_I\lra \dots\lra \bigoplus_{|I|=1} \CS_I\lra \CS(X,\CE)\lra 0
\]

To show that each of $\chi\mind_u P_{\bullet}$ is $\fku_{n-k,k}$-acyclic, it is equivalent to show that when $\sigma$ is a relative projective object, 
\begin{itemize}
\item[(i)] $\rmh_i(\fku_{n-k,k},\CS_I)=0$ for $|I|\geq 1$ and $i\geq 1$. Thus the \v{C}ech resolution is acyclic, and we can use it to compute $\rmh_i(\fku_{n-k,k},\chi\mind_u\sigma)$.
\item[(ii)] $\rmh_l(\rmh_0(\fku_{n-k,k},\oplus_{|I|=\bullet}\CS_I))=0$ for $l\geq 1$. 
\end{itemize}

For (i), we prove $\rmh_i(\fku_{n-k,k},\CS(U,\CE))=0$ for $i\geq 1$. Other cases are exactly the same. By spectral sequence, we have
\[
\rmh_p(\fku_{m-1,k},\rmh_q(\fku_{n-k,k}\cap \fku,\CS(U,\CE)))\Rightarrow \rmh_{p+q} (\fkv_n,\CS(U,\CE)).
\]
Here $\fku_{m-1,k}$ is embedded as a subalgebra 
$\begin{pmatrix}
   0& & 0_{1\times k}\\
    & 0_{(m-1)\times(m-1)}& *\\
    & & 0_{k\times k}
\end{pmatrix}$ of $\fku_{n-k,k}$. As a $\fku_{n-k,k}\cap \fku$ representation, we have $\CS(U,\CE)\simeq \CS(V,\CE)\widehat{\otimes} \CS(U_{n-k,k}\cap U)$. Here $V$ is the unipotent radical of $P_{1,m-1}$. Hence 
\begin{equation*}
    \rmh_q(\fku_{n-k,k}\cap \fku,\CS(U,\CE))\simeq \CS(V,\CE)
\end{equation*}
when $q=0$, and otherwise is zero. In addition, when $q=0$, such an isomorphism intertwines the $\fku_{m-1,k}$-action, where $\fku_{m-1,k}$-action on $\CS(V,\CE)$ is only on the fiber of $\CE$. Since $\sigma$ is projective,
\[
 \rmh_p(\fku_{m-1,k},\rmh_0(\fku_{n-k,k}\cap \fku,\CS(U,\CE)))\simeq \CS(V,\rmh_p(\fku_{m-1,k},\CE))=0
\]
when $p\geq 1$. 

For (ii), we realize $\chi\mind_u \rmh_0(\fku_{m-1,k},\sigma)$ as Schwartz sections of some tempered bundle $\CE'$ over $X':=\ov{P_{1,m-1}}\backslash \GL_{m}$. And $\CS_I'$ is defined as
\[
 \CS(\bigcap_{i\in I} V\cdot w_i,\CE').
\]
Then we have $\rmh_0(\fku_{n-k,k},\CS_I)\simeq \CS_I'$ since $\rmh_0(\fku_{m-1,k},\sigma)$ is Hausdorff, which implies $\chi\mind_u \sigma$ is acyclic.

Consequently, since $\rmh_i(\rmh_0(\fku_{n-k,k},P_{\bullet}))$ is Hausdorff,
\[
\rmh_i(\fku_{n-k,k},\chi\mind_u\sigma)\simeq  \rmh_i(\rmh_0(\fku_{n-k,k},\chi\mind_u P_{\bullet})) \simeq \rmh_i(\chi\mind_u\rmh_0(\fku_{n-k,k}, P_{\bullet}))
\] 
is Hausdorff, where the second isomorphism is our assumption. Then the result follows from 
\[
\chi\mind_u \rmh_i(\fku_{m-1,k},\sigma)\simeq  \chi\mind_u\rmh_i(\rmh_0(\fku_{n-k,k}, P_{\bullet})) \simeq \rmh_i(\chi\mind_u\rmh_0(\fku_{n-k,k}, P_{\bullet})).
\]
where the second isomorphism follows from the exactness of parabolic induction. 
\textbf{Step 2.} We prove the statement when $i=0$. There is a natural map 
\[
 \Gamma:\chi\mind_u\sigma\lra  \chi\mind_u \rmh_0(\fku_{m-1,k},\sigma)
\]
given by 
\[
f\mapsto (g\mapsto \ov{\int_{U_{n-k,k}}f(gv)dv}), \ f\in\chi\mind_u\sigma,\  g\in \GL_{n-k},
\]
and $\overline{\bullet}$ is the projection of $\chi\boxtimes \sigma$ to $\chi\boxtimes \rmh_0(\fku_{m-1,k},\sigma)$. It is easy to verify that $\Gamma$ factors through $\rmh_0(\fku_{n-k,k},\chi\mind_u\sigma)$, which we still denote by $\Gamma$ and is surjective by definition. By calculation in step 1, we have the following commutative diagram
\begin{center}
\begin{adjustbox}{scale=0.95}
$
\begin{tikzcd}
    \rmh_0(\fku_{n-k,k},\bigoplus_{|I|=2}\CS_I )\ar[r]\ar[d,"\simeq"] &\rmh_0(\fku_{n-k,k},\bigoplus_{|I|=1} \CS_I  )\ar[r]\ar[d,"\simeq"]& \rmh_0(\fku_{n-k,k},\CS(X,\CE))\ar[r]\ar[d,"\Gamma"]& 0\\
    \bigoplus_{|I|=2}\CS_I '\ar[r]&\bigoplus_{|I|=1} \CS_I' \ar[r]&\CS(X',\CE')\ar[r]& 0 .
\end{tikzcd}
$
\end{adjustbox}
\end{center}

The horizontal lines of the commutative diagram are both exact, and the first two vertical lines are isomorphisms. Hence the last vertical line is an isomorphism as well.
\end{proof}

When $k=n-1$, we have $\pi_o\simeq \chi\mind \tau$. We can realize $\pi_o$ as $\CS(U_{1,n-1},\chi\boxtimes\tau)$ such that the $U_{1,n-1}$-action is given by translation. Consequently, at this time,
\[
\rmh_i(\fku_{1,n-1},\pi_o)=\begin{cases}
    \chi\boxtimes \tau, &\text{ for }i=0 \\
    0, &\text{ otherwise}.
\end{cases}
\]

\subsection{Closed orbit for $k=1$}
In this subsection, we assume $k=1$ and prove that $\rmh_i(\fku,\pi_c)$ is Casselman-Wallach for $\pi=\tau\times_u\chi$. The key point is that the strong dual of $\pi_c$ is relatively easy to calculate. In addition, $\pi_c$ is nuclear Fr\'echet, hence reflexive( see \cite[Appendix A]{CHM}). By dualizing $\pi_c'$, we observe that $\pi_c$ lies in category $\CC(\fkg,L)_f$, in which we can apply the general result Proposition~\ref{homo-CW}.\\
 
Before discussing $\pi_c'$, we recall some functional analysis. We equip $\U[[\fku]]$ with inverse limit topology, which is nuclear Fr\'echet. Moreover, by checking the definition, we will find that the strong topology and weak topology coincide in
 \[
 \Hom_{cts}(\U[[\fku]]',\BC)\simeq\Hom_{cts}(\U(\ov{\fku}),\BC),
 \]
where the isomorphism comes from the Killing form.
 Moreover, the weak dual of $\tau'$ is metrizable when $\tau$ is a Casselman-Wallach representation of $L$. Hence by \cite[Theorem 34.1]{Tr}, we have 
 $\U[\ov{\fku}]\otimes \tau'$ is complete under $\epsilon$-topology. In other words,
 \[
 \U[\ov{\fku}]\otimes \tau'= \U[\ov{\fku}]\widehat{\otimes} \tau'.
 \]
 Here, we do not distinguish $\epsilon$-completion or projective completion since $\U[\ov{\fku}]$ and $\tau'$ are nuclear. The following theorem holds for general real reductive group $G$ and parabolic subgroup $P$.
\begin{lemma}
Let $\tau$ be a Casselman-Wallach representation of $L$. Let $\pi=\mathrm{Ind}_P^G(\tau)$ be the induced representation of $G$. Consider $P$-orbit on $P\backslash G$, and let $\pi_c$ be defined as in section~\ref{sec-Schwartz fun}. Then we have
\[\pi_c'\simeq \U(\mathfrak{g})\widehat{\otimes}_{U(\mathfrak{p})} \tau'\] as topological $\U(\mathfrak{g})$-modules, where $\tau'$ is regarded as $\mathfrak{p}$-module by trivial extension on $\mathfrak{u}$.
\end{lemma}

\begin{proof}
We first define a map $\vartheta:\U(\mathfrak{g})\otimes \tau'\to \pi'$ by
\[x\otimes y(f):=\frac{d}{dt}\big\vert_{t=0}y(f(\exp(tx)^{-1})), \ x\otimes y\in \fkg\otimes \tau', f\in \pi. \]
The support of distribution $\vartheta(x\otimes y)$ is contained in the closed orbit. Moreover, for every $z\in \fkp$, one has 
\[xz\otimes y (f)=\frac{d}{dt}\big\vert_{t=0}y(f(\exp(tz)^{-1}\exp(tx)^{-1}))=x\otimes zy(f),\] 
where the last equality follows from $f(\exp(tz)^{-1}g)=\tau(\exp(tz)^{-1})f(g)$. Consequently, $\vartheta$ descends to a continuous map, which we still denote by $\vartheta$
\[
   \vartheta:\U(\mathfrak{g})\widehat{\otimes}_{\U(\fkp)} \tau'\to \pi_c'.
\]
By \cite[Lemma 2.4]{LLY21}, $\vartheta$ is a bijection. Hence $\vartheta$ is a topological isomorphism by the open mapping theorem for dual nuclear Fr\'echet space.

\end{proof}

Since $\pi_c$ is reflexive, we have isomorphism as topological $L$-representations by the Killing form
\[
 \pi_c\simeq (\U(\ov{\fku})\widehat{\otimes} \tau')'\simeq \U[[\fku]]\widehat{\otimes} \tau.
\]
We observe that $\pi_c$ falls in category $\CC(\fkg,L)$. In addition, $\pi_c$ has infinitesimal character since $\pi$ has, which implies the following lemma by Lemma~\ref{finite length lem}.

\begin{lemma}
$\pi_c$ is an object in category $\CC(\fkg,L)_f$. 
\end{lemma}

\begin{corollary}\label{closed orb CW}
For every integer $i$, $\rmh_i(\mathfrak{\fku},\pi_c)$ is a Casselman-Wallach $L$-representation.  
\end{corollary}
\begin{proof}
It is a direct consequence of Proposition~\ref{homo-CW}.
\end{proof}

\subsection{Closed orbit for general $k$}
For $\pi=\tau\times_u\chi$, we now consider the unique closed orbit of $P_{n-k,k}$ and its corresponding representation $\pi_c$. By Borel's Lemma, $\pi_c$ admits a decreasing filtration indexed by non-negative integers,
    \[
    \pi_c=(\pi_c)_0\supset (\pi_c)_1\supset \dots ,
    \]
such that for every $j$, 
\[
(\pi_c)_j/(\pi_c)_{j+1}\simeq \tau_j|_{P_{n-k,k-1}}\times_u \chi_j
\]
for some principal series $\tau_j$ of $\GL_{n-1}$ and some character $\chi_j$.  The proof of the following lemma is exactly the same as Step 2 of Proposition~\ref{Jacquet-open}.
\begin{lemma}
   Let $s$ and $m$ be positive integers with $s+m+1=n$. Let $\sigma$ be a representation of $P_{s,m}$, and let $\chi$ be a character. Suppose $\rmh_0(\fku_{s,m},\sigma)$ is Hausdorff, then the quotient map 
    \[
   \sigma\times_u\chi \lra \rmh_0(\fku_{s,m},\sigma)\times_u\chi
    \]
    induces an isomorphism of $\GL_s\times \GL_{m+1}$-representations
    \[
   \rmh_0(\fku_{s,m+1} ,\sigma\times_u\chi )\simeq \rmh_0(\fku_{s,m},\sigma)\times_u\chi.
    \]
    \end{lemma}
By the induction hypothesis on $n$, the above lemma implies that 
$$\rmh_0(\fku_{n-k,k},(\pi_c)_j/(\pi_c)_{j+1})$$ 
is Casselman-Wallach for every non-negative integer $j$. Therefore, $\rmh_0(\fku_{n-k,k},\pi_c)$ is Hausdorff by Lemma~\ref{ext_Hausd}, which is equivalent to that $\fku_{n-k,k}\pi_c$ is closed in $\pi_c$. 

On the other hand, it is well-known that for a Casselman-Wallach representation $\pi$, $\pi/\ov{\fku \pi}$ is a Casselman-Wallach $L$-representation. Hence, the continuous surjection
\[
 \pi/\ov{\fku\pi} \lra \pi_c/\ov{\fku\pi_c}=\pi_c/\fku\pi_c
\]
implies that $\pi_c/\fku \pi_c$ is a Casselman-Wallach representation as well. Consequently, by Remark~\ref{CW-ass-weaker}, $\pi_c$ satisfies condition~\ref{CW_assump}. This implies that $\CJF(\pi_c)\in\CC(\fkg,L)_f$ and an exact sequence
\[
 0\lra \Ker \varphi \lra \pi_c \stackrel{\varphi}{\lra} \CJF(\pi_c)\lra 0.
\]

Recall from Section~\ref{res-max-sec} that $\tau_j|_{P_{n-k,k-1}}$ admits a coarse spectral filtration for every $j$. Thus, $\pi_c$ inherits a coarse spectral filtration. By Proposition~\ref{rearrange prop}, we can assume that there is a closed subspace $\beta$ in the filtration such that the successive quotient in $\beta$ is of the form $I_l,l\neq 0$, and the successive quotient in $\pi_c/\beta$ is of the form $I_0$.  Let $\mathscr{T}$ be the set of successive quotients in the coarse spectral filtration of $\pi$, which is of the form $I_0$. We simply denote $\omega_{\sigma}:=\omega_{\tau}$ when $\sigma=I_0(\tau)$ for some irreducible $\GL_{n-k}\times \GL_k$-representation $\tau$. Let $\mathscr{T}_s$ be the subset of $\mathscr{T}$ consisting of successive quotients $\sigma$ such that
    \[
     s\leq  \mathrm{Re}(\omega_{\sigma})\leq s+1,
    \]
    where $s\in\BZ$. 
\begin{lemma}\label{CJ-dist}
    Let $\sigma=\pi_c^{\flat}/\pi_c^{\sharp}$, where $\pi_c^{\sharp}\subset\pi_c^{\flat}$ are successive closed subspaces in the coarse spectral filtration of $\pi_c$. Then 
    \begin{enumerate}
        \item If $\sigma=I_l(\sigma_0)$ for some positive integer $l$ and some $S_l^n$-representation $\sigma_0$, then 
        \[
        \left(\Ker\varphi\cap \pi_c^{\flat}\right)/\left(\Ker\varphi\cap \pi_c^{\sharp}\right)\simeq \sigma;
        \]
        \item If $\sigma=I_0(\tau)$ for some irreducible $L$-representation $\tau$, then 
        \[
        \left(\Ker\varphi\cap \pi_c^{\flat}\right)/\left(\Ker\varphi\cap \pi_c^{\sharp}\right)= 0.
        \]
    \end{enumerate}
\end{lemma}
\begin{proof}
    We first prove (1). Note that (1) is equivalent to the following equality
    \[
    \pi_c^{\sharp}+\bigcap_j \left( \fku^j\pi_c\cap \pi_c^{\flat}\right)=\pi_c^{\flat}.
    \]
    By definition, the left-hand side is contained in the right-hand side. We prove the reverse inclusion. By Corollary~\ref{bun ver}, we have 
    \[
    \rmh_0(\fku,\sigma)=0.
    \]
    In other words, $\fku\sigma=\sigma$. Hence,  
    \[
  \bigcap_j\fku^j \pi_c^{\flat}\bigg/\left(\bigcap_j\fku^j \pi_c^{\flat}\cap \pi_c^{\sharp}\right)= \bigcap_j\fku^j (\pi_c^{\flat}/\pi_c^{\sharp})=\pi_c^{\flat}/\pi_c^{\sharp}.
    \]
    Consequently, the result follows from
    \[
    \pi_c^{\sharp}+\bigcap_j \left(\fku^j\pi_c\cap \pi_c^{\flat}\right)\supset \pi_c^{\sharp}+\bigcap_j\fku^j\pi_c^{\flat}=\pi_c^{\flat}.
    \]

    We proceed to prove (2). Proposition~\ref{rest_max} allows us to define
    \[
    \Omega_{\pi_c}:=\min\{\mathrm{Re}\, \omega_{\sigma}\mid I_0(\sigma) \text{ is a successive quotient in the coarse spectral filtration of }\pi_c\}.
    \]
     Let $k$ be an integer such that
    \[ \mathrm{Re}\,\omega_{\tau} <2k+\Omega_{\pi_c}.
    \]
    We claim that
    \begin{equation}\label{equality}
          \pi_c^{\sharp}\cap \fku^k\pi_c=\pi_c^{\flat}\cap \fku^k\pi_c,
    \end{equation}
    which will imply the result. Suppose that equation~\ref{equality} does not hold, then there exists an element $v\in \pi_c^{\flat}\cap \fku^k\pi_c$, whose image in $\pi^{\flat}_c/\pi_c^{\sharp}$ is non-zero. Thus, we can replace $\pi_c$ by $\pi_c/\beta$.
    
   By the proof of Proposition~\ref{rest_max}, it is not hard to show that there exists a polynomial $p$ such that
   \[
   \mathrm{Card}(\mathscr{T}_s)\leq p(s)
   \]
   for any integer $s$. Consequently, the infinite product
   \[
  \mathbb{T}:= \prod_{\sigma\in \mathscr{T}} \frac{z^{\ell}-(\mathrm{Re}(\omega_{\sigma})+2k)^{\ell}}{\mathrm{Re}(\omega_{\pi_c^{\flat}/\pi_c^{\sharp}})^{\ell}-(\mathrm{Re}(\omega_{\sigma})+2k)^{\ell}}
   \]
   is convergent for sufficiently large $\ell$ when $z=\begin{pmatrix}
       I_{n-k} & \\
       & -I_k
   \end{pmatrix}\in\Lie(\GL_n)$ by the moderate growth condition. This leads to a contradiction since $\mathbb{T}(\fku^k \pi_c)=0$ but $\mathbb{T}$ acts on $\pi_c^{\flat}/\pi_c^{\sharp}$ by identity.
\end{proof}

Now we can prove that $\rmh_i(\fku,\pi_c)$ is Casselman-Wallach for every integer $i$. By Corollary~\ref{bun ver}, if $\sigma=I_l(\sigma_0)$ for some positive integer $l$ and some $S_l^n$-representation $\sigma_0$, then  $\rmh_i(\fku,\sigma)=0$ for every integer $i$.
Hence, by Lemma~\ref{CJ-dist} and Lemma~\ref{exa-BZ}, we have
\[
\rmh_i(\fku,\Ker\varphi)=0\text{ and } \rmh_i(\fku,\pi_c)\simeq \rmh_i(\fku,\CJF(\pi_c)) 
\]
for every integer $i$. The result now follows from Proposition~\ref{homo-CW}.

\begin{remark}
    Lemma~\ref{CJ-dist} shows that the Casselman-Jacquet functor can separate the trivial extension spectrum and non-trivial extension spectrum. This ideal can be applied to general spectral decomposition, which will be explored in future work. In particular, similarly as $\pi_c$ we can prove that 
    \[
    \rmh_i(\fku,\pi)\simeq \rmh_i(\fku,\CJF(\pi))
    \]
    for a Casselman-Wallach representation $\pi$ of $\GL_n$.
\end{remark}

\subsection{Proof of Theorem~\ref{intro-CW}}
In this subsection, we prove Theorem~\ref{intro-CW}. Let $G$ be a real reductive group. We will show that, in general, to prove $\rmh_i(\fku,\pi)$ is Casselman-Wallach for every Casselman-Wallach representation $\pi$ of $G$ and every parabolic subgroup $P=LU$, it suffices to prove for maximal parabolic subgroups.
\begin{lemma}
    Let $P=LU$ be a parabolic subgroup of $G$, and let $Q=MV$ be a parabolic subgroup of $L$. If $\rmh_i(\fkv,\tau)$ is Casselman-Wallach for every Casselman-Wallach representation $\tau$ of $L$, and $\rmh_i(\fku,\pi)$ is Casselman-Wallach for every Casselman-Wallach representation $\pi$ of $G$, then $\rmh_i(\fku+\fkv,\pi)$ is a Casselman-Wallach $M$-representation for every Casselman-Wallach representation $\pi$ of $G$.
\end{lemma}
\begin{proof}
    Consider the double complex given by the Koszul resolution
    \[
    P_{p,q}:=\wedge^p \fkv\otimes \wedge^q \fku\otimes \pi,
    \]
    then 
    \[
    \rmh_i(\Tot(P_{\bullet,\bullet}))=\rmh_i(\fku+\fkv,\pi).
    \]
    The total complex admits a finite increasing filtration $\CF^j:=\Tot_{p\leq j,\bullet}$ with 
    \[
   E_1^{p,q}= \rmh_q(\CF^p/\CF^{p-1})=\wedge^{p}\fkv\otimes \rmh_q(\fku,\pi).
    \]
    Thus, $E_1^{p,q}$ is Hausdorff, and 
    \[
    E_2^{p,q}=\rmh_p(\fkv,\rmh_q(\fku,\pi))
    \]
    is a Casselman-Wallach representation of $M$ for every integer $p$ and $q$. Consequently,  $E_r^{p,q}$ is a Casselman-Wallach representation of $M$ for every $r\geq 2$ and every integer $p,q$, as $d_r^{p,q}$ is continuous. The result then follows from Lemma~\ref{spec-haus}. 
\end{proof}

\section{Homological Branching Law for $(\GL_{n+1},\GL_n)$}
This section shows that once the Bernstein-Zelevinsky theory is developed, computing the Euler-Poincar\'e characteristic is much more straightforward than determining the $\Hom$-space for the pair $(\GL_{n+1},\GL_n)$. Moreover, we can explore some higher extension vanishing results, which lead to conclusions of the $\Hom$-space.
\subsection{Euler-Poincar\'e characteristic formula}\label{EP_sec}
Recall the definition of the Whittaker model:
\begin{definition}
    Let $\pi$ be a Casselman-Wallach representation of a real reductive group $G$, and let $\theta$ be a non-degenerate unitary character of $U^0$. Define the multiplicity of the Whittaker model as
    \[
    \Wh(\pi):=\dim \Hom_{U^0}(\pi,\theta).
    \]
\end{definition}
 By \cite{CHM}, $\Wh(\pi)$ is finite and is independent of the choice of minimal parabolic subgroup or $\theta$. There is another point of view of $\Wh(\pi)$ as follows. Consider the \textbf{Gelfand-Graev representation}  $\SInd_{U^0}^G(\theta)$, then Shapiro's lemma implies 
 \[ \Wh(\pi)=\dim \Hom_G(\pi\widehat{\otimes}\SInd_{U^0}^G(\theta),\BC).\]
The technique of Lemma~\ref{lem-MVW} yields a concise proof of the following result.
    \begin{lemma}\label{dual-wh}
    Let $\pi$ be a Casselman-Wallach representation of $\GL_n$, then
    \[
    \dim\Psi_0^{n-1}(\pi)=\dim\Psi^{n-1}_0(\pi^{\vee}).
    \]
\end{lemma}
\begin{proof}
    By Lemma~\ref{lem-MVW},
    \[
    \Psi_0^{n-1}(\pi^{\vee})=\pi/\Span\{\kappa\cdot v-\theta(\kappa)v\mid \kappa\in\ov{ \fkn_n},v\in\pi\}
    \]
    for some non-degenerate unitary character $\theta$ of $\ov{N_n}$. Since the multiplicity of the Whittaker model is independent of the choice of minimal parabolic subgroup or non-degenerate unitary character, the statement follows.
\end{proof}
For a representation $\sigma$ of $M_n$ that admits a BZ-filtration with finite bottom layer, we can also define the multiplicity of the Whittaker model as
 \[
    \Wh(\sigma):=\dim \Hom_{N_n}(\sigma,\theta),
    \]
    where $\theta$ is a non-degenerate unitary character of $N_n$. This multiplicity is finite and independent of the choice of $\theta$. Moreover, for every short exact sequence of $M_n$-representations admitting BZ-filtrations
    \[
    0\lra \sigma_1\lra \sigma_2\lra \sigma_3\lra 0,
    \]
    we have $\Wh(\sigma_2)=\Wh(\sigma_1)+\Wh(\sigma_3)$ by Proposition~\ref{homology prop}.
\begin{theorem}\label{EP-GL}
    Let $\pi$ be a Casselman-Wallach representation of $\GL_{n+1}$, and $\tau$ be a Casselman-Wallach representation of $\GL_n$. Then $\pi$ satisfies the homological finiteness condition with respect to $\tau$ and 
    \[
    \EP_{\GL_n}(\pi,\tau)=\Wh(\pi)\cdot\Wh(\tau).
    \]
\end{theorem}
\begin{proof}
    By Theorem~\ref{have_B-Z_fil}, it suffices to prove the theorem for $M_{n+1}$-representation $\pi$ with BZ-filtration. We prove by induction on the level of BZ-filtration. Following the notation of Definition~\ref{def-BZ}, when $\pi$ admits a level $\leq 1$ BZ-filtration
    \[
     \pi=\sigma_0\supset \dots \supset \sigma_m\supset 0.
    \]
    By Proposition~\ref{EP-prop} (1), it suffices to prove $\sigma_i/\sigma_{i+1}$ satisfies the homological finiteness condition with respect to $\tau$ and 
    \[
    \EP_{\GL_n}(\sigma_i/\sigma_{i+1},\tau)=\Wh(\sigma_i/\sigma_{i+1})\cdot \Wh(\tau)
    \]
    for $0\leq i\leq m-1$.
    
    \textbf{Case 1.} ($k_i\neq n$). By Lemma~\ref{ML-lem} and Corollary~\ref{vanishing}, we have
    \[
    \rmh_0^{\CS}( N_n,\sigma_i/\sigma_{i+1}\otimes \theta^{-1})\simeq  \mathop{\varprojlim}\limits_j\rmh_0^{\CS}( N_n,  \sigma_i/\sigma_{i,j}\otimes \theta^{-1})=0.
    \]
    On the other hand,
    \[
    \rmh_l^{\CS}(\GL_n, \sigma_i/\sigma_{i+1}\widehat{\otimes}\tau^{\vee})\simeq \mathop{\varprojlim}\limits_j\rmh_l^{\CS}(\GL_n, \sigma_i/\sigma_{i,j}\widehat{\otimes}\tau^{\vee}).
    \]
   Note that 
   \[
   \rmh_l^{\CS}(\GL_n, \sigma_{i,j}/\sigma_{i,j+1}\widehat{\otimes}\tau^{\vee})\simeq \rmh_l^{\CS}(\GL_n,I^{k_i}E(\pi_{i,j})\widehat{\otimes}\tau^{\vee})\simeq \rmh_l^{\CS}(H_{n,k_i},\pi_{i,j}\widehat{\otimes}( \psi_{n,k_i}\otimes \tau^{\vee}\otimes \delta_{H_{n,k_i}}^{-1/2})),
   \]
   where the second isomorphism comes from Mackey isomorphism and Shapiro's lemma. Consequently, by spectral sequence, we have
   \[
   \rmh_p^S(\GL_{n-k_i},\pi_{i,j}\widehat{\otimes} L^qB_{-}^{k_i}(\tau^{\vee}))\Rightarrow \rmh_{p+q}^{\CS}\left(\GL_n, (\sigma_{i,j}/\sigma_{i,j+1})\widehat{\otimes}\tau^{\vee}\right).
   \]
   By Theorem~\ref{haus thm}, $L^qB_{-}^{k_i}(\tau^{\vee})$ is a Casselman-Wallach representation of $\GL_{n-k_i}$. Hence, by the central character condition on BZ-filtration, 
   \[ \rmh_l^{\CS}(\GL_n, \sigma_{i,j}/\sigma_{i,j+1}\widehat{\otimes}\tau^{\vee})=0,\forall l\in \BZ
   \]
   for all sufficiently large $j$. Therefore, $\sigma_i/\sigma_{i+1}$ satisfies the homological finiteness condition with respect to $\tau$ by Proposition~\ref{EP-prop} (3). Moreover, 
   \[
   \EP_{\GL_n}(\sigma_{i,j}/\sigma_{i,j+1},\tau)=\sum_q (-1)^q\EP_{\GL_{n-k_i}}(\pi_{i,j},L^qB_{-}^{k_i}(\tau^{\vee})^{\vee})=0
   \]
   since $\GL_{n-k_i}$ has non-compact center at this time. Thus, by additive property~\ref{EP-prop} (1), we have $\EP_{\GL_n}(\sigma_i/\sigma_{i+1},\tau)=0$.
   
   \textbf{Case 2.} ($k_i=n$). By Lemma~\ref{bottom fin}, $\sigma_i/\sigma_{i+1}$ has finite filtration 
   \[
   \sigma_i=\sigma_{i,0}\supset \dots \supset \sigma_{i,s}=\sigma_{i+1}.
   \]
   As $\Wh(I^nE(\BC))=1$ by Proposition~\ref{homology prop}, it suffices to prove
   \[
   \EP_{\GL_n}(\sigma_{i,j}/\sigma_{i,j+1},\tau)=\Wh(\tau).
   \]
   By similar calculation in case (1), since $L^q\Psi^{n-1}(\tau^{\vee})=0$ for $q>0$, we have 
   \[
    \EP_{\GL_n}(\sigma_{i,j}/\sigma_{i,j+1},\tau)=\EP_{\GL_{0}}(\BC,\Psi^{n-1}(\tau^{\vee})^{\vee})= \dim (\Psi^{n-1}(\tau^{\vee})).
   \]
   Thus the result follows from the fact that
   \[
   \dim (\Psi^{n-1}(\tau^{\vee}))=\dim (\Psi^{n-1}(\tau))=\Wh(\tau).
   \]
    by Lemma~\ref{dual-wh}.
    
Assume that the statement holds for every $M_n$-representation with a BZ-filtration of level $\leq r$. Let $\pi$ have a BZ-filtration of level $\leq r+1$:
\[
\pi = \sigma_0 \supset \dots \supset \sigma_m \supset 0.
\]
By the Koszul resolution, there exists a positive integer such that, for all $q > N$, $L^q B^{k}_{-} (\tau^{\vee}) = 0$ for every integer $0 \leq k \leq n$. Therefore, the generalized central characters in $L^q B^{k}_{-} (\tau^{\vee})$ for all $q$ and $k$ form a finite set $\mathscr{S}_{\tau}$. Define
\[
c := \min \{\mathrm{Re}\,\chi \mid \chi \in \mathscr{S}_{\tau}\}.
\]
By the central character condition in the definition of the BZ-filtration,  there exists some positive integer $j_0$ such that,  for every integer $0 \leq i \leq m-1$ and all $j \geq j_0$,
\[
\min \Omega_{i,j} > -c,
\]
where $\Omega_{i,j}$ is  as Definition \ref{def-BZ}.
 Consequently, by a similar argument as in Case (1), we have
\[
\rmh_l^{\CS}(\mathrm{GL}_n, \sigma_{i,j} / \sigma_{i,j+1} \widehat{\otimes} \tau^{\vee}) = 0,
\]
for all integers $l$, $0 \leq i \leq m-1$, and $j \geq j_0$. Hence, the homological finiteness and Euler-Poincar\'e characteristic formula follow from the induction hypothesis and additivity property~\ref{EP-prop} (1). 

\end{proof}
\begin{remark}
In \cite[Conjecture 7.6]{Wan}, Chen Wan proposes a conjectural Euler-Poincar\'e characteristic formula in terms of geometric multiplicity. 

Since both the Euler-Poincar\'e characteristic and the geometric multiplicity are additive with respect to representations, Theorem~\ref{EP-GL} confirms the conjecture for the pair $(\mathrm{GL}_{n+1}(\mathbf{k}), \mathrm{GL}_n(\mathbf{k}))$ when $\mathbf{k}$ is an Archimedean local field.

\end{remark}

\subsection{Higher Extension vanishing for generic representations}
By comparing infinitesimal characters, we have the following higher extension vanishing result for generic representations, see \cite{CSa21} for the $p$-adic analogy. 
\begin{theorem}\label{ext-van-thm}
    Let $\pi$ and $\tau$ be irreducible generic representations of $\GL_{n+1}$ and $\GL_n$, respectively. Then
    \[
    \Ext^i_{\GL_n}(\pi\widehat{\otimes}\tau^{\vee},\BC)=0,i>0.
    \]
\end{theorem}
We first sketch the main idea before going to the detailed proof.
\begin{definition}
    Let $\chi$ and $\chi^{\flat}$ be two characters of $\GL_1(\mathbf{k})$, we call \begin{enumerate}
        \item $\chi$ is \textbf{positively linked \textit{to}} $\chi^{\flat}$ ( or $\chi^{\flat}$ is \textbf{positively linked \textit{by}} $\chi$ ) if $\chi^{\flat}=\chi |\det|_{\mathbf{k}}^{1/2}\cdot  (\det)^{r_1}\cdot (\overline{\det})^{r_2}$ for some non-negative integers $r_1,r_2$;
        \item $\chi$ is \textbf{negatively linked \textit{to}} $\chi^{\flat}$ ( or $\chi^{\flat}$ is \textbf{negatively linked \textit{by}} $\chi$ ) if $\chi^{\flat}=\chi |\det|_{\mathbf{k}}^{-1/2} \cdot  (\det)^{r_1}\cdot (\overline{\det})^{r_2}$ for some non-positive integers $r_1,r_2$;
    \end{enumerate}

We say that $\chi$ and $\chi^{\flat}$ are \textbf{linked with} each other if  $\chi$ is positively or negatively linked to $\chi^{\flat}$. 

We say  that $\chi$ and $\chi^{\flat}$ are \textbf{related with} each other if the parabolic induction $\chi\times \chi^{\flat}$ of $\GL_2$ is reducible. 
\end{definition}

\begin{remark}
The definition of ``link" comes from the (opposite) BZ-filtrations, see (i) of Theorem \ref{central_char_of_fil}. 

\end{remark}
Our proof proceeds in three steps. 
\begin{itemize}
    \item First, we prove the statement when both $\pi$ and $\tau$ are products of characters, which serves as a starting point for the inductive argument in the next step. We prove by induction on $m$, the number of characters in $\pi$ that are linked with some characters in $\tau$. When $m=0$, the result follows from extension vanishing for the Gelfand-Graev representation and comparing the infinitesimal character of each non-bottom layer term in the BZ-filtration. For $m>0$, it follows from the ``substitution" technique.   
    \item Next, we prove the statement when only one of $\pi$ or $\tau$ contains some relative discrete series of $\GL_2(\BR)$. This is accomplished by the observation that a character of $\tau$ cannot both (i) positively linked to some characters or some discrete series of $\pi$, and (ii) negatively linked to some characters or some discrete series of $\pi$. Otherwise, it contradicts the irreducibility of $\pi$  (see Definition~\ref{DS_link}). The statement then follows from a substitution argument that replaces discrete series with characters.
    \item Finally, we prove the statement in full generality. To do so, we use the following switching lemma to swap the positions of $\pi$ and $\tau$, and then apply a substitution argument to replace the discrete series appearing in $\pi$ and $\tau$ with characters.
\end{itemize}

\begin{lemma}[Switching lemma]\label{swit-lem}
    Let $\pi$ be an irreducible generic representation of $\GL_{n+1}$ and $\tau$ be an irreducible generic representation of $\GL_n$. Then there exists a countable subset $\mathcal{D}\subset \sqrt{-1} \BR$ such that for all $s_1,s_2\in \sqrt{-1} \BR\backslash \mathcal{D}$, $\tau^{\vee}\times \chi_{0,s_1}\times \chi_{0,s_2}$ is irreducible and 
    \[
    \Ext^i_{\GL_n}(\pi\widehat{\otimes} \tau^{\vee},\BC)\simeq \Ext^i_{\GL_{n+1}}((\tau^{\vee}\times \chi_{0,s_1}\times \chi_{0,s_2})\widehat{\otimes} \pi,\BC)
    \]
    for every integer $i$.
\end{lemma}
\begin{proof}
    The proof follows from the irreducibility criterion in Section~\ref{irr_stand} and the proof of \cite[Corollary 4.4]{CC25}. Note that the proof of \cite[Corollary 4.4]{CC25} only involves the infinitesimal character, and this approach is valid for all extension groups.
\end{proof}

We need one more definition of the linking condition for relative discrete series. This condition is defined according to the BZ-filtration and opposite BZ-filtration of discrete series in Section~\ref{LLC-sec}.
\begin{definition}\label{DS_link}
    Let $\chi$ be a character of $\GL_1(\BR)$, and let $D_{k,t}$ be the relative discrete series defined in Section~\ref{LLC-sec}. We say that 
    \begin{enumerate}
        \item  its upper character is positively linked to $\chi$ if $\chi= \chi_{\epsilon_i,t+\frac{k+1}{2}+i}$ for some non-negative integer $i$;
        \item its upper character is negatively linked to $\chi$ if $\chi= \chi_{\epsilon_{i+k},t-\frac{k+1}{2}+i}$ for some non-positive integer $i$;
        \item  its lower character is positively linked to $\chi$ if $\chi= \chi_{\epsilon_{i+1},t+\frac{k+1}{2}+i}$ for some non-negative integer $i$;
        \item its lower character is negatively linked to $\chi$ if $\chi= \chi_{\epsilon_{i+k+1},t-\frac{k+1}{2}+i}$ for some non-positive integer $i$;
    \end{enumerate}
    When (1) or (3) happens, we say that  the discrete series is positively linked to $\chi$, or $\chi$ is negatively linked to the discrete series. When (2) or (4) happens, we say that the discrete series is negatively linked to $\chi$, or $\chi$ is  positively linked to the discrete series.

     We say  that $\chi$ and $D_{k,t}$ are \textbf{related with} each other if the parabolic induction $\chi\times D_{k,t}$ of $\GL_3$ is reducible. 
\end{definition}
\begin{remark}\label{parity-rem}
    By the above definition, we observe a simple but useful fact about parity.
    \begin{itemize}
        \item The upper character and lower character of a discrete series cannot be positively or negatively linked to the same character.
    \end{itemize}
\end{remark}
An essential distinction between generic and non-generic representations lies in the BZ-filtration: for generic representations, the bottom layer is the Gelfand-Graev representation, whose higher extension groups always vanish. Let $G$ be a real reductive group, and the Gelfand-Graev representation is defined as in Section~\ref{EP_sec}.

\begin{theorem}\label{GG-van}
    For every Casselman-Wallach representation $\pi$ of $G$, we have
    \[
    \Ext^i_G(\pi\widehat{\otimes}\SInd_{U^0}^G(\theta),\BC)=0 
    \text{ for } i\geq 1.
    \]
\end{theorem}
\begin{proof}
    The statement follows from \cite[Theorem 8.2]{CHM} directly. For $G=\GL_n$, it also follows from the BZ-filtration and Proposition~\ref{homology prop}.
\end{proof}

In the following proof, for ``discrete series", we always mean the \textbf{relative discrete series of $\GL_2(\BR)$}.
\begin{proof}[Proof of Theorem~\ref{ext-van-thm}]
    \textbf{Step 1.} Assume that $\pi$ and 
$$\tau=\xi_1\times \dots\times \xi_n$$ 
are products of characters. Let $m(\pi,\tau)$ denote the number of characters in $\pi$ that are linked with some characters in $\tau$. We proceed by induction on $m(\pi,\tau)$. When $m(\pi,\tau)=0$, consider the BZ-filtration of $\pi$. Let $I^{k}E(\pi^{\flat})$ be a successive quotient of this filtration, where $\pi^{\flat}$ is an irreducible representation of $\GL_{n-k}$ with infinitesimal character as described in Theorem~\ref{central_char_of_fil}. If $k=0$, then the infinitesimal character of $\pi^{\flat}$ differs from that of $\tau$ by $m(\pi,\tau)=0$. This implies 
\[
\Ext^l_{\GL_n}\left(I^kE(\pi^{\flat})\widehat{\otimes}\tau^{\vee},\BC\right)=\Ext^l_{\GL_n}\left(\pi^{\flat}\widehat{\otimes}\tau^{\vee},\BC\right)=0
\]
for every integer $l$. For $0<k<n$, by Shapiro's lemma,
\[
\rmh_l^{\CS}\left(\GL_n,I^kE(\pi^{\flat})\widehat{\otimes} \tau^{\vee}\right)\simeq  \rmh_l^{\CS}\left(H_{n,k},\pi^{\flat}\widehat{\otimes}(\psi_{n,k}\otimes\tau^{\vee}\otimes\delta_{H_{n.k}}^{-1/2})\right).
\]
To prove the homology vanishing, we apply the BZ-filtration to
\[
\tau^{\vee}=\xi_1^{-1}\times\dots\times \xi_n^{-1}.
\]
It suffices to show that for every successive subquotient $I^sE(\tau^{\sharp})$ in the BZ-filtration of $\tau^{\vee}$, the following holds for every integer $l$:
\[
\rmh_l^{\CS}(H_{n,k},\pi^{\flat}\widehat{\otimes}(\psi_{n,k}\otimes I^sE(\tau^{\sharp})\otimes\delta_{H_{n.k}}^{-1/2}))=0.
\]
Consider the spectral sequence
\[
\rmh^{\CS}_p\left(\GL_{n-k},\pi^{\flat}\widehat{\otimes} L^qB_{-}^k(I^sE(\tau^{\sharp}))\right)\Rightarrow \rmh_{p+q}^{\CS}\left(H_{n,k},\pi^{\flat}\widehat{\otimes}(\psi_{n,k}\otimes I^sE(\tau^{\sharp})\otimes\delta_{H_{n.k}}^{-1/2})\right).
\]
When $k\neq s$, the left-hand side equals zero by Proposition~\ref{homology prop}. When $k=s$, we have
\[
L^qB_{-}^k\left(I^sE(\tau^{\sharp})\right)=\tau^{\sharp}\otimes \wedge^{q}\fkv_{n-k+1}\otimes|\det|_{\mathbf{k}}^{-1/2},
\]
whose generalized infinitesimal characters differ from $({\pi^{\flat}})^{\vee}$ by the assumption that $m(\pi,\tau)=0$. Therefore,
\[
\rmh^{\CS}_p\left(\GL_{n-k},\pi^{\flat}\widehat{\otimes} L^qB_{-}^k(I^sE(\tau^{\sharp}))\right)=0
\]
for every integers $p,q$. When $k=n$, the higher extension vanishes by Theorem~\ref{GG-van}.

    Suppose that the statement holds for $m(\pi,\tau)=m$, we proceed to prove when $m(\pi,\tau)=m+1$. Write $\pi$ as $\pi_1\times \chi$, where $\chi$ is a character linked with some character in $\tau$. Moreover, by considering $(\pi^{\vee},\tau^{\vee})$, and the Switching Lemma \ref{swit-lem}, that is, replacing $(\pi,\tau)$ by $\left((\tau^{\vee}\times \chi_{0,s_1}\times \chi_{0,s_2}),\pi^{\vee}\right)$ with some $s_1,s_2$ not being linked or related with characters of $\tau^{\vee}$ or $\pi^{\vee}$, we can take $\chi$ to be the one such that $\mathrm{Re}(\chi)\geq \mathrm{Re}(\eta)$, for every $\eta$ in the set 
    \[\left\{\chi_i, \xi_j\ |\ \chi_i\ \text{is linked with}\ \xi_j\right\}.\]  
    We observe that $\chi$ cannot be negatively linked to some character in $\tau$, otherwise, it contradicts the maximality of $\mathrm{Re}(\chi)$. 
    
    The $\ov{M_{n+1}}$ has a unique open orbit and a unique closed orbit on $P_{n,1}\backslash\GL_{n+1}$, which leads to the short exact sequence
    \begin{equation}\label{sub-ex-seq}
            0\lra \pi_o\lra \pi|_{\ov{M_{n+1}}}\lra \pi_c\lra 0.
    \end{equation}    
    We prove the extension vanishing for both $\pi_o$ and $\pi_c$. 
    
    The extension vanishing for $\pi_o$ follows from the ``\textbf{substitution}" technique as follows. Let $\widetilde{\pi}:=\pi_1\times\widetilde{\chi}$, where $\widetilde{\chi}$ is a character such that
    
    \begin{itemize}
        \item it is not linked with any character in $\tau$, and
        \item $\widetilde{\pi}$ is irreducible.
    \end{itemize} 
    
    Moreover, we note that $
    \widetilde{\pi}_o\simeq \pi_o$.
    By induction hypothesis, we have
    \[
    \Ext^l_{\GL_n}(\widetilde{\pi}\widehat{\otimes}\tau^{\vee},\BC)=0 \text{ for }l\geq 1.
    \]
   
    On the other hand, $\widetilde{\pi}_c$ admits a filtration whose successive quotients are of the form
    \[  \left(|\det|^{-1/2}_{\mathbf{k}}\widetilde{\chi}\otimes (\det)^{-i_1} (\overline{\det})^{-i_2}\right)\mind \pi_1|_{\ov{M_n}},\ i_1,i_2\in \mathbb{Z}_{\geq 0}. 
    \]
    Hence, the successive quotients in opposite BZ-filtration of $\widetilde{\pi}_c$ always contain some character linked with $\widetilde{\chi}$. Therefore, by a similar argument as above, comparing the infinitesimal characters, we get 
    \[
    \Ext^l_{\GL_n}(\widetilde{\pi}_c\widehat{\otimes}\tau^{\vee},\BC)=0\text{ for }l\geq 1.
    \] 
    Then the long exact sequence associated to a similar sequence as \eqref{sub-ex-seq} for $\widetilde{\pi}$ implies 
     \[
    \Ext^l_{\GL_n}(\widetilde{\pi}_o\widehat{\otimes}\tau^{\vee},\BC)=\Ext^l_{\GL_n}(\pi_o\widehat{\otimes}\tau^{\vee},\BC)=0\text{ for }l\geq 1.
    \]
Since $\chi$ is not negatively linked to any character in $\tau$, by comparing the infinitesimal characters,
  \[
    \Ext^l_{\GL_n}(\pi_c\widehat{\otimes}\tau^{\vee},\BC)=0\text{ for }l\geq 1.
    \]
  Consequently, the result follows from the long exact sequence associated to~\eqref{sub-ex-seq}.
  
  \textbf{Step 2.} Assume that only one of $\pi$ and $\tau$  contains some discrete series. By the Switching Lemma \ref{swit-lem} as \textbf{Step 1}, we can assume that $\tau$ contains no discrete series. 
  
Notice that for a fixed character $\xi$ of $\tau$, the following two things cannot hold simultaneously:
\begin{itemize}
\item $\xi$ is positively linked to some character or discrete series of $\pi$,
\item $\xi$ is negatively linked to some character or discrete series of $\pi$.
\end{itemize}
Otherwise,  it would contradict the irreducibility of $\pi$: for example, assume $\xi$ is positively linked to a character $\chi$ of $\pi$, and negatively linked to a discrete series $D$ of $\pi$, then the parabolic induction $\chi\times D$ must be reducible  and the irreducibility criterion in  Subsection \ref{irr_stand}.  

We can still use the induction argument as \textbf{Step 1} on the size of the sets: 
\begin{align*}
A_1:= & \ \{(\chi^{\pi},\xi^{\tau})\ |\ \text{character $\chi^{\pi}$ in}\ \pi\  \text{is positively linked to}\ \text{character $\xi^{\tau}$ in}\  \tau\}, \\ 
A_2:= &\ \{(D^{\pi},\xi^{\tau})\ |\ \text{discrete series $D^{\pi}$ in}\ \pi\  \text{is positively linked to}\ \text{character $\xi^{\tau}$ in}\  \tau\}.\end{align*}
Using the Switching lemma and replacing $\xi^{\tau}$ with suitable characters, we can reduce to the case that the above sets are empty.

Next, by a similar argument as \textbf{Step 1} again to the size of the set,
\[
A_3:=\{(\chi^{\pi},D^{\pi})\ |\ \text{character $\chi^{\pi}$ in}\ \pi\  \text{is related with}\ \text{upper or lower character of $D^{\pi}$ in}\  \pi\},
\]
replacing $\chi^{\pi}$ with suitable characters, one reduces to the case that the above set is empty.

Let $m(\pi)$ be the number of discrete series in $\pi$. 
The base case  $m(\pi)=0$ is covered by \textbf{Step 1}.
Assuming the statement holds for $m(\pi)=m$, we prove it for $m(\pi)=m+1$. 

Take $D_{k,t}$ in $\pi$ with minimal $k$ among those of discrete series of $\pi$. 
Then every character of $\tau$ cannot be negatively linked to $D_{k,t}$; otherwise, it contradicts  $A_2=\emptyset$. 

Now we use the \textbf{``substitution" for the discrete series} as follows. The $M_{n+1}$ has a unique open orbit and a unique closed orbit on $P_{2,n-1}\backslash\GL_{n+1}$, which leads to the short exact sequence:
     \begin{equation}\label{subds-ex-seq}
            0\lra \pi_o\lra \pi|_{M_{n+1}}\lra \pi_c\lra 0.
    \end{equation}    
  Here 
  \[\pi_o\simeq \p_1\mind D_{k,t}|_{M_2}\] 
  and $\pi_c$ admits a filtration with successive quotients of the form 
  \[\left(|\det|^{1/2}_{\mathbf{k}}D_{k,t}\otimes_{\BR} \mathrm{Sym}^i(\mathrm{St})\right)\times \pi_1|_{M_{n-1}},\ i\in \mathbb{Z}_{\geq 0},\]
  where $\mathrm{St}$ is the standard representation of $\GL_2$.
  
  By the filtration in~\eqref{DS_pos_fil}, we have a short exact sequence for $\pi_o$
  \begin{equation}\label{subds-ex-2}
        0\lra \pi_o^{\flat}\lra \pi_o\lra \pi_o^{\sharp}\lra 0,
  \end{equation}
  where $\pi_o^{\flat}\simeq \p_1\mind\sigma_1$. Let $\widetilde{\pi}:=(\widetilde{\chi}\times \chi_{0,t+\frac{k}{2}})\times \pi_1$, where $\widetilde{\chi}$ is a character such that
  \begin{itemize}
        \item it is not positively or negatively linked to some character in $\tau$, and
        \item $\widetilde{\pi}$ is irreducible.
  \end{itemize}
 Note that by the minimality of $k$ in the $D_{k,t}$, Lemma \ref{irr-lem} implies that $\chi_{0,t+\frac{k}{2}}$ is not related to any discrete series in $\pi_1$. Since $A_3=\emptyset$,  $\chi_{0,t+\frac{k}{2}}$ is not related to any characters in $\pi_1$. Hence, $\chi_{0,t+\frac{k}{2}}\times \pi_1$ is irreducible, so such $\widetilde{\chi}$ exists.
  
  Likewise, the $M_{n+1}$-action on $P_{2,n-1}\backslash\GL_{n+1}$ leads to an exact sequence
  \begin{equation}\label{subds-ex-til}
            0\lra \widetilde{\pi}_o\lra \widetilde{\pi}|_{M_{n+1}}\lra \widetilde{\pi}_c\lra 0,
    \end{equation}    
 where $\widetilde{\pi}_o\simeq \p_1\mind (\widetilde{\chi}\times \chi_{0,\frac{k}{2}})|_{M_2}$. Furthermore, the $M_2$-action on $B_2\backslash\GL_2$ gives rise to an exact sequence
 \begin{equation}\label{subds-til-2}
      0\lra \widetilde{\pi}_o^{\flat}\lra \widetilde{\pi}_o\lra \widetilde{\pi}_o^{\sharp}\lra 0,
   \end{equation}
  where $\widetilde{\pi}_o^{\flat}\simeq \p_1\mind(\widetilde{\chi}\times \chi_{0,\frac{k}{2}})_o$. Thus, we observe that $\widetilde{\pi}_o^{\flat}\simeq \pi_o^{\flat}$.
  
  By the induction hypothesis, we have
    \[
    \Ext^l_{\GL_n}(\widetilde{\pi}\widehat{\otimes}\tau^{\vee},\BC)=0 \text{ for }l\geq 1.
    \]
   
On the other hand, since the successive quotients in the BZ-filtration of $\widetilde{\pi}_o^{\sharp}$ and $\widetilde{\pi}_c$ contain  some character positively linked by $\widetilde{\chi}$, 
    \[
   \Ext^l_{\GL_n}(\widetilde{\pi}_o^{\sharp}\widehat{\otimes}\tau^{\vee},\BC)=  \Ext^l_{\GL_n}(\widetilde{\pi}_c\widehat{\otimes}\tau^{\vee},\BC)=0\text{ for }l\geq 1
    \]
    by infinitesimal character. 
    
Hence, by the long exact sequence associated to the short exact sequence~\eqref{subds-ex-til} and~\eqref{subds-til-2}, 
    \[
    \Ext^l_{\GL_n}(\widetilde{\pi}^{\flat}_o\widehat{\otimes}\tau^{\vee},\BC)=\Ext^l_{\GL_n}(\pi^{\flat}_o\widehat{\otimes}\tau^{\vee},\BC)=0\text{ for }l\geq 1.
    \]
    
    Since every infinitesimal character of successive quotients in the BZ-filtrations of $\pi_c$ and $\pi_o^{\sharp}$ always contains some characters positively linked by the lower character of $D_{k,t}$ and $A_2=\emptyset$, one has
      \[
    \Ext^l_{\GL_n}(\pi_o^{\sharp}\widehat{\otimes}\tau^{\vee},\BC)= \Ext^l_{\GL_n}(\pi_c\widehat{\otimes}\tau^{\vee},\BC)=0\text{ for }l\geq 1.
    \]
  The result now follows from the long exact sequence associated to~\eqref{subds-ex-seq} and~\eqref{subds-ex-2}.

\textbf{Step 3.} 
We apply the substitution argument as \textbf{Step 2} to the characters in $\pi$ and $\tau$, reducing to the case when there exist no characters in $\pi$ (resp. $\tau$) linked with the characters in $\tau$  (resp. $\pi$), and there exist no characters in $\pi$ (resp. $\tau$) related with the discrete series in $\pi$ (resp. $\tau$). 

Take the discrete series $D_{k,t}$ in $\pi$ or $\tau$ such that $k$ is minimal among those of the discrete series in $\pi$ and  $\tau$. By the switching lemma, one can assume that $D_{k,t}$ is in $\pi$. The upper character of $D_{k,t}$ cannot both (i) positively link to  a character or a discrete series in $\tau$ and (ii) negatively link to  a character or a discrete series in $\tau$.  Otherwise, it contradicts the irreducibility of $\tau$. Meanwhile, the upper character of $D_{k,t}$ does not relate to any characters in $\pi$ by the substitution in the last paragraph. The upper character of $D_{k,t}$ does not relate to any discrete series in $\pi$ by the minimality of $k$. Hence, one can use the substitution argument as \textbf{Step 2} to replace the $D_{k,t}$ by characters. 

Applying the substitution arguments as \textbf{Step 2} to  repeatedly, we can finish the proof.

\end{proof}

We emphasize that although the higher extension groups vanish for generic representations, it is not true in general.
\begin{example}\label{ext_nonzero}
 Let $\mathbf{1}_n$ be the trivial representation of $\rm{GL}_n(\BC)$. Let $\pi=\mathbf{1}_2\times\mathbf{1}_2$ be an irreducible unitary representation of $\rm{GL}_4(\BC)$, and $\tau:=\mathbf{1}_1\times \mathbf{1}_1\times \chi$ be an irreducible unitary representation of $\rm{GL}_3(\BC)$, where $\chi$ is a unitary character of $\rm{GL}_1(\BC)$.
 In Example~\ref{exa-BZ}, we have seen that $\pi|_{\rm{GL}_3(\BC)}$ admits a filtration
 \[
 \pi|_{\rm{GL}_3(\BC)}=\sigma_0\supset \sigma_1\supset\sigma_2\supset 0,
 \]
  where $\sigma_0/\sigma_1$ and $\sigma_1/\sigma_2$ have infinite filtrations such that each irreducible subquotient has a positive central character. Therefore, for every integer $i$,
 \[
 \Ext^i_{\rm{GL}_3(\BC)}\left((\sigma_0/\sigma_2)\widehat{\otimes} \tau^{\vee},\BC\right)=0 \text{ and } \Hom_{\rm{GL}_3(\BC)}(\pi,\tau)\simeq \Hom_{\rm{GL}_3(\BC)}(\sigma_2,\tau).
 \]
 On the other hand,
\begin{align*}
    \rmh_0^{\CS}(\GL_3(\BC),\sigma_2\widehat{\otimes}\tau^{\vee}) & \simeq \rmh_0^{\CS}\left(\GL_3(\BC),\SInd_{M_3}^{\GL_3(\BC)}(E(\mathbf{1}_1\times \mathbf{1}_1)\widehat{\otimes}\tau^{\vee}|_{M_3})\right)\\
    &\simeq \rmh_0^{\CS}\left(M_3, E((\mathbf{1}_1\times \mathbf{1}_1)\cdot |\det|^{-1})\widehat{\otimes}\tau^{\vee}|_{M_3}\right).
\end{align*}
Since there exists a surjective map $\Phi(\tau^{\vee})\twoheadrightarrow (\mathbf{1}_1\times \mathbf{1}_1)\cdot |\det|^{1}$, we have 
\[\dim \rmh_0^{\CS}(\GL_3(\BC),\sigma_2\widehat{\otimes}\tau^{\vee})\geq 1.\]
 By multiplicity one theorem(see \cite[Theorem B]{SZ12}), the dimension is exactly one. From the perspective of Euler-Poincar\'e characteristic, $\pi$ is non-generic and $\tau$ is generic, thus $\EP_{\GL_3(\BC)}(\pi,\tau)=0$ by Theorem~\ref{EP-GL}. This implies that 
 \[
 \Ext^i_{\GL_3(\BC)}(\pi\widehat{\otimes}\tau^{\vee},\BC)\neq 0 \text{ for some } i\geq 1.
 \]
 In fact, this example fits into the framework of non-tempered GGP-conjecture, which now is a theorem in both real and $p$-adic cases, see \cite{GGP20,Chan22} for $p$-adic case and \cite{Boi25,CC25} for real case.
 
\end{example}

\end{document}